\documentclass[11pt]{article}
\usepackage[T1]{fontenc}
\usepackage{titlesec}
\titleformat{\section}
  {\normalfont\fontsize{12}{15}\bfseries}{\thesection}{1em}{}
\usepackage{csquotes}
\usepackage[english]{babel}
\usepackage[top = 1in,tmargin = 1in]{geometry}
\usepackage{graphicx} 
\usepackage{subcaption}
\usepackage{amsfonts,amsmath,amssymb,amsthm,mathtools}
\usepackage[colorlinks=true,citecolor=black,linkcolor=black,urlcolor=black]{hyperref}
\usepackage{cleveref}
\usepackage[backend=biber,style=ieee]{biblatex}

\usepackage{float}
\usepackage{xcolor}
\usepackage[normalem]{ulem}
\usepackage[symbol]{footmisc}
\usepackage{titling}
\usepackage[affil-it]{authblk}

\usepackage[title]{appendix}
\bibliography{main}
\usepackage{titling}
\numberwithin{equation}{section}
\providecommand{\keywords}[1]
{
  \small	
  \textit{Keywords:} #1
}
\newcommand\blfootnote[1]{%
  \begingroup
  \renewcommand\thefootnote{}\footnote{\scriptsize #1}%
  \addtocounter{footnote}{-1}%
  \endgroup
}
\title{\LARGE\hspace{-5mm} Asymptotic-preserving and energy stable dynamical low-rank approximation for thermal radiative transfer equations}
\author[1,$\ast$]{Chinmay Patwardhan}
\author[1]{Martin Frank}
\author[2]{Jonas Kusch}
\affil[1]{\textit{Karlsruhe Institute of Technology, Institute of Applied and Numerical Mathematics, Karlsruhe, Germany}}
\affil[2]{\textit{Norwegian University of Life Sciences, Scientific Computing, \AA s, Norway}}
\date{\vspace{-2.5em}}

\newcommand{\R}{\mathbb{R}} 
\newcommand{\Rp}{\mathbb{R}^{+}} 
\newcommand{\N}{\mathbb{N}} 

\newcommand{\norm}[1]{\left\lVert #1 \right\rVert} 
\newcommand{\abs}[1]{\left\lvert #1 \right\rvert} 

\newcommand{\vart}{t} 
\newcommand{\varx}{x} 
\newcommand{\omg}{\mu} 
\newcommand{\varX}{\textbf{x}} 
\newcommand{\Omg}{\boldsymbol{\Omega}}  

\newcommand{\pardift}{\partial_{\vart}} 
\newcommand{\pardifx}{\partial_{\varx}} 

\newcommand{\domg}[1]{\hspace{#1mm}\mathrm{d}\omg} 
\newcommand{\dx}[1]{\hspace{#1mm}\mathrm{d}\varx} 
\newcommand{\intgomg}[1]{\langle #1 \rangle_{\omg}} 
\newcommand{\intgomgsb}[1]{\left[ #1 \right]_{\omg}} 
\newcommand{\intgx}[1]{\langle #1 \rangle_{\varx}} 

\newcommand{\epsi}{\varepsilon} 
\newcommand{\sol}{c} 
\newcommand{\parden}{f} 
\newcommand{\absorpcoeff}{\sigma^{a}} 
\newcommand{\absorpcoeffb}{\boldsymbol{\absorpcoeff}}
\newcommand{\absorpcoefflb}{\sigma_{0}} 
\newcommand{\scalflux}{\phi} 
\newcommand{\Temp}{T} 
\newcommand{\bTemp}{\textbf{\Temp}}
\newcommand{\Planck}{B(\Temp)} 
\newcommand{\specheat}{c_{\nu}} 
\newcommand{\StefBoltzconst}{a} 

\newcommand{\SOPlanck}{B} 

\newcommand{\mic}{g} 
\newcommand{\meso}{h} 
\newcommand{\Ident}{\mathcal{I}} 
\newcommand{\DPlanck}{\SOPlanck'(\Temp)} 
\newcommand{\DPlanckC}{\Psi}  
\newcommand{\facalph}{\kappa} 

\newcommand{\pk}[1]{P_{#1}} 
\newcommand{\pknorm}[1]{\gamma_{#1}} 
\newcommand{\ak}[1]{a_{#1}} 
\newcommand{\fluxMat}{\textbf{A}} 
\newcommand{\unitvec}[1]{\textbf{\textit{e}}_{#1}} 
\newcommand{\unitvecN}{\unitvec{1}} 
\newcommand{\mick}[1]{\mic_{#1}} 
\newcommand{\micvec}{\textbf{\textit{\mic}}} 
\newcommand{\bvec}{\textbf{\textit{b}}} 

\newcommand{\FfluxMat}{\fluxMat_{f}} 
\newcommand{\absFfluxMat}{\abs{\fluxMat_{f}}} 
\newcommand{\FTmat}{\Tmat_{f}} 
\newcommand{\Fba}{\ba_{f}} 
\newcommand{\bv}{\textbf{v}}
\newcommand{\ba}{\textbf{a}}

\newcommand{\quadpt}{\omg} 
\newcommand{\weight}{w} 
\newcommand{\Tmat}{\textbf{T}} 
\newcommand{\quadptMat}{\textbf{M}} 
\newcommand{\absquadMat}{\abs{\quadptMat}} 
\newcommand{\absfluxMat}{\abs{\fluxMat}} 
\newcommand{\fluxMatp}{\fluxMat^{+}} 
\newcommand{\fluxMatm}{\fluxMat^{-}} 
\newcommand{\fluxMatpm}{\fluxMat^{\pm}} 

\newcommand{\Nx}{N_{x}} 
\newcommand{\deltax}{\Delta\varx} 
\newcommand{\deltat}{\Delta\vart} 
\newcommand{\Diffp}{\mathcal{D}^{+}} 
\newcommand{\Diffm}{\mathcal{D}^{-}} 
\newcommand{\Diffpm}{\mathcal{D}^{\pm}} 
\newcommand{\Diffmp}{\mathcal{D}^{\mp}} 
\newcommand{\AdvecOp}{\mathcal{L}} 
\newcommand{\DiffO}{\mathcal{D}^{0}} 
\newcommand{\deltaO}{\delta^{0}} 
\newcommand{\bdeltaO}{\boldsymbol{\deltaO}}

\newcommand{\RoeTmat}{\widetilde{\Tmat}} 
\newcommand{\RoeEigenval}{\widetilde{\quadptMat}}
\newcommand{\absRoeEigenval}{\abs{\widetilde{\quadptMat}}}

\newcommand{\bphi}{\boldsymbol{\phi}}
\newcommand{\bzeta}{\boldsymbol{\zeta}}
\newcommand{\bigO}[1]{\mathcal{O}(#1)}

\newcommand{\energy}{e}

\newcommand{\Mr}{\mathcal{M}_{\rank}} 
\newcommand{\TyMr}[1]{\mathcal{T}_{#1}\Mr} 
\newcommand{\micmat}{\normalfont\textbf{g}}
\newcommand{\DLRAProj}{\textbf{P}}

\newcommand{\rank}{r}
\newcommand{\Xc}{X} 
\newcommand{\Vc}{\textbf{V}} 
\newcommand{\Sc}{S} 
\newcommand{\Xcvec}{\textbf{\textit{\Xc}}} 
\newcommand{\Vmat}{\textbf{\Vc}} 
\newcommand{\Smat}{\textbf{\Sc}} 
\newcommand{\Kcvec}{\textbf{\textit{K}}}
\newcommand{\Lmat}{\textbf{L}}
\newcommand{\Mbug}{\textbf{M}}
\newcommand{\Nbug}{\textbf{N}}

\newcommand{\hXmat}{\widehat{\Xmat}}
\newcommand{\hXvec}{\widehat{\Xvec}}
\newcommand{\hVmat}{\widehat{\Vmat}}
\newcommand{\hSmat}{\widehat{\Smat}}
\newcommand{\hMbug}{\widehat{\Mbug}}
\newcommand{\hNbug}{\widehat{\Nbug}}
\newcommand{\Pmat}{\textbf{P}} 
\newcommand{\Qmat}{\textbf{Q}} 
\newcommand{\BSigma}{\boldsymbol{\Sigma}} 
\newcommand{\Nrank}{\rank_{1}} 
\newcommand{\Pmattrunc}{\Pmat_{\Nrank}}
\newcommand{\Qmattrunc}{\Qmat_{\Nrank}}
\newcommand{\BSigmatrunc}{\BSigma_{\Nrank\times\Nrank}}

\newcommand{\Xvec}{\textbf{\textit{\Xc}}} 
\newcommand{\Kvec}{\textbf{\textit{K}}}
\newcommand{\AdvecOpK}{\AdvecOp_{K}}
\newcommand{\AdvecOpL}{\AdvecOp_{L}}
\newcommand{\AdvecOpS}{\AdvecOp_{S}}

\newcommand{\Smatinit}{\widetilde{\Smat}}
\newcommand{\Xmat}{\textbf{\Xc}}
\newcommand{\Kmat}{\textbf{K}}

\newcommand{\ProjX}{P^{X}}
\newcommand{\ProjV}{P^{V}}

\newcommand{\PN}{$\mathrm{P}_{N}$ }

\newtheorem{lemma}{Lemma}

\newcommand{\LL}{L^{2}}

\newtheorem{theorem}{Theorem}
\newtheorem{property}{Property}
\theoremstyle{remark}
\newtheorem{remark}{Remark}

\begin{document}

\maketitle
\begin{abstract}
    The thermal radiative transfer equations model temperature evolution through a background medium as a result of radiation. When a large number of particles are absorbed in a short time scale, the dynamics tend to a non-linear diffusion-type equation called the Rosseland approximation. The main challenges for constructing numerical schemes that exhibit the correct limiting behavior are posed by the solution's high-dimensional phase space and multi-scale effects. In this work, we propose an asymptotic-preserving and rank-adaptive dynamical low-rank approximation scheme based on the macro-micro decomposition of the particle density and a modified augmented basis-update \& Galerkin integrator. We show that this scheme, for linear particle emission by the material, dissipates energy over time under a step size restriction that captures the hyperbolic and parabolic CFL conditions. We demonstrate the efficacy of the proposed method in a series of numerical experiments.
\end{abstract}

\keywords{thermal radiative transfer equations, energy stability, asymptotic-preserving scheme, dynamical low-rank approximation, macro-micro decomposition}
    
\blfootnote{$^{\ast}$Corresponding author.}\blfootnote{Chinmay Patwardhan: chinmay.patwardhan@kit.edu}\blfootnote{Martin Frank: martin.frank@kit.edu}\blfootnote{Jonas Kusch: jonas.kusch@nmbu.no}
  \section{Introduction}\label{section:Introduction}
The radiation of particles from a hot source into a cold medium and the corresponding formation of a thermal heat front, known as a Marshak wave, is well modeled by the thermal radiative transfer equations. They consist of a coupled system of partial differential equations governing the transport of particles (represented by the particle density $\parden$) and the temperature evolution of the medium ($\Temp$) \cite{RadHeatTrans} given by
\begin{subequations}
\begin{equation*}
\frac{1}{c}\partial_{t}f + \hspace{1mm}\boldsymbol{\Omega}\cdot\nabla_{\textbf{x}}f = \sigma^{a}(B(T) - f),
\end{equation*}
\begin{equation*}
    c_{\nu}\partial_{t}T = \int_{\mathbb{S}^{2}}\sigma^{a}(f- B(T))\hspace{1mm}\mathrm{d}\boldsymbol{\Omega}.
\end{equation*}
\end{subequations}
The particle density $\parden$ depends on time $t$, position $\varX$ and direction of flight $\Omg\in \mathbb{S}^{2}$ and the temperature $\Temp$ on time and position. $\sol$ and $\specheat$ represent the speed of light and the specific heat of the material, respectively. These two equations are coupled through the absorption and emission of particles by the background material. Numerically simulating the thermal radiative transfer equations poses several challenges. First, to evolve and store the particle density, high memory, and computational resources are required. This is due to its high-dimensional phase space, consisting of temporal, spatial, and angular variables, which can be up to six-dimensional for a three-dimensional spatial domain. Second, when a large number of these particles are absorbed in small time scales, the dynamics of the system asymptotically converge to a diffusion-type non-linear partial differential equation known as the Rosseland approximation \cite{Rosseland_approx_ref} which reads
\begin{equation*}
    (c_{\nu} + 4aT^{3})\partial_{t}T = \nabla_{\textbf{x}}\cdot\left( \frac{4ac}{3\sigma^{a}}T^{3}\nabla_{\textbf{x}}T \right),
\end{equation*}
where $\StefBoltzconst$ is the radiation constant. Since the dynamics tend to a diffusion problem, the numerical scheme should also capture this behavior without having to resolve prohibitively small time scales. Numerical methods that do so while efficiently treating the stiffness arising from large absorption terms are called asymptotic-preserving (AP) schemes \cite{Jin_AP_Def}. Work on asymptotic-preserving schemes for kinetic equations can, for example, be found in \cite{Jin_AP_Def,HU2017103,MR1706723,MR2460781}.


To address the computational challenges posed by a high-dimensional phase space, we use dynamical low-rank approximation (DLRA) \cite{doi:10.1137/050639703}, which is a model-order reduction strategy that has recently gained popularity in solving kinetic equations. The fundamental idea behind dynamical low-rank approximation is to evolve the solution on the manifold of rank $ \rank $ functions $\Mr$ by projecting the dynamics to the tangent space of $\Mr$. The evolution equations thus obtained can be interpreted as a Galerkin system, with $2\rank$ basis functions for the phase space variables, which evolves the coefficients and basis functions according to the dynamics of the problem. Using standard time integrators to evolve the coefficient and basis functions leads to unstable numerical schemes and thus robust integrators, like the projector-splitting integrator (PSI) \cite{doi:10.1007/s10543-013-0454-0}, and the basis-update \& Galerkin (BUG) integrators \cite{doi:10.1007/s10543-021-00873-0,robustBUG,ceruti2023parallel,ceruti2024robust}, have been developed. 

DLRA has, for example, been shown to reduce computational costs in dose computation in radiation therapy planning \cite{refId0}, in high-scattering problems \cite{EINKEMMER2021110353} and neutron criticality problems \cite{Kusch_2022}. In recent works, DLRA was used for the thermal radiative transfer equations \cite{10.5445/IR/1000154134,baumann2023energy} where it was shown to significantly reduce computational time. The work in \cite{baumann2023energy} proposes a low-rank scheme based on the augmented BUG integrator of \cite{robustBUG} for thermal radiative transfer. Though this scheme is energy-stable and preserves mass locally, it does not include multiscale effects frequently arising in thermal radiative transfer. 

Since the thermal radiative transfer equations tend to a diffusion equation, for small time scales, the dynamics are restricted to the manifold of low-rank functions \cite{Rosseland_approx_ref} and thus can be accurately represented by a low-rank approximation. Thus, a DLRA scheme combined with techniques to preserve asymptotic behavior can be highly beneficial in tackling both numerical challenges simultaneously. This was investigated for the related problem of radiation transport in \cite{MR4253315}, where the PSI, along with a macro-micro decomposition to construct an asymptotic-preserving scheme, is used. A key challenge for such schemes is to prove stability and provide a CFL condition that takes into account effects at long time scales (kinetic regime) and small time scales (diffusive regime). In contrast to \cite{MR4253315}, the scheme constructed in \cite{einkemmer2022asymptoticpreserving} uses the fixed-rank BUG integrator and is energy stable under a CFL restriction, which captures both the kinetic and diffusive regimes. In the kinetic regime, a large number of particles stream in all directions, increasing the rank required to resolve the solution. In contrast, in the diffusive regime, a large number of streaming particles are absorbed and diffused, thus lowering the required rank of the solution \cite{Rosseland_approx_ref}. This is well demonstrated in the experiments from \cite{einkemmer2022asymptoticpreserving,MR4253315}. Thus, using a fixed-rank integrator without prior knowledge of the required rank in the regime results in either higher computational costs (due to over-approximation) or a poorly resolved solution (due to under-approximation). One of the ways to address this is by using a DLRA scheme that appropriately chooses and evolves the rank of the solution according to the regime.

This work proposes an asymptotic-preserving and rank-adaptive DLRA scheme for the thermal radiative transfer equations in slab geometry and analyzes its properties. The novelty of this work can be summarized in the following:
\begin{itemize}
    \item \textit{An asymptotic-preserving, mass conservative and rank-adaptive DLRA integrator:} We propose a new asymptotic-preserving BUG integrator for the thermal radiative transfer equations, based on the macro-micro decomposition \cite{MR2460781, MR1842224} of the particle density, the basis augmentation step from \cite{robustBUG} and conservative truncation \cite{EINKEMMER2023112060,einkemmer2023conservation}, to capture the underlying dynamics of the system. The proposed algorithm is locally mass conservative.
    \item \textit{A stability analysis for the proposed asymptotic-preserving DLRA scheme}: We show that the proposed integrator is stable in the energy norm under a CFL restriction that captures both the kinetic and the diffusive regime under linear emission of particles by the background material.
\end{itemize}
This paper is structured as follows: Following the introduction in \Cref{section:Introduction}, in \Cref{section:Background} we review background concepts that are used in this paper and build the modal macro-micro equations for the thermal radiative transfer equations. In \Cref{section:EnergyStableMMM}, we present a spatio-temporal discretization for the modal macro-micro equations, which is asymptotic-preserving and stable under a CFL restriction that captures the kinetic and diffusive regime. In \Cref{section:DLRAforMMM} we present dynamical low-rank integrators for the thermal radiative transfer equations and the stability results. Specifically, in \Cref{section:BUGforMMM}, we present the evolution equations for the modal macro-micro equations using the fixed-rank BUG integrator \cite{doi:10.1007/s10543-021-00873-0} and prove its stability property. In \Cref{section:raBUGforMMM}, we present the asymptotic-preserving BUG integrator and prove the stability of the scheme. Finally, numerical experiments are presented in \Cref{section:Numericalresults}.


\section{Background}\label{section:Background}
\noindent In this section, we review the basic ideas and concepts that are used in this work. The first subsection describes the thermal radiative transfer equations in slab geometry for the gray approximation, its asymptotic limit, and the macro-micro decomposition \cite{MR2460781} of the particle density and its angular discretization. In the second subsection, we look at dynamical low-rank approximation \cite{doi:10.1137/050639703}, the fixed-rank BUG integrator \cite{doi:10.1007/s10543-021-00873-0} and the augmented BUG integrator \cite{robustBUG}. 

\subsection{Thermal radiative transfer equations}
In this paper, we consider the dimensionless form of the gray (i.e., frequency-averaged) thermal radiative transfer equations in slab geometry,
\begin{subequations}\label{eq:RTE}
\begin{equation}\label{eq:RTEKin}
\frac{\epsi^{2}}{\sol}\pardift\parden + \epsi\omg\pardifx\parden = \absorpcoeff(\Planck - \parden),
\end{equation}
\begin{equation}\label{eq:RTETemp}
    \epsi^{2}\specheat\pardift\Temp = \int_{-1}^{1}\absorpcoeff(\parden- \Planck)\domg{1}.
\end{equation}
\end{subequations}
In the above equations, $\parden(\vart,\varx,\omg)$ represents the particle density (or angular flux) at time $\vart\in\Rp$, position $\varx\in D\subset\R$ and direction of flight $\omg\in[-1,1]$. The temperature of the material is given by $\Temp(\vart,\varx)$ and depends on time and position. These are supplemented with initial and boundary conditions, which are later specified according to the problem. $\Planck$ describes the emission of particles by the background material due to blackbody radiation at its current temperature. It is given by the Stefan-Boltzmann law, $$\Planck = \StefBoltzconst\sol\Temp^{4},$$ where $\StefBoltzconst$ is the radiation constant and $\sol$ is the speed of light. The rate of absorption and emission of particles by the background material is specified by the absorption cross-section $\absorpcoeff(\varx)$, where we assume that $\absorpcoeff(\varx)\geq\absorpcoefflb>0$. We denote the integral over $\omg$ as $ \intgomg{\cdot} = \int_{-1}^{1}\cdot\domg{1} $ and thus the scalar flux of the particle density is defined as, $\scalflux(\vart,\varx) =  \frac{1}{2}\intgomg{f} $. 

 In \eqref{eq:RTE} as $\epsi$ tends to zero, absorption effects dominate the dynamics. A Hilbert expansion of the particle density $\parden$ yields that, as $\epsi \to 0$, the particles are distributed as $\Planck$, i.e., $\parden = \Planck$, while the evolution of temperature is given by a diffusion-type non-linear equation known as the Rosseland approximation \cite{Rosseland_approx_ref}:
\begin{equation}\label{eq:RosselandApprox}
    \specheat\pardift\Temp = \frac{2}{3}\pardifx\left( \frac{1}{\absorpcoeff}\DPlanck\pardifx\Temp \right) - \frac{2}{\sol}\DPlanck\pardift\Temp,
\end{equation}
where $\DPlanck = \frac{d}{d\Temp}\Planck$.


\subsubsection{Macro-micro decomposition}
\noindent The asymptotic analysis of the thermal radiative equations \eqref{eq:RTE} shows that multiple time scales are involved in the evolution of temperature and particles. In particular, effects occur at times scales of order $\bigO{1}$, $\bigO{\epsi}$, and $\bigO{\epsi^{2}}$ and must be correctly resolved to capture the underlying dynamics of the system. One way to do this is by decomposing the particle density into variables that describe the macroscopic and microscopic effects. This type of decomposition of the particle density is called a macro-micro decomposition and was first proposed in \cite{MR2460781}. Since the thermal radiative transfer equations involve three time scales, the particle density is decomposed into macroscopic ($B$), microscopic ($\mic$), and mesoscopic ($\meso$) variables. For the thermal radiative transfer equations, this macro-micro ansatz was first proposed in \cite{MR1842224}. To be precise, we make the following ansatz for the particle density
\begin{equation}\label{eq:macromicrodecomp}
    \parden(\vart,\varx,\omg) = \SOPlanck(\Temp(\vart,\varx)) + \epsi\mic(\vart,\varx,\omg) + \epsi^{2}\meso(\vart,\varx),
\end{equation}
where $\intgomg{\mic} = 0$.
Note that since $\intgomg{\mic} = 0$, the total mass of the system is conserved and was used in \cite{koellermeier2023macromicro} to construct a numerical scheme that conserves mass in shallow water equations.
\begin{remark}
The rationale behind calling $\meso$ the mesoscopic variable instead of the microscopic variable, despite scaling as $\epsi^{2}$ is that it arises as the leading scaled quantity in the decomposition of the macroscopic quantity of radiation transport equation, the scalar flux, and does not depend on the angular variable. 
\end{remark}

To obtain evolution equations for $\meso, \mic$ and $\Temp$ we substitute the macro-micro ansatz \eqref{eq:macromicrodecomp} in the radiative transfer equations \eqref{eq:RTE} yielding the following evolution equations
\begin{subequations}\label{eq:MacMicSys}
    \begin{equation}\label{eq:MacMicSysmes}
         \frac{\epsi^{2}}{\sol}\pardift\meso + \frac{\facalph}{\sol}\absorpcoeff\DPlanck\meso + \frac{1}{2}\pardifx\intgomg{\omg\mic} = -\absorpcoeff\meso,
    \end{equation}
    \begin{equation}\label{eq:MacMicSysmic}
        \frac{\epsi^{2}}{\sol}\pardift\mic + \epsi\left( \Ident - \frac{1}{2}\intgomg{\cdot} \right)(\omg\pardifx\mic) + \DPlanck\omg\pardifx\Temp + \epsi^{2}\omg\pardifx\meso = -\absorpcoeff\mic,
    \end{equation}
    \begin{equation}\label{eq:MacMicSysmac}
        \pardift\Temp = \facalph\absorpcoeff\meso,
    \end{equation}
\end{subequations}
where we set $\facalph = \frac{2}{\specheat}$ for ease of presentation. Note that by comparing the $\bigO{\epsi^{0}}$ terms in all three equations of \eqref{eq:MacMicSys} we obtain the Rosseland approximation \eqref{eq:RosselandApprox} \cite{Rosseland_approx_ref,MR1842224}.

\paragraph{\hspace{0.1mm}Initial and boundary conditions}It remains to describe the initial and boundary conditions for the macro-micro equations \eqref{eq:MacMicSys}. Note that we can write the microscopic variable $\mic$ and mesoscopic variable $\meso$ as
\begin{subequations}\label{eq:ICBC}
    \begin{equation}\label{eq:ICBCmic}
        \mic(\vart,\varx,\omg) = \frac{1}{\epsi}\left( \parden(\vart,\varx,\omg) - \frac{1}{2}\intgomg{\parden(\vart,\varx,\omg)} \right),
    \end{equation}
    \begin{equation}\label{eq:ICBCmeso}
        \meso(\vart,\varx) = \frac{1}{\epsi^{2}}\left( \frac{1}{2}\intgomg{\parden(\vart,\varx,\omg)} - \Planck(\vart,\varx) \right).
    \end{equation}
\end{subequations}
Thus, for given initial and boundary conditions of the radiative transfer equations \eqref{eq:RTE}, we use the above relations to derive the initial and boundary conditions for the macro-micro equations \eqref{eq:MacMicSys}.

\subsubsection{Angular discretization of microscopic variable}\label{subsection:Modal MM}
The microscopic variable $\mic$ depends on the direction of flight, $\omg$, and must be discretized in the angular domain. In this work, we use the method of moments or the \PN method \cite{case1967linear} to discretize in $\omg$. To obtain the moment equations, let $\{ \widetilde{P}_{k} \}_{k\in\N\cup\{0\}}$ be orthogonal Legendre polynomials with standard $\LL([-1,1])$ norms $\pknorm{k}$, given by $\pknorm{k}^{2} = \frac{2}{2k+1}$. Let $\pk{k} = \widetilde{P}_{k}/\pknorm{k}$  denote the $k^{\text{th}}$ orthonormal Legendre polynomial satisfying the recurrence relation
\begin{equation*}
    \omg\pk{k} = \ak{k-1}\pk{k-1} + \ak{k}\pk{k+1}, \qquad \ak{k} = \frac{k+1}{\sqrt{(2k+1)(2k+3)}}.
\end{equation*}

\noindent The \PN ansatz for $\mic$ then reads
\begin{equation*}
    \mic(\vart,\varx,\omg) \approx \mic_{\mathrm{P}_{N}}(\vart,\varx,\omg) = \sum_{k=0}^{N}\mick{k}(\vart,\varx)\pk{k}(\omg),
\end{equation*}
where $\mick{k}$ is called the k$^{th}$ moment of the system. Since $\pk{k}$ is orthonormal the k$^{th}$ moment is given by $\mick{k} = \intgomg{\mic\pk{k}}$. To obtain evolution equations for the moments $\mick{k}$, $k = 0,1,\ldots,N$, we multiply \eqref{eq:MacMicSysmic} by $\pk{k}$ and integrate over $\omg\in[-1,1]$ . The evolution equation for the k$^{\text{th}}$ moment is then given by
\begin{equation*}
    \frac{\epsi^{2}}{\sol}\pardift\mick{k} + \epsi\pardifx\left( \ak{k-1}\mick{k-1} + \ak{k}\mick{k+1} \right) - \epsi\frac{\pknorm{0}\pknorm{1}}{2}\pardifx\mick{1}\delta_{k0} + \pknorm{1}\left(\DPlanck\pardifx\Temp + \epsi^{2}\pardifx\meso \right)\delta_{k1} = -\absorpcoeff\mick{k}.
\end{equation*}
Note that since $\intgomg{\mic} = 0$, we get $\mick{0} = 0$ and we obtain the following system of modal macro-micro equations
\begin{subequations}\label{eq:ModalMacMicSys}
    \begin{equation}\label{eq:ModalMacMicSysmeso}
        \frac{\epsi^{2}}{\sol}\pardift\meso + \frac{\facalph}{\sol}\absorpcoeff\DPlanck\meso + \frac{\pknorm{1}}{2}\pardifx\mick{1} = -\absorpcoeff\meso,
    \end{equation}
    \begin{equation}\label{eq:ModalMacMicSysmic}
        \frac{\epsi^{2}}{\sol}\pardift\micvec + \epsi\fluxMat\pardifx\micvec + \bvec\left(\DPlanck\pardifx\Temp + \epsi^{2}\pardifx\meso \right) = -\absorpcoeff\micvec,
    \end{equation}
    \begin{equation}\label{eq:ModalMacMicSysmac}
        \pardift\Temp = \facalph\absorpcoeff\meso,
    \end{equation}
\end{subequations}
where
$$ \micvec = (\mick{1},\ldots,\mick{N})^{\top}\in\R^{N}, \qquad \fluxMat = \begin{bmatrix}
    0 & \ak{1} & & \\
    \ak{1} & 0 & \ddots &\\
    & \ddots & \ddots & \\
    &&&\ak{N-1}\\
    &&\ak{N-1}&0
\end{bmatrix}\in\R^{N\times N} \text{ and }\quad  \bvec = (\pknorm{1},0,\ldots,0)^{\top}\in \R^{N}. $$


\subsection{Dynamical low-rank approximation} \label{subsection:DLRA}
In this subsection, we give an overview of the DLRA put forth in \cite{doi:10.1137/050639703}. The fundamental motivation behind DLRA is to evolve a solution on a low-rank manifold of a given rank. To make this more concrete, let $\mic_{ik} = \mic_{k}(\vart,\varx_{i})$ be the evaluation of the k$^{\text{th}}$ moment of the microscopic variable $\mick{k}$ at spatial point $\varx_{i}$. The goal is then to evolve $\micmat$ such that it stays on the manifold of rank $\rank$ matrices, $\Mr$. 
In DLRA, a low-rank approximation is computed by projecting the dynamics of the problem onto the tangent space of the manifold \cite{doi:10.1137/050639703}.

For any matrix $\micmat_{\rank}\in\Mr\subset\R^{\Nx\times N}$ we have the factorization
\begin{equation}
    \micmat_{\rank}(\vart) = \Xmat(\vart)\Smat(\vart)\Vmat(\vart)^{\top}.
\end{equation}
This means that the solution matrix is spanned by the spatial basis $\Xmat:\Rp\to\R^{\Nx\times\rank}$, the moment basis $\Vmat:\Rp\to\R^{N\times\rank}$, and the coefficient matrix $\Smat:\Rp\to\R^{\rank\times\rank}$. In DLRA, the basis and coefficient matrices are evolved on the low-rank manifold such that, for $\micmat_{\rank}(t)\in\Mr$ and a given right-hand side $\textbf{F}:\R^{\Nx\times N}\rightarrow \R^{\Nx\times N}$, the following minimization problem is satisfied at all times $t$
\begin{equation*}
     \underset{\dot{\micmat}_{\rank}(\vart)\in\TyMr{\micmat_{\rank}(\vart)}}{\text{min}}\norm{\dot{\micmat}_{\rank}(\vart) - \textbf{F}(\micmat_{\rank}(\vart))}_{F}.
\end{equation*}
Here $\TyMr{\micmat_{\rank}(\vart)}$ denotes the tangent space of $\Mr$ at $\micmat_{\rank}$. A reformulation of this minimization problem \autocite[Lemma~4.1]{doi:10.1137/050639703} projects the right-hand side onto the tangent space and requires solving
\begin{equation*}
    \dot{\micmat}_{\rank}(\vart) = \DLRAProj(\micmat_{\rank}(\vart))\textbf{F}(\micmat_{\rank}(\vart))
\end{equation*}
where
\begin{equation*}
    \DLRAProj(\micmat_{\rank}) \textbf{Z} = \Xmat\Xmat^{\top}\textbf{Z} - \Xmat\Xmat^{\top}\textbf{Z}\Vmat\Vmat^{\top} + \textbf{Z}\Vmat\Vmat^{\top}
\end{equation*}
is the projection onto the tangent space $\TyMr{\micmat_{\rank}(\vart)}$.

Following \cite{doi:10.1137/050639703}, evolution equations can be derived for the factorized solution from the above equations as
\begin{subequations}
    \begin{equation*}
        \dot{\Smat}(\vart) = \Xmat(\vart)^{\top}\textbf{F}(\micmat_{\rank}(\vart))\Vmat(\vart),
    \end{equation*}
    \begin{equation*}
        \dot{\Xmat}(\vart) = (\textbf{I} - \Xmat(\vart)\Xmat(\vart)^{\top})\textbf{F}(\micmat_{\rank}(\vart))\Vmat(\vart)\Smat(\vart)^{-1},
    \end{equation*}
    \begin{equation*}
        \dot{\Vmat}(\vart) = (\textbf{I} - \Vmat(\vart)\Vmat(\vart)^{\top})\textbf{F}(\micmat_{\rank}(\vart))^{\top}\Xmat(\vart)\Smat(\vart)^{-\top}.
    \end{equation*}
\end{subequations}
In case of over-approximation of the rank, the coefficient matrix $\Smat$ becomes nearly singular which is a source of instabilities. Thus, robust integrators that do not invert the coefficient matrix have been developed \cite{doi:10.1007/s10543-013-0454-0,doi:10.1007/s10543-021-00873-0,robustBUG,ceruti2023parallel}. In this work, we use the fixed-rank BUG \cite{doi:10.1007/s10543-021-00873-0} and the augmented BUG integrator \cite{robustBUG}, which we describe here in brief.

For a given factorized initial solution, $\micmat^{0}_{\rank} = \Xmat^{0}\Smat^{0}\Vmat^{0,\top}$, one step of the fixed-rank BUG integrator updates the factors $\Xmat,\Smat,\Vmat$ from time $\vart_{0}$ to $\vart_{1}$ by the following sub-steps
\begin{itemize}
    \item[\textbf{K-step}] Update $\Xmat^{0}$ to $\Xmat^{1}$ by solving
\begin{equation*}
    \dot{\Kmat}(\vart) = \textbf{F}(\Kmat(\vart)\Vmat^{0,\top})\Vmat^{0}, \qquad \Kmat(\vart_{0}) = \Xmat^{0}\Smat^{0}.
\end{equation*}
Compute $\Kmat(\vart_{1}) = \Xmat^{1}\Smat^{K}$, e.g. by using QR decomposition, and store $\Mbug = \Xmat^{1,\top}\Xmat^{0} $. 

\item[\textbf{$\Lmat$-step}] Update $\Vmat^{0}$ to $\Vmat^{1}$ by solving
\begin{equation*}
    \dot{\Lmat}(\vart) = \Xmat^{0,\top}\textbf{F}(\Xmat^{0}\Lmat(\vart)), \qquad \Lmat(\vart_{0}) = \Smat^{0}\Vmat^{0,\top}.
\end{equation*}
Compute $\Lmat(\vart_{1})^{\top} = \Vmat^{1}\Smat^{L,\top}$, e.g. by using QR decomposition, and store $\Nbug = \Vmat^{1,\top}\Vmat^{0} $. 
\item[\textbf{$\Smat$-step}] Update the coefficient matrix $\Smat^{0}$ to $\Smat^{1} $  by performing Galerkin step in the updated basis
\begin{equation*}
    \dot{\Smat}(\vart) = \Xmat^{1,\top}\textbf{F}(\Xmat^{1}\Smat(\vart)\Vmat^{1,\top})\Vmat^{1}, \qquad \Smat(\vart_{0}) = \Mbug\Smat^{0}\Nbug^{\top}
\end{equation*}
and set $\Smat^{1} = \Smat(\vart_{1})$.
\end{itemize}
\noindent Then the approximation at the next time step is set as $\micmat_{\rank}(\vart_{1}) = \Xmat^{1}\Smat^{1}\Vmat^{1,\top} $. 

Using a fixed-rank integrator comes with several challenges. First, since the rank of the solution is not known beforehand, it is usually over-approximated, which leads to increased computational costs. Second, the rank of the solution may vary over time \cite{robustBUG,refId0,baumann2023energy}; thus, a fixed-rank integrator may not capture the solution correctly. Moreover, the fixed-rank BUG integrator does not preserve solution invariances. To overcome these, a rank-adaptive extension of the fixed-rank BUG integrator, known as the augmented BUG integrator, was presented in \cite{robustBUG} that appends extra spatial and angular basis vectors by reusing the old basis and truncates the rank to a prescribed tolerance $\vartheta$.

To present the algorithm, we denote all quantities of rank $2\rank$ with hats and those of rank $\rank$ without. Then, one step of the augmented BUG integrator updates the solution, $\micmat^{0}_{\rank} = \Xmat^{0}\Smat^{0}\Vmat^{0,\top}$ of rank $\rank$ (note that $\micmat_{\rank}$ represents low-rank approximation and not an approximation of rank $\rank$), from time $\vart_{0}$ to $\vart_{1}$ through the following steps
\begin{enumerate}
    \item Update and expand the spatial and angular basis in parallel.\\
    \begin{itemize}
        \item[\textbf{$\Kmat$-step}] Solve
\begin{equation*}\label{eq:raBUGKstep}
    \dot{\Kmat}(\vart) = \textbf{F}(\Kmat(\vart)\Vmat^{0,\top})\Vmat^{0}, \qquad \Kmat(\vart_{0}) = \Xmat^{0}\Smat^{0},
\end{equation*}
and compute the updated basis matrix $\hXmat\in\R^{\Nx\times 2\rank} $ as an orthonormal basis of $\left[ \Kmat(\vart_{1}), \Xmat^{0}\right]$ and store $\hMbug = \hXmat^{\top}\Xmat^{0}\in\R^{2\rank\times\rank}$. 

        \item[\textbf{$\Lmat$-step}]   Solve
\begin{equation*}\label{eq:raBUGLstep}
    \dot{\Lmat}(\vart) = \Xmat^{0,\top}\textbf{F}(\Xmat^{0}\Lmat(\vart)), \qquad \Lmat(\vart_{0}) = \Smat^{0}\Vmat^{0,\top},
\end{equation*}
and compute the updated basis matrix $\hVmat\in\R^{\N\times 2\rank} $ as an orthonormal basis of $ \left[\Lmat(\vart_{1})^{\top}, \Vmat^{0} \right] $ and store $\hNbug = \hVmat^{\top}\Vmat^{0}\in\R^{2\rank\times\rank}$.
    \end{itemize}

\item Update the coefficient matrix $\Smat^{0} $ to $\hSmat $  by performing Galerkin step in the updated and expanded basis
\begin{equation*}\label{eq:raBUGSstep}
    \dot{\hSmat}(\vart) = \hXmat^{1,\top}\textbf{F}(\hXmat\hSmat(\vart)\hVmat^{\top})\hVmat, \qquad \hSmat(\vart_{0}) = \hMbug\Smat^{0}\hNbug^{\top}.
\end{equation*}

\item  Truncation to new rank $\Nrank$.\\
Compute the SVD decomposition of $\hSmat$
\begin{equation*}
    \hSmat = \Pmat\BSigma\Qmat^{\top},
\end{equation*}
where $\Pmat,\Qmat\in\R^{2\rank\times 2\rank}$ are orthogonal matrices and $\BSigma\in\R^{2\rank\times 2\rank} $ is a diagonal matrix with singular values, $ \hat{\sigma}_{1},\ldots,\hat{\sigma}_{2\rank}$. The new rank $\Nrank$ is chosen as $1\leq\Nrank\leq 2\rank$ such that, for some user-defined $\vartheta$, the following is satisfied:
\begin{equation*}
    \left(\sum_{i = \Nrank + 1}^{2\rank}\hat{\sigma}_{i}^{2}\right)^{1/2} \leq \vartheta.
\end{equation*}
To set the updated factors, we define $\Pmattrunc$ and $\Qmattrunc$ to be the matrices containing the first $\Nrank$ columns of $\Pmat$ and $\Qmat$, respectively. $\BSigmatrunc$ is set as the diagonal matrix containing the first $\Nrank$ singular values of $\hSmat$. Then the updated factors are set as $\Xmat^{1} = \hXmat\Pmattrunc$, $\Vmat^{1} = \hVmat\Qmattrunc$ and $\Smat^{1} = \BSigmatrunc$ and, the approximation at time $\vart_{1}$ is then $\micmat^{1}_{\rank} = \Xmat^{1}\Smat^{1}\Vmat^{1,\top}$.
\end{enumerate}

\begin{remark}
    Note that often in practice, to truncate the rank, a relative tolerance of the form $\vartheta\cdot\norm{\BSigma}_{2}$ is used.
\end{remark}

\section{Energy stability of modal macro-micro equations}\label{section:EnergyStableMMM}
\noindent Having discretized the macro-micro equations in angle \eqref{eq:ModalMacMicSys}, in this section, we present an asymptotic-preserving spatio-temporal discretization and investigate its energy stability.

\subsection{Spatio-temporal discretization}
\noindent  We start by noting that for $i,j\in\{1,\ldots,N\}$,
we can represent the $(i,j)^{\text{th}}$ term of the flux matrix $\fluxMat$ by a quadrature rule. That is,
\begin{equation*}
    \mathrm{A}_{ij} = \intgomg{\omg\pk{i}\pk{j}} = \int_{-1}^{1}\omg\pk{i}(\omg)\pk{j}(\omg)\domg{1} \approx \sum_{k=1}^{N+1}\weight_{k}\quadpt_{k}\pk{i}(\quadpt_{k})\pk{j}(\quadpt_{k}),
\end{equation*}
where $(\quadpt_{k})_{k = 1,\ldots,N+1}$ and $(\weight_{k})_{k = 1,\ldots,N+1}$ are quadrature points and weights given by the Gauss-Legendre quadrature rule. If we define the matrices $\Tmat\in\R^{N\times(N+1)}$, with $T_{ik} = \sqrt{\weight_{k}}\pk{i}(\quadpt_{k})$, and $\quadptMat\in\R^{(N+1)\times(N+1)}$, with $\mathrm{M}_{ij} = \quadpt_i\delta_{ij}$, then we can write the flux matrix as $\fluxMat = \Tmat\quadptMat\Tmat^{\top}$. Given $(\absquadMat)_{ij} = \abs{\mathrm{M}_{ij}}$ we can define a stabilization matrix for a finite volume discretization as $\abs{\fluxMat} = \Tmat\absquadMat\Tmat^{\top}$ and $\fluxMatpm = \frac{1}{2}\Tmat(\quadptMat \pm \absquadMat)\Tmat^{\top}$.
\begin{remark}
    The choice of the stabilization matrix used here is not the standard Roe matrix $\RoeTmat\absRoeEigenval\RoeTmat^{\top}$ where $\fluxMat = \RoeTmat\RoeEigenval\RoeTmat^{\top}$ is the eigendecomposition of the flux matrix. That is, the columns of $\RoeTmat\in\R^{N\times N}$ consist of orthogonal eigenvectors of $\fluxMat$ and $\RoeEigenval\in\R^{N\times N} $ has the corresponding eigenvalues on the diagonal. The factorization of the flux matrix used, consisting of transformation matrices in $\R^{N\times(N+1)}$, is needed for the diagonalization of the modal scheme for showing stability in the energy norm. This choice of stabilization matrix was first presented in \cite{einkemmer2022asymptoticpreserving} for the radiative transport equation.
\end{remark}
We discretize in space using an equidistant staggered grid, with $\deltax = 1/\Nx$, for a given number of $\Nx\in\N$ spatial cells. The cell interfaces are given by $\varx_{1/2},\ldots,\varx_{\Nx+1/2}$ and the midpoints by $\varx_{i}$ for $i\in\{1,\ldots,\Nx\}$.
The temperature ($\Temp$) and mesoscopic variable ($\meso$) are resolved at the full grid points $\varx_{i}$ whereas the microscopic variable ($\micvec$) is evaluated at the cell interfaces $\varx_{i+1/2}$.\\
Since $\Planck = \StefBoltzconst\sol\Temp^{4}$ we get $\DPlanck = 4\StefBoltzconst\sol\Temp^{3}$ and thus, to simplify the presentation of the algorithm, we define
\begin{equation*}
    \DPlanckC(\vart,\varx) = 4(\Temp(\vart,\varx))^{3}.
\end{equation*}
The values of $\DPlanckC$ at $\varx_{i}$ and $\varx_{i+1/2}$ are given by $\DPlanckC^{n}_{i} = \DPlanckC(\vart^{n},\varx_{i}) $ and $ \DPlanckC^{n}_{i+1/2} = \frac{\DPlanckC^{n}_{i+1} + \DPlanckC^{n}_{i}}{2} $, respectively. Finally, to discretize in time we employ a forward-backward Euler scheme and obtain the following modal macro-micro scheme
\begin{subequations}\label{eq:FDLMM}
    \begin{equation}\label{eq:FDLMMmeso}
        \frac{\epsi^{2}}{\sol}\left( \frac{\meso^{n+1}_{i} - \meso^{n}_{i}}{\deltat} \right) + \StefBoltzconst\facalph\hspace{.5mm}\absorpcoeff_{i}\DPlanckC^{n}_{i}\meso^{n+1}_{i} + \frac{\pknorm{1}}{2}\DiffO\mick{1,i}^{n+1} = -\absorpcoeff_{i}\meso^{n+1}_{i},
    \end{equation}
    \begin{equation}\label{eq:FDLMMmic}
        \frac{\epsi^{2}}{\sol}\left( \frac{\micvec^{n+1}_{i+1/2} - \micvec^{n}_{i+1/2}}{\deltat} \right) + \epsi\AdvecOp\micvec^{n}_{i+1/2} + \bvec\deltaO\left(\StefBoltzconst\sol\hspace{.5mm}\DPlanckC^{n}_{i+1/2}\Temp^{n}_{i+1/2} + \epsi^{2}\meso^{n}_{i+1/2} \right) = -\absorpcoeff_{i+1/2}\micvec^{n+1}_{i+1/2},
    \end{equation}
    \begin{equation}\label{eq:FDLMMmac}
        \frac{\Temp^{n+1}_{i} - \Temp^{n}_{i}}{\deltat} = \facalph\hspace{.5mm}\absorpcoeff_{i}\meso^{n+1}_{i},
    \end{equation}
\end{subequations}
where,
\begin{align*}
    \Diffm\micvec_{i+1/2} &= \frac{\micvec_{i+1/2} - \micvec_{i-1/2}}{\deltax}, & \qquad\Diffp\micvec_{i+1/2} &= \frac{\micvec_{i+3/2} - \micvec_{i+1/2}}{\deltax},\\
    \DiffO\micvec_{i} &= \frac{\micvec_{i+1/2} - \micvec_{i-1/2}}{\deltax} (= \Diffm\micvec_{i+1/2}), & \qquad\deltaO \Temp_{i+1/2} &= \frac{\Temp_{i+1} - \Temp_{i}}{\deltax},\\
    \AdvecOp\micvec_{i+1/2} &= (\fluxMatp\Diffm + \fluxMatm\Diffp)\micvec_{i+1/2}. &&
\end{align*}

\begin{theorem}\label{theorem:APLin}
In the limit $\epsi \to 0$, the modal macro-micro scheme \eqref{eq:FDLMM} gives a consistent discretization of the diffusion equation 
    \begin{equation*}
    (1 + \frac{2\StefBoltzconst\DPlanckC}{\specheat})\pardift\Temp = \frac{2\StefBoltzconst\sol}{3\specheat}\pardifx\left(\frac{1}{\absorpcoeff}\DPlanckC\pardifx\Temp\right).
\end{equation*}
\end{theorem}
\begin{proof}
    The $\bigO{\epsi^{0}}$ term in \eqref{eq:FDLMMmic} is given by
\begin{equation}\label{eq:eq5}
    -\absorpcoeff_{i+1/2}\mick{1,i+1/2}^{n+1} = \pknorm{1}\StefBoltzconst\sol\DPlanckC^{n}_{i+1/2}\deltaO\Temp^{n}_{i+1/2}.
\end{equation}
Similarly, the $\bigO{\epsi^{0}}$ term in \eqref{eq:FDLMMmeso} is
\begin{equation*}
    -\absorpcoeff_{i}\meso^{n+1}_{i} = \StefBoltzconst\facalph\hspace{.5mm}\absorpcoeff_{i}\DPlanckC^{n}_{i}\meso^{n+1}_{i} + \frac{\pknorm{1}}{2}\DiffO\mick{1,i}^{n+1} 
\end{equation*}
which on substituting $\mick{1,i+1/2}^{n+1}$ from \eqref{eq:eq5} and collecting $\meso_{i}^{n+1}$ terms on the left hand side yields
\begin{equation*}
    (1+\StefBoltzconst\facalph\DPlanckC^{n}_{i})\absorpcoeff_{i}\meso_{i}^{n+1} = \frac{\pknorm{1}^{2}}{2}\StefBoltzconst\sol\left[ \dfrac{\frac{\DPlanckC^{n}_{i+1/2}}{\absorpcoeff_{i+1/2}}(T^{n}_{i+1} - T^{n}_{i}) - \frac{\DPlanckC^{n}_{i-1/2}}{\absorpcoeff_{i-1/2}}(T^{n}_{i} - T^{n}_{i-1})}{(\deltax)^{2}} \right].
\end{equation*}
Thus, substituting $\absorpcoeff_{i}\meso^{n+1}_{i}$ in \eqref{eq:FDLMMmac}
\begin{equation*}
    (1+\frac{2\StefBoltzconst\DPlanckC^{n}_{i}}{\specheat})\left(\frac{\Temp^{n+1}_{i} - \Temp^{n}_{i}}{\deltat}\right) = \frac{2\StefBoltzconst\sol}{3\specheat}\left[ \dfrac{\frac{\DPlanckC^{n}_{i+1/2}}{\absorpcoeff_{i+1/2}}(T^{n}_{i+1} - T^{n}_{i}) - \frac{\DPlanckC^{n}_{i-1/2}}{\absorpcoeff_{i-1/2}}(T^{n}_{i} - T^{n}_{i-1})}{(\deltax)^{2}} \right],
\end{equation*}
where we re-substitute $\facalph = \frac{2}{\specheat}$. This is a discretization of the limiting diffusion equation with an explicit Euler discretization in time and centered differences for spatial derivatives.
    
\end{proof}
\subsection{Stability analysis}
Next, we investigate the stability of the modal macro-micro scheme \eqref{eq:FDLMM} in energy norm for a linearized version of the problem as described in \cite{JIN2017182}. The linearization assumes that the particles are emitted from the background material proportional to the temperature (instead of the 4$^{th}$ power of temperature as given by the Stefan-Boltzmann law). That is, we set $\Planck = \StefBoltzconst\sol\Temp$ and thus $\DPlanckC = 1$. Other strategies to linearize the problem include the Su-Olsen closure \cite{SU19971035} in which the specific heat, $\specheat$, is assumed to be proportional to $\Temp^{3}$. 

Substituting the value of $\DPlanckC$ in \eqref{eq:FDLMM} we get the following modal macro-micro scheme with linear emission of particles:
\begin{subequations}\label{eq:LinMacMicsys}
    \begin{equation}\label{eq:LinMacMicsysmeso}
        \frac{\epsi^{2}}{\sol}\left( \frac{\meso^{n+1}_{i} - \meso^{n}_{i}}{\deltat} \right) + \StefBoltzconst\facalph\hspace{.5mm}\absorpcoeff_{i}\meso^{n+1}_{i} + \frac{\pknorm{1}}{2}\DiffO\mick{1,i}^{n+1} = -\absorpcoeff_{i}\meso^{n+1}_{i},
    \end{equation}
    \begin{equation}\label{eq:LinMacMicsysmic}
         \frac{\epsi^{2}}{\sol}\left( \frac{\micvec^{n+1}_{i+1/2} - \micvec^{n}_{i+1/2}}{\deltat} \right) + \epsi\AdvecOp\micvec^{n}_{i+1/2} + \bvec\deltaO\left(\StefBoltzconst\sol\hspace{.5mm}\Temp^{n}_{i+1/2} + \epsi^{2}\meso^{n}_{i+1/2} \right) = -\absorpcoeff_{i+1/2}\micvec^{n+1}_{i+1/2},
    \end{equation}
    \begin{equation}\label{eq:LinMacMicsysmac}
        \frac{\Temp^{n+1}_{i} - \Temp^{n}_{i}}{\deltat} = \facalph\hspace{.5mm}\absorpcoeff_{i}\meso^{n+1}_{i}.
    \end{equation}
\end{subequations}
The following norms are defined for the scalar- and vector-valued functions
\begin{equation*}
    \norm{u}^{2} = \sum_{i}u_{i}^{2}\deltax, \qquad \norm{\bphi}^{2} = \sum_{i}(\bphi_{i+1/2}^{\top}\bphi_{i+1/2})\deltax.
\end{equation*}
\noindent Then, for the linearized modal macro-micro scheme, we have the following stability result:
\begin{theorem}\label{theorem:LinMacMicEnStab}
    Assume that the time step $\deltat$ fulfills the CFL condition for all $k$, such that $\quadpt_{k}\neq 0$,
    \begin{equation}\label{eq:CFLCond}
        \deltat \leq \frac{1}{5\sol\beta_{N}}\left( \frac{2\epsi\deltax}{\abs{\quadpt_{k}}} + \frac{\absorpcoefflb\deltax^{2}}{\quadpt_{k}^{2}} \right),
    \end{equation}
    where $\beta_{N} = \underset{k}{\mathrm{max}}\hspace{2mm}\weight_{k}(N+1)$ and $\sol $ is the speed of light. Then, the scheme \eqref{eq:LinMacMicsys} is energy stable, that is, $$\energy^{n+1} \leq \energy^{n},$$ where the energy is defined as
    \begin{equation*}
        \energy^{n} = \norm{\StefBoltzconst\Temp^{n} + \frac{\epsi^{2}}{\sol}\meso^{n}}^{2} + \norm{\frac{\epsi}{\pknorm{0}\sol}\micmat^{n}}^{2} + \norm{\sqrt{\frac{\StefBoltzconst\specheat}{2}}\Temp^{n}}^{2}.
    \end{equation*}
\end{theorem}
\begin{remark}
    For the sake of compactness, the proof of this theorem, along with all the required lemmas, are presented in \Cref{appendix:ProofLinMMMstab}. The proof follows the energy stability result in \cite{JIN2017182} and combines it with the results obtained for the modal macro-micro scheme for radiation transport from \cite{einkemmer2022asymptoticpreserving}. It is roughly divided into three parts; the first part bounds $\norm{\StefBoltzconst\Temp^{n+1} + \frac{\epsi^{2}}{\sol}\meso^{n+1}}^{2} + \norm{\frac{\epsi}{\pknorm{0}\sol}\micmat^{n+1}}^{2}$ from above using \eqref{eq:LinMacMicsysmeso} and \eqref{eq:LinMacMicsysmic}. In the second part we derive an upper bound for $\norm{\sqrt{\frac{a\specheat}{2}}\Temp^{n+1}}^{2}$ from \eqref{eq:LinMacMicsysmac}. Combining the bounds obtained in the first and the second part, we show energy stability subject to step size restriction given by the CFL condition \eqref{eq:CFLCond} in the third part of the proof. 
\end{remark}


\section{Dynamical low-rank approximation for the modal macro-micro equations}\label{section:DLRAforMMM}
\noindent The macro-micro decomposition \cite{MR2460781,MR1842224} allows us to construct an asymptotic-preserving and energy-stable numerical algorithm for the thermal radiative transfer equations. However, the microscopic variable $\mic$ is still a high-dimensional quantity since it depends on time, space, and direction of flight. Thus, to reduce computational costs, we use dynamical low-rank approximation \cite{doi:10.1137/050639703} to approximate the solution of $\mic$ by a low-rank factorization. This section is divided into two subsections; in the first subsection, we derive evolution equations for the low-rank factorization of $\mic$ for the fixed-rank BUG integrator \cite{doi:10.1007/s10543-021-00873-0}. We present an asymptotic-preserving spatio-temporal discretization and show that the numerical scheme is energy stable for the linearization presented in \Cref{section:EnergyStableMMM}. In the second subsection, we extend the scheme to the augmented BUG integrator \cite{robustBUG}.

Consider the microscopic equation \eqref{eq:ModalMacMicSysmic} given by
\begin{equation}
    \frac{\epsi^{2}}{\sol}\pardift\micvec + \epsi\fluxMat\pardifx\micvec + \bvec\left(\StefBoltzconst\sol\DPlanckC\hspace{.5mm}\pardifx\Temp + \epsi^{2}\pardifx\meso \right) = -\absorpcoeff\micvec.
\end{equation}
The low-rank ansatz for the microscopic variable $\micvec$ reads:
\begin{equation*}\label{eq:DLRAansatz}  
    \micvec(\vart,\varx) \approx \sum_{p,q = 1}^{\rank}\Xc_{p}(\vart,\varx)\Sc_{pq}(\vart)\Vc_{q}(\vart)^{\top},
\end{equation*}
where $\rank\in\N$ is some given rank and $\Vc_{q} = (V_{1,q},\ldots,V_{N,q})^{\top} \in \R^{N} $.
Thus, we can write the above sum as
\begin{equation*}
    \micvec(\vart,\varx) = \Xcvec(\vart,\varx)^{\top}\Smat(\vart)\Vmat(\vart)^{\top},
\end{equation*}
where
\begin{equation*}
    \Xcvec = \left( \Xc_{1},\ldots,\Xc_{r} \right)^{\top}\in\R^{\rank}, \quad \Smat = \left( \Sc_{pq} \right)_{p,q=1}^{\rank}\in\R^{\rank\times\rank}, \quad \Vmat = \begin{bmatrix}
        \Vc_{1} & \ldots & \Vc_{r}
    \end{bmatrix}\in\R^{N\times\rank}.
\end{equation*}
\subsection{Fixed-rank modal macro-micro BUG integrator}\label{section:BUGforMMM}
With this low-rank ansatz for $\micvec$ we can now write down the individual steps of the fixed-rank BUG integrator \cite{doi:10.1007/s10543-021-00873-0} for updating the microscopic variable. To this end let the solution at time $\vart_{n}$ be given by $\micvec^{n}(\varx) = \Xcvec^{n}(\varx)^{\top}\Smat^{n}\Vmat^{n,\top} $, then the evolution equations for updating $\Xcvec, \Smat, \Vmat$ are as follows:
\begin{itemize}
    \item[\textbf{K-step}]  For $\Kcvec(\vart,\varx)^{\top} = \Xcvec(\vart,\varx)^{\top}\Smat(t) $ solve
\begin{equation*}
    \frac{\epsi^{2}}{\sol}\pardift\Kcvec(\vart,\varx) = -\epsi\left[\Vmat^{n,\top}\fluxMat\Vmat^{n} \right]\pardifx\Kcvec(\vart,\varx) - \Vmat^{n,\top}\bvec(\StefBoltzconst\sol\DPlanckC\pardifx\Temp + \epsi^{2}\pardifx\meso) - \absorpcoeff\Kcvec(\vart,\varx),
\end{equation*}
with the initial condition $\Kcvec(\vart_{n},\varx) = \Xcvec^{n}(\varx)^{\top}\Smat^{n}$. We denote the updated spatial basis vectors by $\Xcvec^{n+1}(\varx)$ which is obtained as the orthonormal basis of $\Kcvec(\vart_{n+1},x)$.

\item[\textbf{L-step}] For $\Lmat(\vart) = \Vmat(\vart)\Smat(\vart)^{\top}$ solve
\begin{equation*}
    \frac{\epsi^{2}}{\sol}\dot{\Lmat}(\vart) = -\epsi\fluxMat^{\top}\Lmat(\vart)\intgx{\pardifx\Xcvec^{n},\Xcvec^{n,\top}} - \bvec\intgx{\StefBoltzconst\sol\DPlanckC\pardifx\Temp + \epsi^{2}\pardifx\meso,\Xcvec^{n,\top}}.
\end{equation*}
where $\intgx{\cdot,\cdot} $ denotes the $\LL$ - inner product over the spatial domain, and we have the initial condition given by $ \Lmat(\vart_{n}) = \Vmat^{n}\Smat^{n,\top} $. We denote the updated angular basis matrix by $\Vmat^{n+1}$, which is obtained as the orthonormal basis of $\Lmat(\vart_{n+1})$.

\item[\textbf{S-step}] Perform a Galerkin step in the updated spatial and angular basis according to 
\begin{equation*}
    \begin{aligned}
        \frac{\epsi^{2}}{\sol}\dot{\Smat}(\vart) &= -\epsi\intgx{\Xcvec^{n+1},\pardifx\Xcvec^{n+1,\top}}\Smat(\vart)\Vmat^{n+1,\top}\fluxMat\Vmat^{n+1} - \intgx{\Xcvec^{n+1},\StefBoltzconst\sol\DPlanckC\pardifx\Temp + \epsi^{2}\pardifx\meso}\bvec^{\top}\Vmat^{n+1}\\
        &\qquad -\intgx{\absorpcoeff\Xcvec^{n+1},\Xcvec^{n+1,\top}}\Smat(\vart),
    \end{aligned}
\end{equation*}
with the initial condition $\Smat(\vart_{n}) = \intgx{\Xcvec^{n+1},\Xcvec^{n,\top}}\Smat^{n}\Vmat^{n,\top}\Vmat^{n+1} $.

\end{itemize}


\subsubsection{Spatio-temporal discretization}
The update equations for $\Temp$ and $\meso$ from \Cref{subsection:Modal MM} along with the evolution equations for the low-rank factors of $\micvec$ give the fixed-rank modal macro-micro BUG equations for the thermal radiative transfer equations. Similar to \Cref{section:EnergyStableMMM}, we discretize the fixed-rank modal macro-micro BUG equations in space and time. First, we define
\begin{equation*}
    \Xvec^{n}_{i+1/2} = \frac{1}{\deltax}\int_{x_{i}}^{x_{i+1}}\Xcvec(\vart_{n},\varx) \dx{1}
\end{equation*}
and $\Kvec_{i+1/2}(\vart) = \Xvec_{i+1/2}(\vart)^{\top}\Smat(\vart)\in\R^{\rank} $. Then, for the prescribed data $\Xvec^{n},\Vmat^{n},\Smat^{n},\meso^{n},\Temp^{n} $ at time $\vart_{n}$ the fixed-rank modal macro-micro BUG scheme updates the solution at time $\vart_{n}$ through the following steps,
\begin{itemize}
    \item[\textbf{K-step}] Update
    \begin{equation}\label{eq:FDKstep}
         \frac{\epsi^{2}}{\sol}\left[\frac{\Kvec^{n+1}_{i+1/2} -\Kvec^{n}_{i+1/2}}{\deltat}\right] = -\epsi\AdvecOpK\Kvec^{n}_{i+1/2} - \Vmat^{n,\top}\bvec\deltaO(\StefBoltzconst\sol\DPlanckC^{n}_{i+1/2}\Temp^{n}_{i+1/2} + \epsi^{2}\meso^{n}_{i+1/2}) - \absorpcoeff_{i+1/2}\Kvec^{n+1}_{i+1/2},
    \end{equation}
    where $\Kvec^{n,\top}_{i+1/2} = \Xvec^{n,\top}_{i+1/2}\Smat^{n} $ and
    \begin{equation*}
        \AdvecOpK\Kvec^{n}_{i+1/2} = \left[ \Vmat^{n,\top}\fluxMatp\Vmat^{n} \right] \Diffm\Kvec^{n}_{i+1/2}+ \left[ \Vmat^{n,\top}\fluxMatm\Vmat^{n} \right]\Diffp\Kvec^{n}_{i+1/2}.
    \end{equation*}
    Compute $\Xvec^{n+1}_{i+1/2}$ as the orthonormal basis of $\Kvec^{n+1}_{i+1/2}$.

    \item[\textbf{L-step}] Update
    \begin{equation}\label{eq:FDLstep}
    \begin{aligned}
        \frac{\epsi^{2}}{\sol}\left[ \frac{\Lmat^{n+1} - \Lmat^{n}}{\deltat} \right] &= -\epsi\AdvecOpL\Lmat^{n} - \bvec\sum_{i}\Xvec^{n,\top}_{i+1/2}\deltaO(\StefBoltzconst\sol\DPlanckC^{n}_{i+1/2}\Temp^{n}_{i+1/2} + \epsi^{2}\meso^{n}_{i+1/2})\\ &\qquad - \Lmat^{n+1}\sum_{i}\absorpcoeff_{i+1/2}\Xvec^{n}_{i+1/2}\Xvec^{n,\top}_{i+1/2}
    \end{aligned}
    \end{equation}
    where $\Lmat^{n}  = \Vmat^{n}\Smat^{n,\top} $ and
    \begin{equation*}
        \AdvecOpL\Lmat^{n} = \fluxMatp\Lmat^{n}\sum_{i}\Diffm\Xvec^{n}_{i+1/2}\Xvec^{n,\top}_{i+1/2} + \fluxMatm\Lmat^{n}\sum_{i}\Diffp\Xvec^{n}_{i+1/2}\Xvec^{n,\top}_{i+1/2}.
    \end{equation*}
    Compute $\Vmat^{n+1}$ as the orthonormal basis of $\Lmat^{n+1}$.
    
    \item[\textbf{S-step}]  Update
    \begin{equation}\label{eq:FDSstep}
    \begin{aligned}
        \frac{\epsi^{2}}{\sol}\left[ \frac{\Smat^{n+1} - \Smatinit^{n}}{\deltat} \right] &= -\epsi\AdvecOpS\Smatinit^{n} - \sum_{i}\Xvec^{n+1}_{i+1/2}\deltaO(\StefBoltzconst\sol\DPlanckC^{n}_{i+1/2}\Temp^{n}_{i+1/2} + \epsi^{2}\meso^{n}_{i+1/2})\bvec^{\top}\Vmat^{n+1} \\ & \qquad - \sum_{i}\absorpcoeff_{i+1/2}\Xvec^{n+1}_{i+1/2}\Xvec^{n+1,\top}_{i+1/2}\Smat^{n+1}
    \end{aligned}
    \end{equation}
    where $\Smatinit^{n} = \sum_{j}\Xcvec^{n+1}_{j+1/2}\Xcvec^{n,\top}_{j+1/2}\Smat^{n}\Vmat^{n,\top}\Vmat^{n+1} $ and
    \begin{equation*}
    \begin{aligned}
        \AdvecOpS\Smat^{n} &= \sum_{i}\Xvec^{n+1}_{i+1/2}\Diffm\Xvec^{n+1,\top}_{i+1/2}\Smat^{n}\Vmat^{n+1,\top}\fluxMatp\Vmat^{n+1}\\ &\qquad + \sum_{i}\Xvec^{n+1}_{i+1/2}\Diffp\Xvec^{n+1,\top}_{i+1/2}\Smat^{n}\Vmat^{n+1,\top}\fluxMatm\Vmat^{n+1},
    \end{aligned}
    \end{equation*}

    \item[] Update $\Temp,\meso$:
    \begin{equation}\label{eq:DLRAmesotFD}
        \frac{\epsi^{2}}{\sol}\left( \frac{\meso^{n+1}_{i} - \meso^{n}_{i}}{\deltat} \right) + \StefBoltzconst\facalph\hspace{.5mm}\absorpcoeff_{i}\DPlanckC^{n}_{i}\meso^{n+1}_{i} + \frac{\pknorm{1}}{2}\DiffO\Xvec^{n+1,\top}_{i+1/2}\Smat^{n+1}\Vmat^{n+1,\top}\unitvecN = -\absorpcoeff_{i}\meso^{n+1}_{i},
    \end{equation}
    \begin{equation}\label{eq:DLRAmacFD}
        \frac{\Temp^{n+1}_{i} - \Temp^{n}_{i}}{\deltat} = \facalph\hspace{.5mm}\absorpcoeff_{i}\meso^{n+1}_{i}.
    \end{equation}
    
\end{itemize}

\begin{theorem}\label{theorem:APBUG}
    In the limit $\epsi \to 0$, the fixed-rank modal macro-micro BUG scheme given by \cref{eq:FDKstep,eq:FDLstep,eq:FDSstep,eq:DLRAmesotFD,eq:DLRAmacFD} gives a consistent discretization of the diffusion equation 
    \begin{equation*}
    \left(1 + \frac{2\StefBoltzconst\DPlanckC}{\specheat}\right)\pardift\Temp = \frac{2\StefBoltzconst\sol}{3\specheat}\pardifx\left(\frac{1}{\absorpcoeff}\pardifx\Temp\right).
    \end{equation*}
\end{theorem}
\begin{proof}
    As $\epsi\to 0$, from the K-step \eqref{eq:FDKstep} and L-step \eqref{eq:FDLstep} we obtain
    \begin{equation*}\label{eq:DLRAlinAP1}
        \Vmat^{n,\top}\bvec\DPlanckC^{n}_{i+1/2}\deltaO(\StefBoltzconst\sol\Temp^{n}_{i+1/2}) = -\absorpcoeff_{i+1/2}\Kvec^{n+1}_{i+1/2}
    \end{equation*}
    and 
    \begin{equation*}\label{eq:DLRAlinAP2}
        \Lmat^{n+1}\sum_{i}\absorpcoeff_{i+1/2}\Xvec^{n}_{i+1/2}\Xvec^{n,\top}_{i+1/2} = -\bvec\sum_{i}\Xvec^{n,\top}_{i+1/2}\DPlanckC^{n}_{i+1/2}\deltaO(\StefBoltzconst\sol\Temp^{n}_{i+1/2}).
    \end{equation*}
    If $\Kvec^{n+1}_{i+1/2} $ is factorized as $\Kvec^{n+1,\top}_{i+1/2} = \Xvec^{n+1,\top}_{i+1/2}\Smat_{K} $ then $\frac{\DPlanckC^{n}_{i+1/2}}{\absorpcoeff_{i+1/2}}\deltaO(\StefBoltzconst\sol\Temp^{n}_{i+1/2})  $ lies in the range space of $\Xvec^{n+1}_{i+1/2}$. Similarly, if $\Lmat^{n+1} = \Vmat^{n+1}\Smat_{L}^{\top} $ and $\left( \Smat_{L}^{\top}\sum_{i}\absorpcoeff_{i+1/2}\Xvec^{n}_{i+1/2}\Xvec^{n,\top}_{i+1/2} \right) $ is invertible, then $ \bvec $ lies in the range space of $\Vmat^{n+1}$.

    Now as $\epsi\to 0$, from the S-step we obtain
    \begin{equation}\label{eq:DLRAlinAP3}
        - \left(\sum_{i}\DPlanckC^{n}_{i+1/2}\deltaO(\StefBoltzconst\sol\Temp^{n}_{i+1/2})\Xvec^{n+1}_{i+1/2}\right)\left(\bvec^{\top}\Vmat^{n+1}\right) = \left(\sum_{i}\absorpcoeff_{i+1/2}\Xvec^{n+1}_{i+1/2}\Xvec^{n+1,\top}_{i+1/2}\right)\Smat^{n+1}.
    \end{equation}
    Note that since $\frac{\DPlanckC^{n}_{i+1/2}}{\absorpcoeff_{i+1/2}}\deltaO(\StefBoltzconst\sol\Temp^{n}_{i+1/2}) = \left[\sum_{j}\frac{\DPlanckC^{n}_{j+1/2}}{\absorpcoeff_{j+1/2}}\deltaO(\StefBoltzconst\sol\Temp^{n}_{j+1/2})\Xvec^{n+1,\top}_{j+1/2} \right]\Xvec^{n+1}_{i+1/2} $ we have 
    \begin{equation*}\label{eq:DLRAlinAP4}
    \begin{aligned}
        \sum_{i}\deltaO(\StefBoltzconst\sol\DPlanckC^{n}_{i+1/2}\Temp^{n}_{i+1/2})\Xvec^{n+1}_{i+1/2} &= \sum_{i}\absorpcoeff_{i+1/2}\left(\frac{\DPlanckC^{n}_{i+1/2}}{\absorpcoeff_{i+1/2}}\deltaO(\StefBoltzconst\sol\Temp^{n}_{i+1/2})\Xvec^{n+1}_{i+1/2} \right)\\
        & = \left(\sum_{i}\absorpcoeff_{i+1/2}\Xvec^{n+1}_{i+1/2}\Xvec^{n+1,\top}_{i+1/2}\right)\left(\sum_{j}\frac{\DPlanckC^{n}_{j+1/2}}{\absorpcoeff_{j+1/2}}\deltaO(\StefBoltzconst\sol\Temp^{n}_{j+1/2})\Xvec^{n+1}_{j+1/2}\right).
    \end{aligned}
    \end{equation*}
    Hence, \eqref{eq:DLRAlinAP3} becomes
    \begin{equation*}\label{eq:DLRAlinAP5}
        \Smat^{n+1} = - \left(\sum_{j}\frac{\DPlanckC^{n}_{j+1/2}}{\absorpcoeff_{j+1/2}}\deltaO(\StefBoltzconst\sol\Temp^{n}_{j+1/2})\Xvec^{n+1}_{j+1/2}\right)\left(\bvec^{\top}\Vmat^{n+1}\right)
    \end{equation*}
    and since $\frac{\DPlanckC^{n}_{i+1/2}}{\absorpcoeff_{i+1/2}}\deltaO(\StefBoltzconst\sol\Temp^{n}_{i+1/2})$ and $\bvec$ lie in the range of the updated spatial and angular basis, scalar multiplication with $\Xvec^{n+1,\top}_{i+1/2}$ and $\Vmat^{n+1,\top}$ from the left and right implies
    \begin{equation}
        \micvec^{n+1}_{i+1/2} = -\frac{\DPlanckC^{n}_{i+1/2}}{\absorpcoeff_{i+1/2}}\deltaO(\StefBoltzconst\sol\Temp^{n}_{i+1/2})\bvec^{\top}.
    \end{equation}
    The rest of the proof follows along the lines of \Cref{theorem:APLin}. 
\end{proof}


\subsubsection{Energy stability}

Next, we investigate the stability of the fixed-rank modal macro-micro BUG scheme in energy norm for the linearized problem \eqref{eq:LinMacMicsys}. For the following decomposition of the micro variable
\begin{equation*}
    \micmat^{n} = \begin{bmatrix}
        \micvec^{n}_{1/2}\\
        \vdots\\
        \micvec^{n}_{\Nx + 1/2}
    \end{bmatrix} = \Xmat^{n}\Smat^{n}\Vmat^{n,\top},
\end{equation*}
the norm is defined as
\begin{equation*}
    \norm{\micmat^{n}}^{2} = \norm{\Xmat^{n}\Smat^{n}\Vmat^{n,\top}}^{2}_{F} \deltax 
\end{equation*}
Additionally, we state the following property that we use in the proof of energy stability:
\begin{property}\label{lemma:sumofsquares}
    For any $\{ c_{i}\}_{i=1,\ldots,\Nx}\in\R $ and $\{d_{i}\}_{i=1,\ldots,\Nx}\in\R $ we have
    \begin{equation*}
        \sum_{i}c_{i}d_{i} = \frac{1}{2}\sum_{i}c^{2}_{i} + \frac{1}{2}\sum_{i}d^{2}_{i} - \frac{1}{2}\sum_{i}(c_{i}-d_{i})^{2}. 
    \end{equation*}
\end{property}

\begin{theorem}\label{theorem:DLRALinMacMicEnStab}
    Assume that the time step $\deltat$ fulfills the CFL condition \eqref{eq:CFLCond} from \Cref{theorem:LinMacMicEnStab}. Then, the fixed-rank modal macro-micro BUG scheme given by \cref{eq:FDKstep,eq:FDLstep,eq:FDSstep,eq:DLRAmesotFD,eq:DLRAmacFD} is energy stable for the linearised problem \eqref{eq:LinMacMicsys}, that is, $$\energy^{n+1} \leq \energy^{n},$$ where the energy is defined as
    \begin{equation*}\label{eq:ESDLRAthm2}
        \energy^{n} = \norm{\StefBoltzconst\Temp^{n} + \frac{\epsi^{2}}{\sol}\meso^{n}}^{2} + \norm{\frac{\epsi}{\pknorm{0}\sol}{\normalfont\Xmat^{n}\Smat^{n}\Vmat^{n,\top}}}^{2} + \norm{\sqrt{\frac{\StefBoltzconst\specheat}{2}}\Temp^{n}}^{2}.
    \end{equation*}
\end{theorem}
\begin{proof}
   Since the proof of the theorem follows along the lines of \Cref{theorem:LinMacMicEnStab}, to shorten the presentation, we only present the parts of the proof that differ from \Cref{theorem:LinMacMicEnStab}. That is, we show that the inequalities \eqref{eq:Thm1Eq3} and \eqref{eq:Thm1Eq2} hold for the low-rank scheme. We begin by rewriting the S-step \eqref{eq:FDSstep} of the fixed-rank modal macro-micro BUG scheme as 
    \begin{equation*}\label{eq:ESDLRAthm3}
        \begin{aligned}
            \frac{\epsi^{2}}{\sol\deltat}\Smat^{n+1} &=   \frac{\epsi^{2}}{\sol\deltat}\Smatinit^{n} -\epsi\AdvecOpS\Smatinit^{n} - \sum_{j}\Xvec^{n+1}_{j+1/2}\deltaO(\StefBoltzconst\sol\Temp^{n}_{j+1/2} + \epsi^{2}\meso^{n}_{j+1/2})\bvec^{\top}\Vmat^{n+1} \\ & \qquad - \sum_{j}\absorpcoeff_{j+1/2}\Xvec^{n+1}_{j+1/2}\Xvec^{n+1,\top}_{j+1/2}\Smat^{n+1}
        \end{aligned}
    \end{equation*}
    and multiply $\Xvec^{n+1,\top}_{i+1/2}$ and $\Vmat^{n+1,\top}$ from the left and the right, respectively. If we define $\widetilde{\micvec}^{n}_{i+1/2} = \Xvec^{n+1,\top}_{i+1/2}\Smatinit^{n}\Vmat^{n+1,\top} $ and $\micvec^{n+1}_{i+1/2} = \Xvec^{n+1,\top}_{i+1/2}\Smat^{n+1}\Vmat^{n+1,\top}$, we obtain
    \begin{equation}\label{eq:ESDLRAthm4}
        \begin{aligned}
            \frac{\epsi^{2}}{\sol\deltat}\micvec^{n+1}_{i+1/2} &= \frac{\epsi^{2}}{\sol\deltat}\widetilde{\micvec}^{n}_{i+1/2} -\epsi\Xvec^{n+1,\top}_{i+1/2}\AdvecOpS\Smatinit^{n}\Vmat^{n+1,\top} \\ &\quad- \sum_{j}\Xvec^{n+1,\top}_{i+1/2}\Xvec^{n+1}_{j+1/2}\deltaO(\StefBoltzconst\sol\Temp^{n}_{j+1/2} + \epsi^{2}\meso^{n}_{j+1/2})\bvec^{\top}\Vmat^{n+1}\Vmat^{n+1,\top}  \\&\quad - \sum_{j}\absorpcoeff_{j+1/2}\Xvec^{n+1,\top}_{i+1/2}\Xvec^{n+1}_{j+1/2}\micvec^{n+1}_{j+1/2}\Vmat^{n+1}\Vmat^{n+1,\top}.
        \end{aligned}
    \end{equation}
    Defining the projection matrix onto the spatial basis as $\textbf{P}^{X}\in\R^{\\Nx\times\Nx}$ with entries $$\ProjX_{ij} = \Xvec^{n+1,\top}_{i+1/2}\Xvec^{n+1}_{j+1/2} = \sum_{q}\Xc^{n+1}_{i+1/2,q}\Xc^{n+1}_{j+1/2,q} $$
    and the projection matrix onto the angular basis as $\textbf{P}^{V} = \Vmat^{n+1}\Vmat^{n+1,\top}\in\R^{N\times N}$, \eqref{eq:ESDLRAthm4} reads
    \begin{equation*}\label{eq:ESDLRAthm5}
        \begin{aligned}
            \frac{\epsi^{2}}{\sol\deltat}\micvec^{n+1}_{i+1/2} &= \frac{\epsi^{2}}{\sol\deltat}\widetilde{\micvec}^{n}_{i+1/2} -\epsi\Xvec^{n+1,\top}_{i+1/2}\AdvecOpS\Smatinit^{n}\Vmat^{n+1,\top} \\ &\quad- \sum_{j}\ProjX_{ij}\deltaO(\StefBoltzconst\sol\Temp^{n}_{j+1/2} + \epsi^{2}\meso^{n}_{j+1/2})\bvec^{\top}\textbf{P}^{V}   - \sum_{j}\absorpcoeff_{j+1/2}\ProjX_{ij}\micvec^{n+1}_{j+1/2}\textbf{P}^{V}.
        \end{aligned}
    \end{equation*}
   Thus, if $1\leq k\leq N$, the evolution equation for the $k^{\text{th}}$ moment, $\mic^{n+1}_{i+1/2,k} $, is given by
    \begin{equation}\label{eq:ESDLRAthm6}
        \begin{aligned}
            \frac{\epsi^{2}}{\sol\deltat}\mic^{n+1}_{i+1/2,k} &= \frac{\epsi^{2}}{\sol\deltat}\widetilde{\mic}^{n}_{i+1/2,k} -\epsi\Xvec^{n+1,\top}_{i+1/2}\AdvecOpS\Smatinit^{n}\Vmat^{n+1,\top}\unitvec{k} \\ &\quad- \sum_{j}\ProjX_{ij}\deltaO(\StefBoltzconst\sol\Temp^{n}_{j+1/2} + \epsi^{2}\meso^{n}_{j+1/2})\bvec^{\top}\textbf{P}^{V}\unitvec{k}   - \sum_{j}\absorpcoeff_{j+1/2}\ProjX_{ij}\micvec^{n+1}_{j+1/2}\textbf{P}^{V}\unitvec{k},
        \end{aligned}
    \end{equation}
    where $\unitvec{k} = (\delta_{ik})_{i=1,\ldots,N}$. 
    First, we consider the second term on the right-hand side of \eqref{eq:ESDLRAthm6} and split it into two sub-equations
    \begin{align*}
        \Xvec^{n+1,\top}_{i+1/2}\AdvecOpS\Smatinit^{n}\Vmat^{n+1,\top}\unitvec{k} & = \Xvec^{n+1,\top}_{i+1/2}\left[\sum_{j}\Xvec^{n+1}_{j+1/2}\Diffm\Xvec^{n+1,\top}_{j+1/2}\Smatinit^{n}\Vmat^{n+1,\top}\fluxMatp\Vmat^{n+1}  \right.\\ 
        & \qquad \left.+ \sum_{j}\Xvec^{n+1}_{j+1/2}\Diffp\Xvec^{n+1,\top}_{j+1/2}\Smatinit^{n}\Vmat^{n+1,\top}\fluxMatm\Vmat^{n+1} \right]\Vmat^{n+1,\top}\unitvec{k} \\ 
        &= \sum_{j,\ell,q}\ProjX_{ij}\Diffm\widetilde{\mic}^{n}_{j+1/2,\ell}A^{+}_{\ell q}\ProjV_{qk} +\sum_{j,\ell,q}\ProjX_{ij}\Diffp\widetilde{\mic}^{n}_{j+1/2,\ell}A^{-}_{\ell q}\ProjV_{qk}.
    \end{align*}

    \noindent Similarly, expanding the third and the fourth term on the right-hand side  of \eqref{eq:ESDLRAthm6} we get 
    \begin{equation}\label{eq:ESDLRAthm10}
        \begin{aligned}
            \frac{\epsi^{2}}{\sol\deltat}\mic^{n+1}_{i+1/2,k} &= \frac{\epsi^{2}}{\sol\deltat}\widetilde{\mic}^{n}_{i+1/2,k} - \epsi\sum_{j,\ell,q}\ProjX_{ij}\Diffm\widetilde{\mic}^{n}_{j+1/2,\ell}A^{+}_{\ell q}\ProjV_{qk} - \epsi\sum_{j,\ell,q}\ProjX_{ij}\Diffp\widetilde{\mic}^{n}_{j+1/2,\ell}A^{-}_{\ell q}\ProjV_{qk}\\
            & \qquad -\sum_{j,q}\ProjX_{ij}\deltaO(\StefBoltzconst\sol\Temp^{n}_{j+1/2} + \epsi^{2}\meso^{n}_{j+1/2})b_{q}\ProjV_{qk} - \sum_{j,q}\ProjX_{ij}\absorpcoeff_{j+1/2}\mic^{n+1}_{j+1/2,q}\ProjV_{qk}.
        \end{aligned}
    \end{equation}
   Multiplying \eqref{eq:ESDLRAthm10} by $\mic^{n+1}_{i+1/2,k}\deltax$ and summing over $i$, $k$ we get
       \begin{equation}\label{eq:ESDLRAthm11}
        \begin{aligned}
            \frac{\epsi^{2}}{\sol\deltat}\sum_{i,k}(\mic^{n+1}_{i+1/2,k})^{2}\deltax &= \frac{\epsi^{2}}{\sol\deltat}\sum_{i,k}\widetilde{\mic}^{n}_{i+1/2,k}\mic^{n+1}_{i+1/2,k}\deltax - \epsi\sum_{j,\ell,q}\Diffm\widetilde{\mic}^{n}_{j+1/2,\ell}A^{+}_{\ell q}\sum_{i,k}\ProjV_{qk}\ProjX_{ij}\mic^{n+1}_{i+1/2,k}\deltax\\
            &\qquad- \epsi\sum_{j,\ell,q}\Diffp\widetilde{\mic}^{n}_{j+1/2,\ell}A^{-}_{\ell q}\sum_{i,k}\ProjV_{qk}\ProjX_{ij}\mic^{n+1}_{i+1/2,k}\deltax\\
            & \qquad -\sum_{j,q}\deltaO(\StefBoltzconst\sol\Temp^{n}_{j+1/2} + \epsi^{2}\meso^{n}_{j+1/2})b_{q}\sum_{i,k}\ProjV_{qk}\ProjX_{ij}\mic^{n+1}_{i+1/2,k}\deltax \\
            &\qquad- \sum_{j,q}\absorpcoeff_{j+1/2}\mic^{n+1}_{j+1/2,q}\sum_{i,k}\ProjV_{qk}\ProjX_{ij}\mic^{n+1}_{i+1/2,k}\deltax.
        \end{aligned}
    \end{equation}
    Using 
    \begin{equation}\label{eq:ESDLRAthm12}
        \sum_{i}\ProjX_{ij}\mic^{n+1}_{i+1/2,k} = \mic^{n+1}_{j+1/2,k}, \qquad \sum_{k}\ProjV_{qk}\mic^{n+1}_{j+1/2,k} = \mic^{n+1}_{j+1/2,q}.
    \end{equation}
    and \Cref{lemma:sumofsquares}, \eqref{eq:ESDLRAthm11} reduces to
   \begin{equation*}\label{eq:ESDLRAthm13}
        \begin{aligned}
            \frac{\epsi^{2}}{2\sol\deltat}\left(\norm{\micmat^{n+1}}^{2} - \norm{\widetilde{\micmat}^{n}}^{2} + \norm{\micmat^{n+1} -\widetilde{\micmat}^{n} }^{2} \right) &= - \epsi\sum_{j,\ell,q}\Diffm\widetilde{\mic}^{n}_{j+1/2,\ell}A^{+}_{\ell q}\mic^{n+1}_{j+1/2,q}\deltax\\
            &\qquad- \epsi\sum_{j,\ell,q}\Diffp\widetilde{\mic}^{n}_{j+1/2,\ell}A^{-}_{\ell q}\mic^{n+1}_{j+1/2,q}\deltax\\
            & \qquad -\sum_{j,q}\deltaO(\StefBoltzconst\sol\Temp^{n}_{j+1/2} + \epsi^{2}\meso^{n}_{j+1/2})b_{q}\mic^{n+1}_{j+1/2,q}\deltax \\
            &\qquad- \sum_{j,q}\absorpcoeff_{j+1/2}\widetilde{\mic}^{n+1}_{j+1/2,q}\mic^{n+1}_{j+1/2,q}\deltax.
        \end{aligned}
    \end{equation*}
   Collecting into a vector, we get
    \begin{equation*}\label{eq:ESDLRAthm14}
        \begin{aligned}
            \frac{\epsi^{2}}{2\sol\deltat}\left(\norm{\micmat^{n+1}}^{2} - \norm{\widetilde{\micmat}^{n}}^{2} + \norm{\micmat^{n+1} -\widetilde{\micmat}^{n} }^{2} \right) &= - \epsi\sum_{j}\left(\AdvecOp^{\top}\widetilde{\micvec}^{n}_{j+1/2}\right)\micvec^{n+1,\top}_{j+1/2}\deltax\\
            & \qquad -\sum_{j}\bvec^{\top}\deltaO(\StefBoltzconst\sol\Temp^{n}_{j+1/2} + \epsi^{2}\meso^{n}_{j+1/2})\micvec^{n+1,\top}_{j+1/2}\deltax \\
            &\qquad- \sum_{j}\absorpcoeff_{j+1/2}\micvec^{n+1}_{j+1/2}\micvec^{n+1,\top}_{j+1/2}\deltax.
        \end{aligned}
    \end{equation*}
    The above equation is equivalent to \eqref{eq:Thm1Eq3}. Similarly, substituting the temperature update \eqref{eq:DLRAmacFD} in \eqref{eq:DLRAmesotFD} we get
    \begin{equation}\label{eq:ESDLRAthm15}
        \left( \frac{\StefBoltzconst\Temp^{n+1}_{i} + \frac{\epsi^{2}}{\sol}\meso^{n+1}_{i} - \StefBoltzconst\Temp^{n}_{i} -\frac{\epsi^{2}}{\sol}\meso^{n+1}_{i}}{\deltat}  \right) + \frac{\pknorm{1}}{2}\DiffO\Xvec^{n+1,\top}_{i}\Smat^{n+1}\Vmat^{n+1,\top}\unitvecN = -\absorpcoeff_{i}\meso^{n+1}_{i}.
    \end{equation}
    Multiplying \eqref{eq:ESDLRAthm15} by $\StefBoltzconst\Temp^{n+1}_{i} + \frac{\epsi^{2}}{\sol}\meso^{n+1}_{i}$, summing over $i$ and using \Cref{lemma:sumofsquares} yields
    \begin{equation}\label{eq:ESDLRAthm16}
    \begin{aligned}
         \frac{1}{2\deltat}\left( \norm{\StefBoltzconst\Temp^{n+1} + \frac{\epsi^{2}}{\sol}\meso^{n+1}} - \norm{\StefBoltzconst\Temp^{n} +\frac{\epsi^{2}}{\sol}\meso^{n+1}}\right.&
         \left. + \norm{\StefBoltzconst\Temp^{n+1} + \frac{\epsi^{2}}{\sol}\meso^{n+1} -\StefBoltzconst\Temp^{n} -\frac{\epsi^{2}}{\sol}\meso^{n+1}}  \right) \\
         +  \frac{\pknorm{1}}{2}\sum_{i}(\StefBoltzconst\Temp^{n+1}_{i} + \frac{\epsi^{2}}{\sol}\meso^{n+1}_{i})\DiffO\mic^{n+1}_{i,1}& = -\sum_{i}\absorpcoeff_{i}(\StefBoltzconst\Temp^{n+1}_{i} + \frac{\epsi^{2}}{\sol}\meso^{n+1}_{i})\meso^{n+1}_{i}
    \end{aligned}
    \end{equation}
    which is the same as \eqref{eq:Thm1Eq2}. The rest of the proof follows along the lines of \Cref{theorem:LinMacMicEnStab} (see \Cref{appendix:ProofLinMMMstab}).

\end{proof}

\subsubsection{Local mass conservation}
\begin{theorem}
    The fixed-rank modal macro-micro BUG scheme is locally conservative. I.e., if the scalar flux at time $t_n$ is denoted by $\Phi_i^n = acT_i^n + \varepsilon^2 h_i^n$, where $n\in\{0,1\}$ and $g_{i+1/2,k}^{n+1} = \sum_{\ell,m}X_{i+1/2,\ell}^{n+1}S_{\ell m}^{n+1} V_{km}^{n+1}$ the scheme fulfills the discrete conservation law
\begin{subequations}
    \begin{equation}\label{eq:ConsPhi}
        \frac{\Phi^{n+1}_{i} - \Phi^{n}_{i}}{\deltat} + \sol\frac{\pknorm{1}}{2}\DiffO\mick{1,i}^{n+1} = -\sol\absorpcoeff_{i}\meso^{n+1}_{i} ,
    \end{equation}
    \begin{equation}
        \frac{\specheat}{2}\left(\frac{\Temp^{n+1}_{i} - \Temp^{n}_{i}}{\deltat}\right) = \absorpcoeff_{i}\meso^{n+1}_{i}\, .
    \end{equation}
\end{subequations}    
\end{theorem}
\begin{proof}
    Since, for zero or periodic boundary conditions, $\sum_i \sol\frac{\pknorm{1}}{2}\DiffO\mick{1,i}^{n+1} = 0$, this means that the total mass $\sum_{i}\left(\frac{1}{\sol}\Phi_{i}^{n} + \frac{\specheat}{2}\Temp^{n}_{i} \right)$ is conserved over all time steps $n$. This result is a direct consequence of the macro-micro strategy as shown in \cite{koellermeier2023macromicro} and follows from multiplying \eqref{eq:DLRAmacFD} with $\StefBoltzconst\sol$ and adding \eqref{eq:DLRAmesotFD} for the linearized problem.
\end{proof}

\subsection{Asymptotic-preserving modification to augmented BUG integrator}\label{section:raBUGforMMM}
We see from \Cref{theorem:APBUG} that the updated spatial and angular basis span $\frac{\DPlanckC^{n}_{i+1/2}}{\absorpcoeff_{i+1/2}}\deltaO(\StefBoltzconst\sol\Temp^{n}_{i+1/2})  $ and $\bvec$, respectively. Unlike the fixed-rank BUG integrator \cite{doi:10.1007/s10543-021-00873-0}, naively using the augmented BUG integrator \cite{robustBUG} does not guarantee that this property is fulfilled since the truncation step may prune away essential basis vectors. Thus, we propose the following modification to the augmented BUG integrator, based on its basis-augmentation step and conservative truncation \cite{einkemmer2023conservation}, to obtain an asymptotic-preserving scheme. To ease the presentation of the integrator, we consider the spatially and angularly discretized problem from \Cref{section:BUGforMMM} so that $\micmat \in\R^{\Nx\times N}$. Thus, the low-rank ansatz takes the form 
\begin{equation*}
    \micmat(\vart) = \Xmat(\vart)\Smat(\vart)\Vmat(\vart)^{\top}.    
\end{equation*}
Then, one step of the modal macro-micro BUG scheme updates $\Xmat^{n},\Vmat^{n},\Smat^{n},\meso^{n},\Temp^{n} $ from time $\vart_{n}$ to $\vart_{n+1}$ by the following steps
\begin{enumerate}
    \item Spatial and angular basis update\\

    \indent \textbf{$\Kmat$-step}: Update $\Kmat(\vart_{n+1})\in\R^{\Nx\times\rank} $ according to the K-step \eqref{eq:FDKstep} of the fixed-rank modal macro-micro BUG scheme. Then compute $\hXmat^{n+1}$ as an orthonormal basis of \\$\begin{bmatrix}
        (\absorpcoeffb)^{-1}\DPlanckC^{n}\bdeltaO(\StefBoltzconst\sol\bTemp^{n}) & \Kmat(\vart_{n+1}) & \Xmat^{n}
    \end{bmatrix}$ and store $\hMbug = \hXmat^{n+1,\top}\Xmat^{n} \in \R^{(2\rank+1)\times\rank} $.

    \indent \textbf{$\Lmat$-step}: Update $\Lmat(\vart_{n+1})\in\R^{N\times \rank}$ according to the L-step \eqref{eq:FDLstep} of the fixed-rank modal macro-micro BUG scheme. Then compute $\hVmat^{n+1}$ as an orthonormal basis of $\begin{bmatrix}
        \bvec & \Lmat(\vart_{n+1}) & \Vmat^{n}
    \end{bmatrix}$ and store $\hNbug = \hVmat^{n+1,\top}\Vmat^{n}\in \R^{(2\rank+1)\times\rank} $.

    \item Perform a Galerkin update of the coefficient matrix $\hSmat^{n+1}$ similar to \eqref{eq:raBUGSstep} with the initial condition $\Smatinit^{n} = \hMbug^{\top}\Smat^{n}\hNbug^{\top} $ and the right-hand side as described in \eqref{eq:FDSstep} of the fixed-rank modal macro-micro BUG scheme.

    \item Asymptotic-preserving splitting of the basis matrices and truncation \\

    Set $\widehat{\Kmat} = \hXmat^{n+1}\hSmat^{n+1}$ and split $\widehat{\Kmat} = \begin{bmatrix}
        \widehat{\Kmat}^{\text{ap}} & \widehat{\Kmat}^{\text{rem}}
    \end{bmatrix}$ into basis vectors, where \\$\widehat{\Kmat}^{\text{ap}} = \begin{bmatrix}
        (\absorpcoeffb)^{-1}\DPlanckC^{n}\bdeltaO(\StefBoltzconst\sol\bTemp^{n})
    \end{bmatrix}\in\R^{\Nx\times 1} $ and the remaining basis vectors into $\widehat{\Kmat}^{\text{rem}}\in\R^{\Nx\times 2\rank} $. Similarly, split the angular basis $\hVmat = \begin{bmatrix}
        \hVmat^{\text{ap}} & \hVmat^{\text{rem}}
    \end{bmatrix}$, where $\hVmat^{\text{ap}} = \begin{bmatrix}
        \bvec        
    \end{bmatrix} \in\R^{N\times 1} $ and $\hVmat^{\text{rem}}\in\R^{N\times 2\rank} $.

    Next compute the QR decomposition of $\widehat{\Kmat}^{\text{rem}} $
    \begin{equation*}\label{eq:raBUGeq1}
        \widehat{\Kmat}^{\text{rem}} = \widehat{\Xmat}^{\text{rem}}\widehat{\Smat}^{\text{rem}}.
    \end{equation*}
    and for truncating the rank compute the SVD of $ \widehat{\Smat}^{\text{rem}} $ as
    \begin{equation}\label{eq:raBUGeq2}
        \widehat{\Smat}^{\text{rem}} = \textbf{U}\BSigma\textbf{W}^{\top}.
    \end{equation}
    We truncate the remaining basis vectors to  $1\leq\rank^{*}\leq 2\rank$ such that if $\hat{\sigma}_{i}$, $ i = 1,\ldots,2\rank$, are the singular values of $\widehat{\Smat}^{\text{rem}}$ then for some user defined tolerance $\vartheta$ the following is satisfied:
    \begin{equation*}
        \left(\sum_{i=\rank^{*} + 1}^{2\rank}\hat{\sigma}_{i}\right)^{1/2} \leq \vartheta.
    \end{equation*}
    The new rank is set as $\Nrank = \rank^{*} + 1 $. Then let $ \widehat{\textbf{U}}\in\R^{2\rank\times\rank^{*}} $ and $ \widehat{\textbf{W}}\in\R^{2\rank\times\rank^{*}} $ be the matrices containing the first $\rank^{*}$ columns of $\textbf{U}$ and $\textbf{W}$, respectively. Similarly, let $\widehat{\BSigma} $ be the first $\rank^{*}\times\rank^{*}$ block of $\BSigma$; then we set 
    \begin{equation*}\label{eq:raBUGeq3}
        \Xmat^{\text{rem}} =  \widehat{\Xmat}^{\text{rem}}\widehat{\textbf{U}}, \qquad \Smat^{\text{rem}} = \widehat{\BSigma}, \qquad \textbf{W}^{n+1} = \hVmat^{\text{rem}}\widehat{\textbf{W}}.
    \end{equation*}

    Then we get the updated angular basis $\Vmat^{n+1}\in\R^{N\times\Nrank} $ by adding columns, i.e. \begin{equation*}
        \Vmat^{n+1} = \begin{bmatrix}
        \hVmat^{\text{ap}} & \textbf{W}^{n+1}
    \end{bmatrix}.
    \end{equation*}\label{eq:raBUGeq4}
    For the updated spatial basis, we first compute the QR decomposition of $\widehat{\Kmat}^{\text{ap}} $ as 
    \begin{equation*}\label{eq:raBUGeq5}
        \widehat{\Kmat}^{\text{ap}} = \Xmat^{\text{ap}}\Smat^{\text{ap}}.
    \end{equation*}
    Then set $\widehat{\Xmat} = \begin{bmatrix}
        \Xmat^{\text{ap}} &  \Xmat^{\text{rem}}
    \end{bmatrix}$ and subsequently perform a QR decomposition to obtain the updated spatial basis matrix $\Xmat^{n+1}\in\R^{\Nx\times\Nrank}$,
    \begin{equation}\label{eq:raBUGeq6}
        \Xmat^{n+1}\textbf{R}_{2} = \widehat{\Xmat}.
    \end{equation}

    Finally we set the updated coefficient matrix $\Smat^{n+1}\in\R^{\Nrank\times\Nrank}$ to be
    \begin{equation}\label{eq:raBUGeq7}
        \Smat^{n+1} = \textbf{R}_{2}\begin{bmatrix}
            \Smat^{\text{ap}} & \textbf{0}\\
            \textbf{0} & \Smat^{\text{rem}}
        \end{bmatrix}
    \end{equation}
    and the approximation at the next time step is set as $\micmat^{n+1} = \Xmat^{n+1}\Smat^{n+1}\Vmat^{n+1,\top} $.
    \item Update $\Temp,\meso$:
    \begin{equation}\label{eq:raBUGeq8}
        \frac{\epsi^{2}}{\sol}\left( \frac{\meso^{n+1}_{i} - \meso^{n}_{i}}{\deltat} \right) + \StefBoltzconst\facalph\hspace{.5mm}\absorpcoeff_{i}\DPlanckC^{n}_{i}\meso^{n+1}_{i} + \frac{\pknorm{1}}{2}\DiffO\Xvec^{n+1,\top}_{i+1/2}\Smat^{n+1}\Vmat^{n+1,\top}\unitvecN = -\absorpcoeff_{i}\meso^{n+1}_{i},
    \end{equation}
    \begin{equation}\label{eq:aBUGeq9}
        \frac{\Temp^{n+1}_{i} - \Temp^{n}_{i}}{\deltat} = \facalph\hspace{.5mm}\absorpcoeff_{i}\meso^{n+1}_{i}.
    \end{equation}
\end{enumerate}

\begin{lemma}\label{lemma:APraBUGAux1}
    For the proposed modal macro-micro BUG scheme, we have 
    \begin{equation*}
        \textbf{R}_{2}^{-\top}\begin{bmatrix}
           (\Smat^{\text{ap}})^{-\top} & \textbf{0}\\
           \textbf{0} & \widehat{\textbf{U}}^{\top}(\widehat{\Smat}^{\text{rem}})^{-\top}\
       \end{bmatrix}\hSmat^{n+1,\top}\hSmat^{n+1}\begin{bmatrix}
           \textbf{I}_{m} & \textbf{0}\\
           \textbf{0} & \widehat{\textbf{W}}
       \end{bmatrix} = \Smat^{n+1},
    \end{equation*}
    where the matrices are as defined above and $\textbf{I}_{m}$ is the $m\times m$ identity matrix.
\end{lemma}
\begin{proof}
    See \Cref{appendix:ProofAPraBUGAux1}.
\end{proof}

\begin{theorem}\label{theorem:APraBUG}
    The proposed modal macro-micro BUG scheme is asymptotic-preserving in the sense of \Cref{theorem:APBUG}.
\end{theorem}
\begin{proof}
   From the K- and L-step of the modal macro-micro BUG scheme we get $(\absorpcoeffb)^{-1}\DPlanckC^{n}\bdeltaO(\StefBoltzconst\sol\bTemp^{n})\in\text{Range }(\hXmat^{n+1}) $ and $\bvec\in\text{Range }(\hVmat^{n+1}) $. Thus, along the lines of \Cref{theorem:APBUG} for $\epsi\to 0$ we get from the $S$-step of the modal macro-micro BUG scheme 
   \begin{equation}\label{eq:APraBUGeq1}
       -\left(\hXmat^{n+1,\top}(\absorpcoeffb)^{-1}\DPlanckC^{n}\bdeltaO(\StefBoltzconst\sol\bTemp^{n})\right)\left(\bvec^{\top}\hVmat^{n+1}\right) = \hSmat^{n+1}.
   \end{equation}
   Now, to show that the proposed scheme is asymptotic-preserving, we need to show two things; first, that $(\absorpcoeffb)^{-1}\DPlanckC^{n}\bdeltaO(\StefBoltzconst\sol\bTemp^{n})\in\text{Range}(\Xmat^{n+1}) $ and $\bvec\in\text{Range}(\Vmat^{n+1}) $. Second, we need to show that the above relation \eqref{eq:APraBUGeq1} also holds for the truncated factor matrices $\Xmat^{n+1},\Vmat^{n+1}$ and $\Smat^{n+1}$. The first property follows directly from the construction of the scheme. For the latter, we can represent $\Xmat^{n+1}$ as 
   \begin{align*}
       \Xmat^{n+1} &= \begin{bmatrix}
           \Xmat^{\text{ap}} & \Xmat^{\text{rem}}
       \end{bmatrix}\textbf{R}_{2}^{-1}\\
       &= \begin{bmatrix}
           \widehat{\Kmat}^{\text{ap}}(\Smat^{\text{ap}})^{-1} & \widehat{\Xmat}^{\text{rem}}\widehat{\textbf{U}}
       \end{bmatrix}\textbf{R}_{2}^{-1}\\
        &= \begin{bmatrix}
           \widehat{\Kmat}^{\text{ap}}(\Smat^{\text{ap}})^{-1} & \widehat{\Kmat}^{\text{rem}}(\widehat{\Smat}^{\text{rem}})^{-1}\widehat{\textbf{U}}
       \end{bmatrix}\textbf{R}_{2}^{-1}\\
       &= \begin{bmatrix}
           \widehat{\Kmat}^{\text{ap}} & \widehat{\Kmat}^{\text{rem}}
       \end{bmatrix}\begin{bmatrix}
           (\Smat^{\text{ap}})^{-1} & \textbf{0}\\
           \textbf{0} & (\widehat{\Smat}^{\text{rem}})^{-1}\widehat{\textbf{U}}
       \end{bmatrix}\textbf{R}_{2}^{-1}\\
       & = \hXmat^{n+1}\hSmat^{n+1}\begin{bmatrix}
           (\Smat^{\text{ap}})^{-1} & \textbf{0}\\
           \textbf{0} & (\widehat{\Smat}^{\text{rem}})^{-1}\widehat{\textbf{U}}
       \end{bmatrix}\textbf{R}_{2}^{-1}.
   \end{align*}
   Thus we get
   \begin{equation*}\label{eq:APraBUGeq2}
       \Xmat^{n+1,\top} = \textbf{R}_{2}^{-\top}\begin{bmatrix}
           (\Smat^{\text{ap}})^{-\top} & \textbf{0}\\
           \textbf{0} & \widehat{\textbf{U}}^{\top}(\widehat{\Smat}^{\text{rem}})^{-\top}
       \end{bmatrix}\hSmat^{n+1,\top}\hXmat^{n+1,\top}
   \end{equation*}
   and similarly, we get the following relation for the updated angular basis
   \begin{align*}
       \Vmat^{n+1} &= \begin{bmatrix}
           \hVmat^{\text{ap}} & \textbf{W}^{n+1}
       \end{bmatrix}\\
        &= \begin{bmatrix}
           \hVmat^{\text{ap}} & \hVmat^{\text{rem}}
       \end{bmatrix}\begin{bmatrix}
           \textbf{I}_{m} & \textbf{0}\\
           \textbf{0} & \widehat{\textbf{W}}
       \end{bmatrix}.
   \end{align*}
   This yields the relation
   \begin{equation*}\label{eq:APraBUGeq3}
       \Vmat^{n+1} = \hVmat^{n+1}\begin{bmatrix}
           \textbf{I}_{m} & \textbf{0}\\
           \textbf{0} & \widehat{\textbf{W}}
       \end{bmatrix},
   \end{equation*}
   where $\textbf{I}_{m}$ is the $m\times m$ identity matrix.
    Now we multiply \eqref{eq:APraBUGeq1} by $\textbf{R}_{2}^{-\top}\begin{bmatrix}
           (\Smat^{\text{ap}})^{-\top} & \textbf{0}\\
           \textbf{0} & \widehat{\textbf{U}}^{\top}(\widehat{\Smat}^{\text{rem}})^{-\top}
       \end{bmatrix}\hSmat^{n+1,\top} $ from the left and by $\begin{bmatrix}
           \textbf{I}_{m} & \textbf{0}\\
           \textbf{0} & \widehat{\textbf{W}}
       \end{bmatrix} $ from the right and using \Cref{lemma:APraBUGAux1} for the right-hand side gives 
    \begin{equation*}\label{eq:APraBUGeq4}
           -\left(\Xmat^{n+1,\top}(\absorpcoeffb)^{-1}\DPlanckC^{n}\bdeltaO(\StefBoltzconst\sol\bTemp^{n})\right)\left(\bvec^{\top}\Vmat^{n+1}\right) = \Smat^{n+1}.
       \end{equation*}
    Thus by multiplying by $\Xmat^{n+1}$ and $\Vmat^{n+1}$ from the left and the right and using the fact that $(\absorpcoeffb)^{-1}\DPlanckC^{n}\bdeltaO(\StefBoltzconst\sol\bTemp^{n})\in\text{Range}(\Xmat^{n+1}) $ and $\bvec\in\text{Range}(\Vmat^{n+1}) $ we can show that the proposed scheme is asymptotic-preserving by following the steps from \Cref{theorem:APBUG}. 
\end{proof}
\subsubsection{Energy stability}
\begin{theorem}\label{theorem:raBUGMacMicEnStab}
    Assume that the CFL condition \eqref{eq:CFLCond} from \Cref{theorem:LinMacMicEnStab} holds. Then the modal macro-micro BUG scheme is energy stable for the linearized problem \eqref{eq:LinMacMicsys}, where the energy is the same as defined in \Cref{theorem:DLRALinMacMicEnStab}. 
\end{theorem}
\begin{proof}
  Similar to \Cref{theorem:DLRALinMacMicEnStab} we start by considering the $S$-step of the modal macro-micro BUG scheme where,
   \begin{equation}\label{eq:ESraBUGthm1}
        \begin{aligned}
            \frac{\epsi^{2}}{\sol\deltat}\hSmat^{n+1} &=   \frac{\epsi^{2}}{\sol\deltat}\Smatinit^{n} -\epsi\AdvecOpS\Smatinit^{n} - \sum_{j}\hXvec^{n+1}_{j+1/2}\deltaO(\StefBoltzconst\sol\Temp^{n}_{j+1/2} + \epsi^{2}\meso^{n}_{j+1/2})\bvec^{\top}\hVmat^{n+1} \\ & \qquad - \sum_{j}\absorpcoeff_{j+1/2}\hXvec^{n+1}_{j+1/2}\hXvec^{n+1,\top}_{j+1/2}\hSmat^{n+1}.
        \end{aligned}
    \end{equation}
    We note that $\sum_{i}\hXvec^{n+1,\top}_{i+1/2}\Smatinit^{n}\hVmat^{n+1,\top} = \sum_{i}\Xvec^{n,\top}_{i+1/2}\Smat^{n}\Vmat^{n,\top} $. Multiplying \eqref{eq:ESraBUGthm1} by $\hXvec^{n+1,\top}_{i+1/2}$ from the left and $\hVmat^{n+1,\top} $ from the right, then summing over $i$ we get
    \begin{equation*}\label{eq:ESraBUGthm2}
        \begin{aligned}
             \frac{\epsi^{2}}{\sol\deltat}\sum_{i}\hXvec^{n+1,\top}_{i+1/2}\hSmat^{n+1}\hVmat^{n+1,\top}\deltax &= \frac{\epsi^{2}}{\sol\deltat}\sum_{i}\Xvec^{n,\top}_{i+1/2}\Smat^{n}\Vmat^{n,\top}\deltax - \epsi\sum_{i}\hXvec^{n+1,\top}_{i+1/2}\AdvecOpS\Smatinit\hVmat^{n+1,\top}\deltax \\&- \sum_{i,j}\hXvec^{n+1,\top}_{i+1/2}\hXvec^{n+1}_{j+1/2}\deltaO(\StefBoltzconst\sol\Temp^{n}_{j+1/2} + \epsi^{2}\meso^{n}_{j+1/2})\bvec^{\top}\hVmat^{n+1}\hVmat^{n+1} \deltax\\
             & \qquad - \sum_{i,j}\absorpcoeff_{j+1/2}\hXvec^{n+1,\top}_{i+1/2}\hXvec^{n+1}_{j+1/2}\hXvec^{n+1,\top}_{j+1/2}\hSmat^{n+1}\hVmat^{n+1} \deltax.
        \end{aligned}
    \end{equation*}
    Since the truncation step of the integrator does not increase the norm of the solution
    \begin{equation*}
        \norm{\Xmat^{n+1}\Smat^{n+1}\Vmat^{n+1,\top}} = \norm{\Smat^{n+1}} \leq \norm{\hSmat^{n+1}} = \norm{\hXmat^{n+1}\hSmat^{n+1}\hVmat^{n+1,\top}}.
    \end{equation*}
    The rest of the proof follows along the lines of \Cref{theorem:DLRALinMacMicEnStab}.
\end{proof}


\section{Numerical results} \label{section:Numericalresults}
The following numerical results can be reproduced with the openly available source code \cite{publication_codes}.
\subsection{Rectangular pulse test case}
We consider gray thermal radiative transfer equations in slab geometry \eqref{eq:RTE} on the spatial domain $D = [-10, 10]$. The initial distribution of the temperature is given by the rectangular pulse 
\begin{equation*}
    \Temp(\vart = 0,\varx) = \frac{100}{\absorpcoeff(\varx)}\cdot\chi_{[-0.5,0.5]}(\varx).
\end{equation*}
The particle density is initially at an equilibrium with the temperature and is given by
\begin{equation*}
    \parden(\vart = 0,\varx,\omg) = \StefBoltzconst\sol\Temp(\vart = 0,\varx).
\end{equation*}
Subsequently, as time progresses, the particles move into all directions $\omg\in[-1,1]$ while undergoing isotropic absorption at the rate $\absorpcoeff(\varx) = 0.5$. In this test case, we assume that no particles are present at the boundary during the entire simulation, and the temperature remains zero at the boundary as well. The initial and boundary conditions for $\mic$ and $\meso$ can be derived from those for temperature and particle density by using the relations \eqref{eq:ICBC}. We assume that constants are scaled to $1$, i.e., we set the radiation constant $\StefBoltzconst = 1$, speed of light $\sol = 1$, and the specific heat is $\specheat = 1$. The mass at time $\vart_{n}$ is defined as
\begin{equation}
    m^{n} = \sum_{i}\left(\StefBoltzconst\Temp^{n}_{i} + \frac{\epsi^{2}}{\sol}\meso^{n}_{i} + \frac{\specheat}{2}\Temp^{n}_{i} \right)\deltax.
\end{equation}

The spatial domain is divided into $\Nx = 501$ spatial cells, and we use $N = 100$ Legendre polynomials to represent the angular variable. The moment solutions (P$_{N}$) are computed using the modal macro-micro scheme \eqref{eq:LinMacMicsys}. For $\epsi = 1$, we choose a rank of $\rank = 5,15$ for the fixed-rank modal macro-micro BUG integrator (frBUG) and an initial rank of $\rank = 1$ for the modal macro-micro BUG integrator (BUG). For rank truncation we set the tolerance parameter to $\vartheta = 5\cdot10^{-2}\norm{\BSigma}_{2}$ and the end time is set to $\vart_{\text{end}} = 1.5$. The step size is chosen according to \eqref{eq:CFLCond} 
\begin{equation*}
    \deltat =  \underset{k}{\text{min}} \left\{  \frac{1}{5\sol\beta_{N}}\left( \frac{2\epsi\deltax}{\abs{\quadpt_{k}}} + \frac{\absorpcoefflb\deltax^{2}}{\quadpt_{k}^{2}} \right) \right\}
\end{equation*}
where, for $N = 100$ the step size is minimal for $\quadpt_{k} = -0.999719$ which gives the step size $\deltat \approx 0.005$. We compare the low-rank approximations with the moment solutions P$_{5}$, P$_{15}$ and P$_{100}$, and the results are given in \Cref{fig:Rectangular_pulse_1D_Kin}. The relative mass error of all the low-rank solutions as well as the P$_{100}$ solution is given in \Cref{fig:Rectangular_pulse_1D_mass_kin}. Note that the chosen step size for the P$_{5}$ and P$_{15}$ solutions differs from that of the rest of the solutions and is minimal for a different quadrature point, which we do not specify here. From the plots of temperature and scalar flux, we see that the solution of the fixed-rank modal macro-micro BUG integrator with $\rank = 15$ (BUG$_{15}$) and the modal macro-micro BUG integrator agree well with the full moment solution P$_{100}$ and are different from the Rosseland diffusion limit. Additionally, the BUG$_{5}$ approximation performs much better than the P$_{5}$ solution. All the integrators dissipate energy over time, as we see from \Cref{fig:Rectangular_pulse_1D_Diff}. 

\begin{figure}[H]
    \centering
    \begin{subfigure}[b]{0.49\linewidth}
        \includegraphics[width=0.9\linewidth]{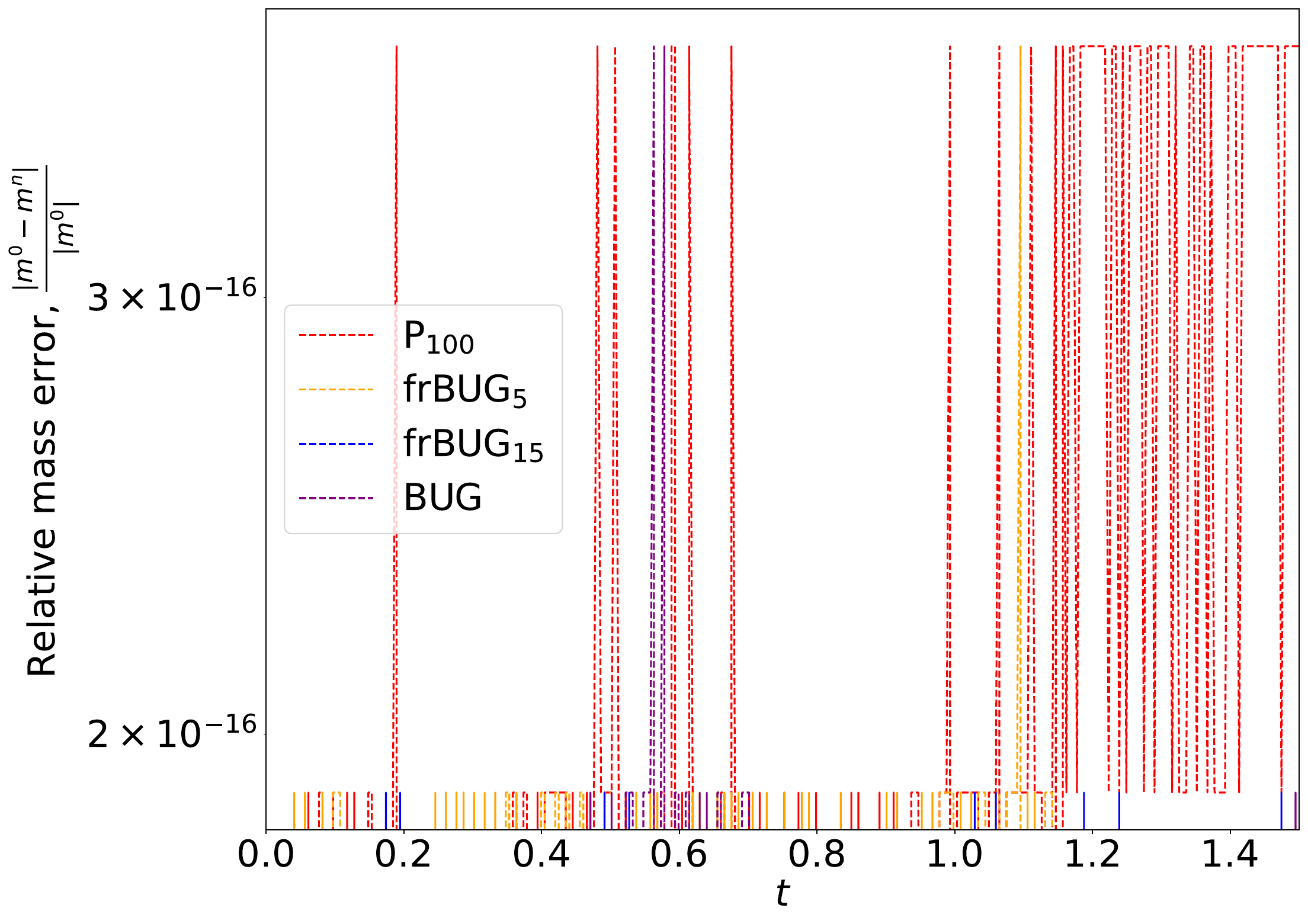}
        \caption{$\epsi = 1.0$}
         \label{fig:Rectangular_pulse_1D_mass_kin}
    \end{subfigure}
    \begin{subfigure}[b]{0.49\linewidth}
        \includegraphics[width=0.85\linewidth]{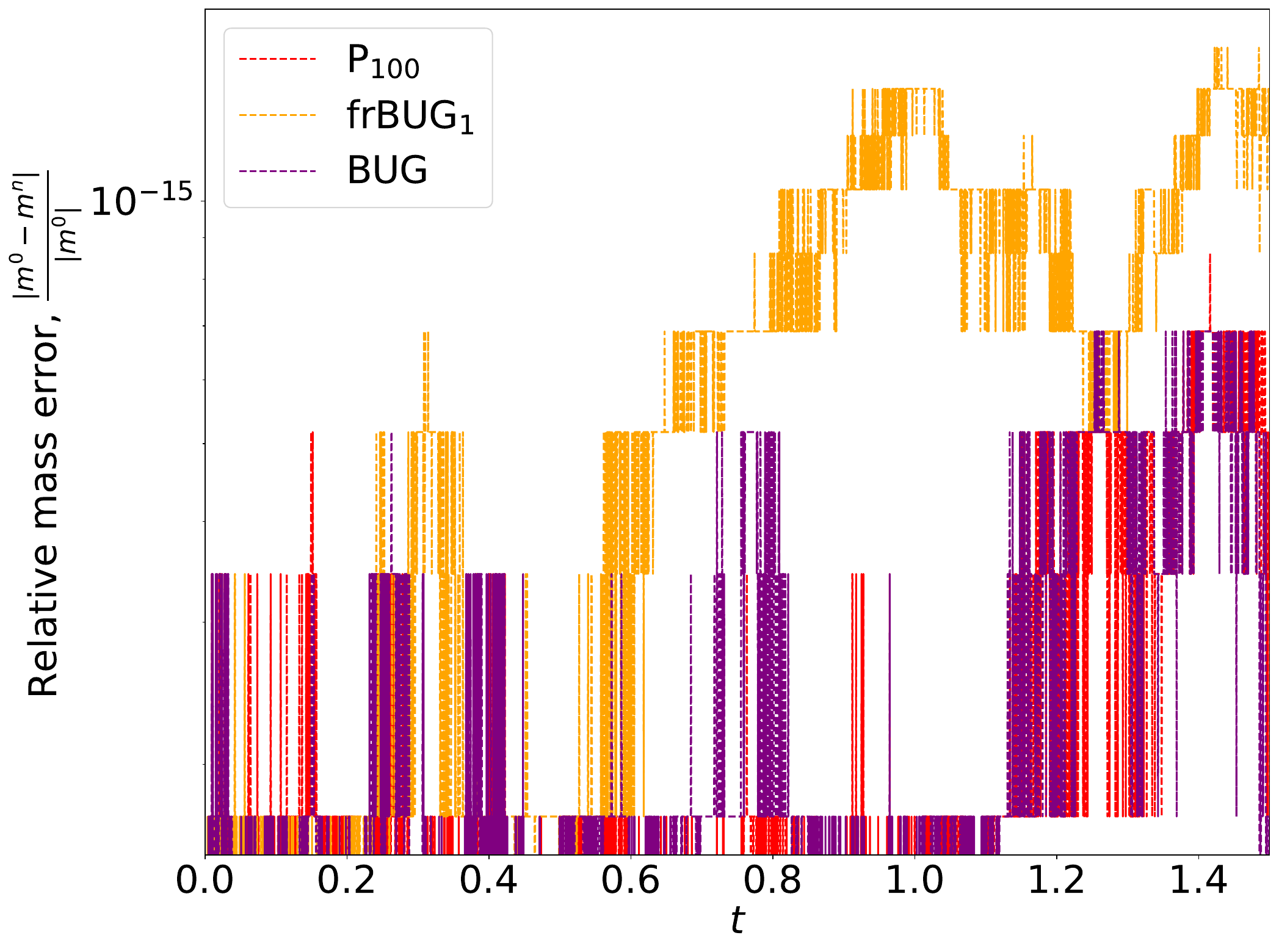}
        \caption{$\epsi = 10^{-5}$}
        \label{fig:Rectangular_pulse_1D_mass_dif}
    \end{subfigure}
    \caption{Relative mass error for the rectangular pulse test case in the kinetic and diffusive regime.}
    \label{fig:Rectangular_pulse_1D_mass}
\end{figure}

\begin{figure}[H]
    \centering
    \begin{subfigure}[b]{0.49\linewidth}
        \includegraphics[width=0.9\linewidth]{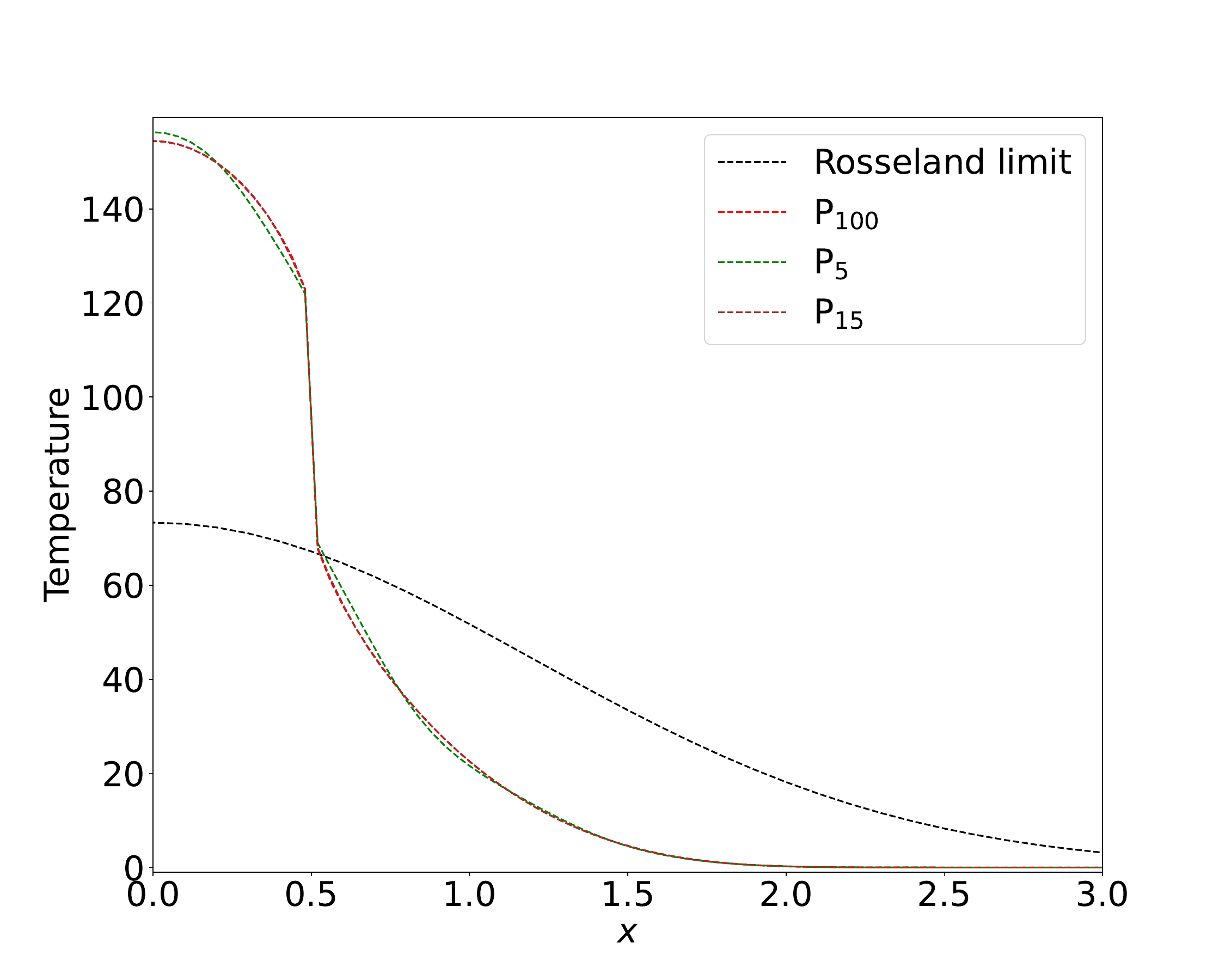}
    \end{subfigure}
    \begin{subfigure}[b]{0.49\linewidth}
        \includegraphics[width=0.9\linewidth]{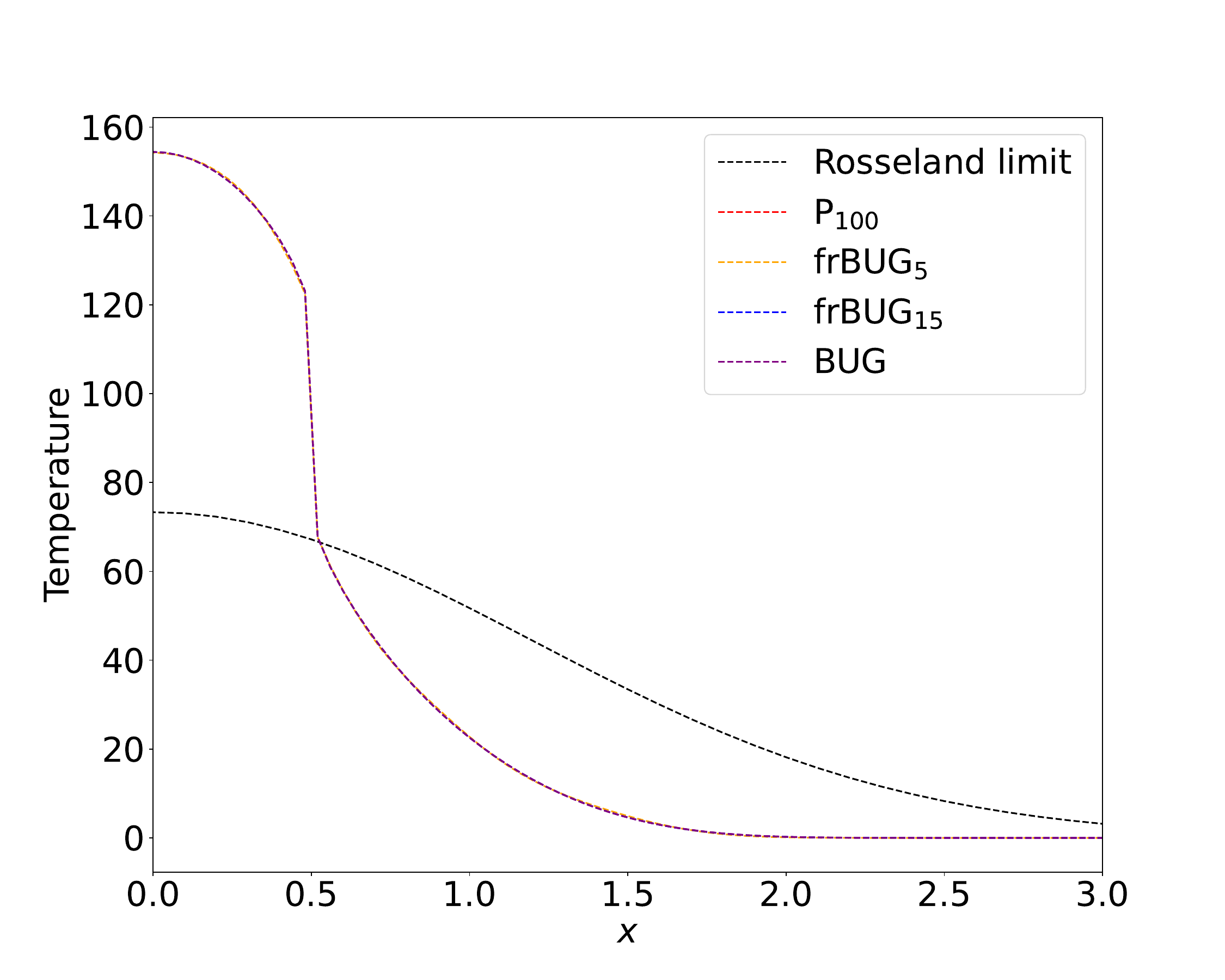}
    \end{subfigure}

    \begin{subfigure}[b]{0.49\linewidth}
        \includegraphics[width=0.9\linewidth]{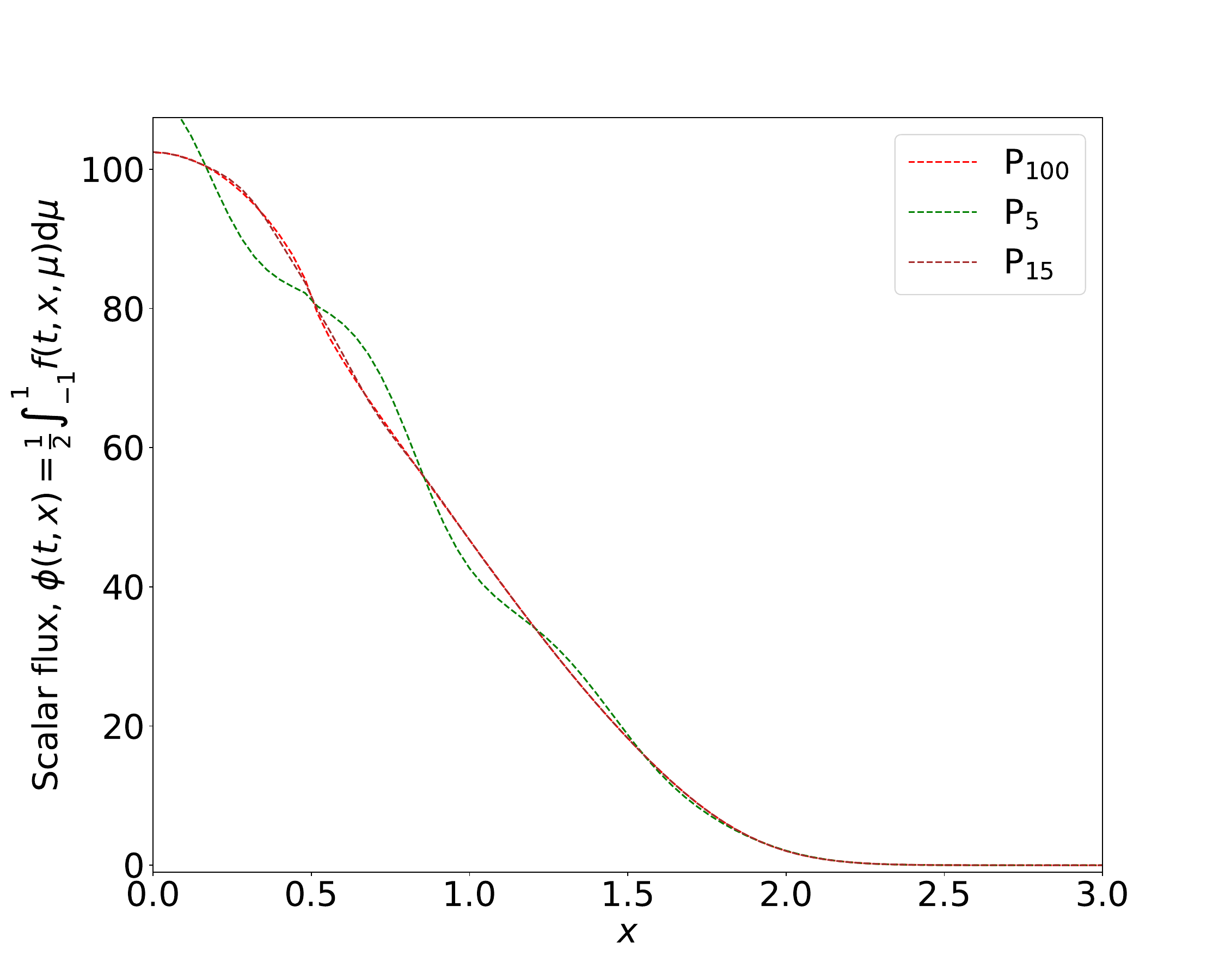}
        \caption{moment methods}
    \end{subfigure}
    \begin{subfigure}[b]{0.49\linewidth}
        \includegraphics[width=0.9\linewidth]{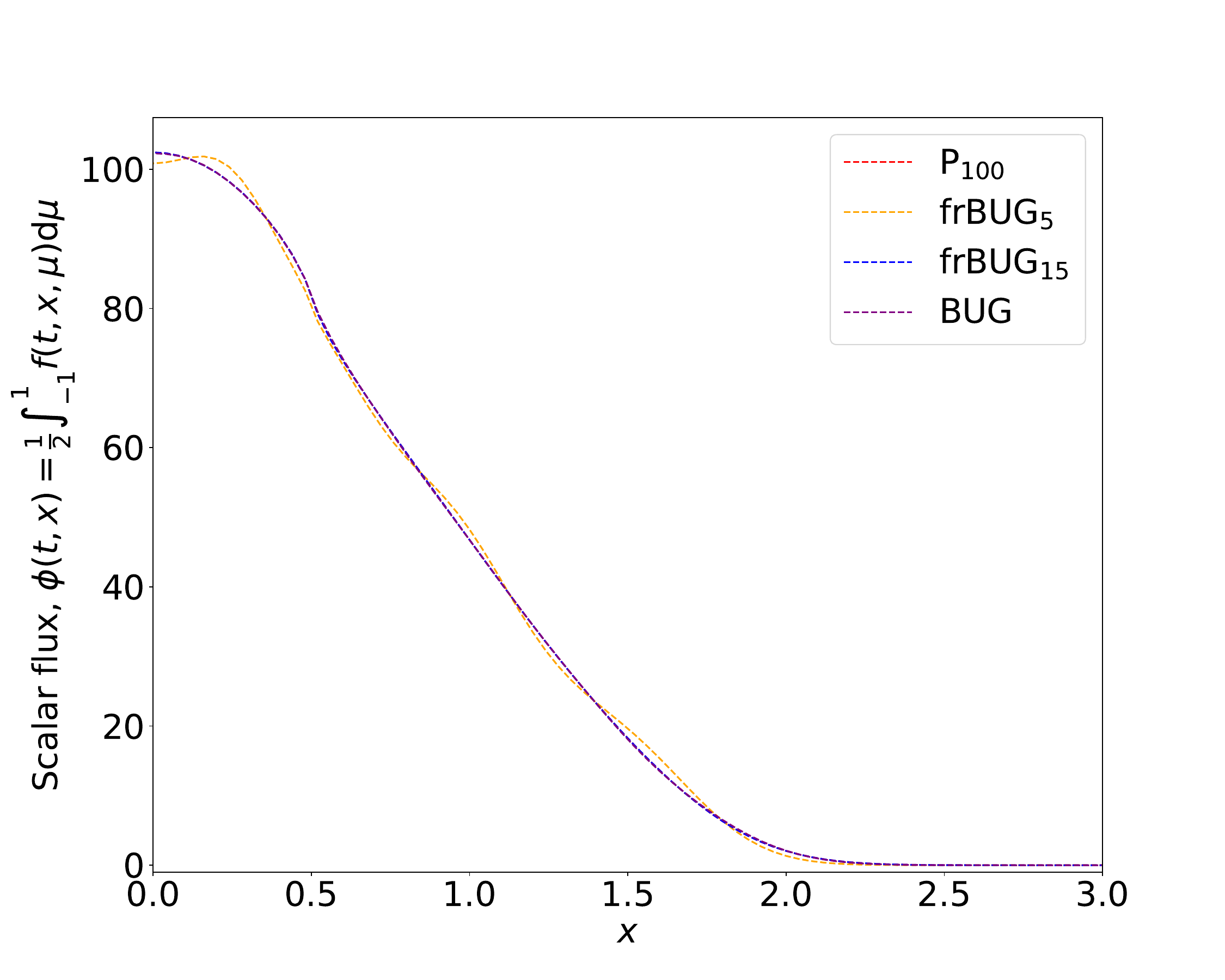}
        \caption{low-rank methods}
    \end{subfigure}

    \begin{subfigure}[b]{0.49\linewidth}
        \includegraphics[width=0.85\linewidth]{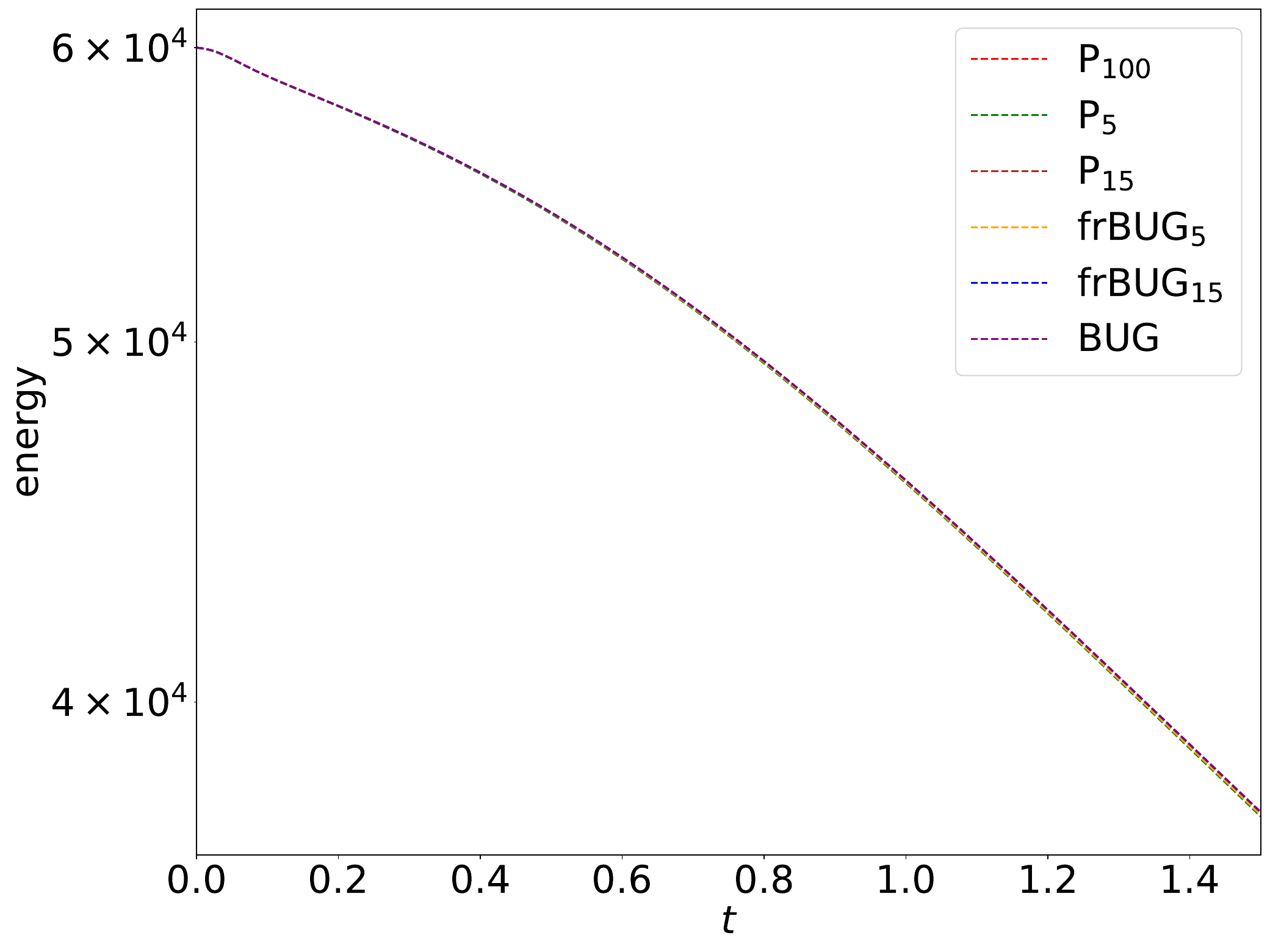}
        \caption{energy over time}
    \end{subfigure}
    \begin{subfigure}[b]{0.49\linewidth}
        \includegraphics[width=0.9\linewidth]{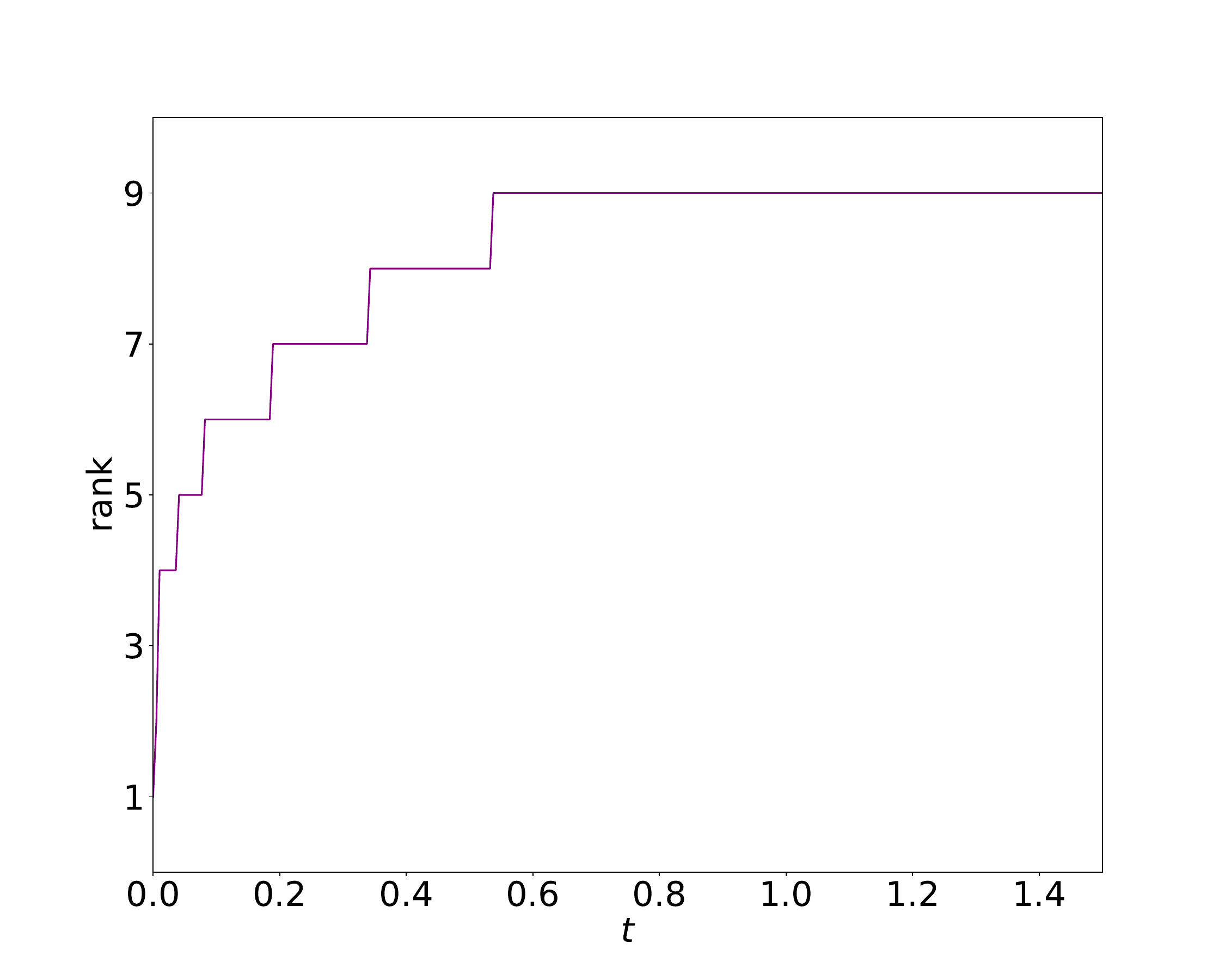}
        \caption{rank over time BUG}
    \end{subfigure}
    \caption{Numerical results of the rectangular pulse test case in the kinetic regime, i.e., $\epsi = 1$ at $\vart = 1.5$. In the first row, we present the temperature profile at end-time for the moment and low-rank methods; in the second row, we have the corresponding scalar flux. In the last row, we have the energy of the system over time for all the methods and the rank evolution of the BUG integrator.}
    \label{fig:Rectangular_pulse_1D_Kin}
\end{figure}

For $\epsi = 10^{-5}$, we use a coarser spatial grid with $\Nx = 201$ cells. The rank used for the fixed-rank modal macro-micro BUG integrator is $\rank = 1$, and the modal macro-micro BUG integrator starts with the same initial rank $\rank = 1$. The tolerance parameter and end time are the same as in the kinetic regime. We see from \Cref{fig:Rectangular_pulse_1D_Diff,fig:Rectangular_pulse_1D_mass} that the solutions from the full modal macro-micro integrator, fixed-rank modal macro-micro BUG integrator, and modal macro-micro BUG integrator agree well with the limiting Rosseland approximation. Additionally, all the methods dissipate energy over time. 

\begin{figure}[H]
    \centering
    \begin{subfigure}[b]{0.49\linewidth}
        \includegraphics[width=0.9\linewidth]{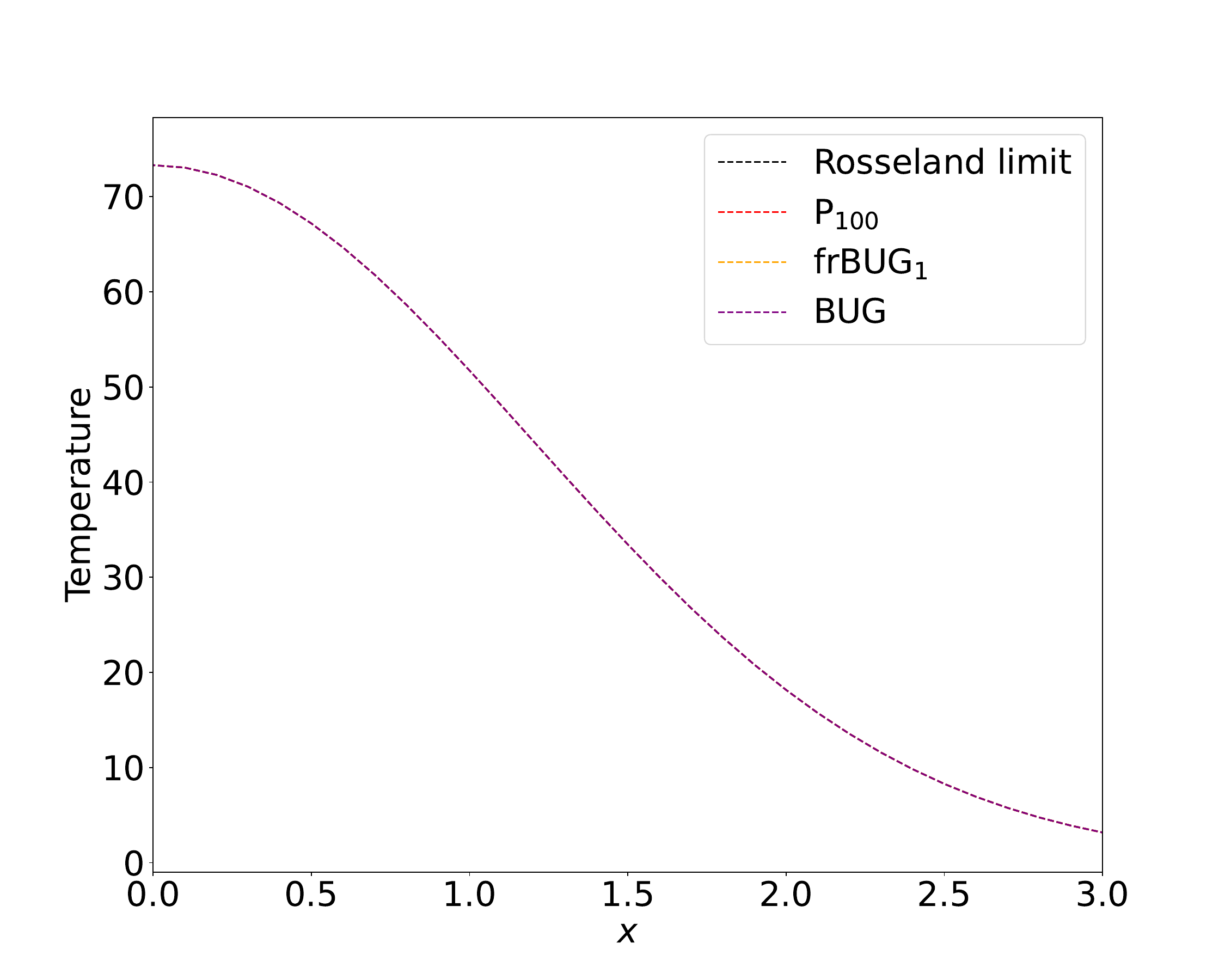}
    \end{subfigure}
    \begin{subfigure}[b]{0.49\linewidth}
        \includegraphics[width=0.9\linewidth]{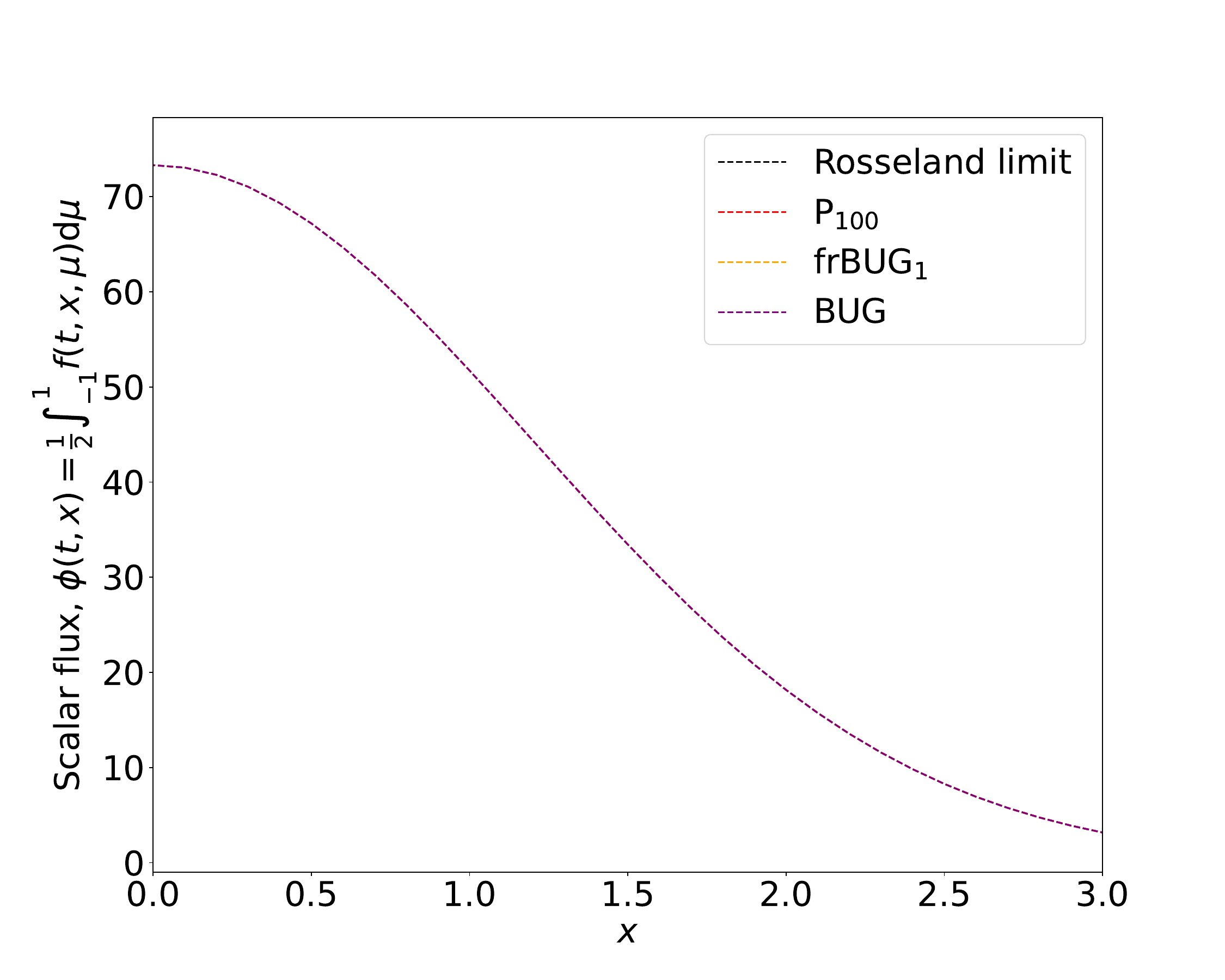}
    \end{subfigure}

    \begin{subfigure}[b]{0.49\linewidth}
        \includegraphics[width=0.9\linewidth]{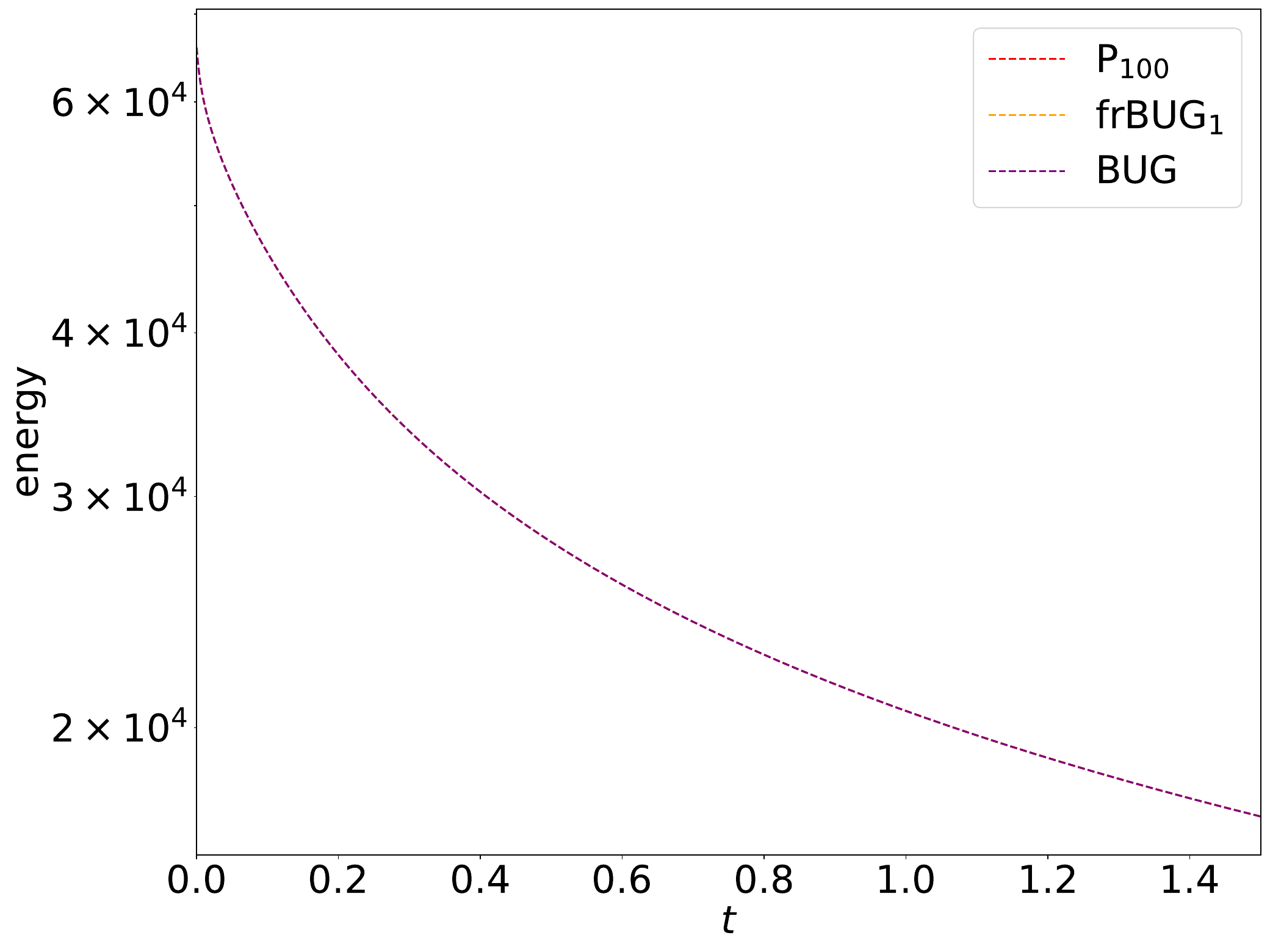}
    \end{subfigure}
    \begin{subfigure}[b]{0.49\linewidth}
        \includegraphics[width=0.9\linewidth]{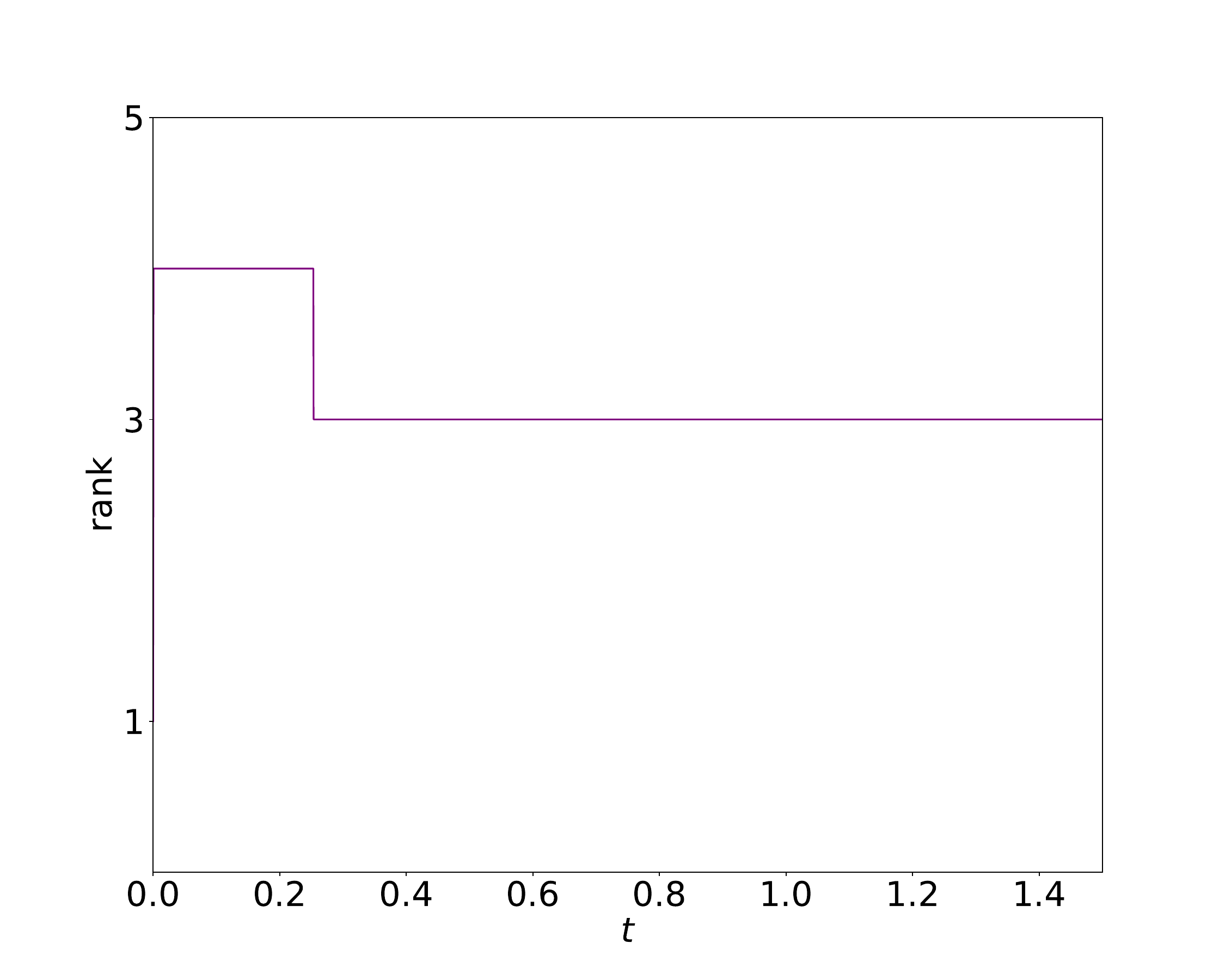}
    \end{subfigure}
    \caption{Numerical results of the rectangular pulse test case in the diffusive regime, i.e. $\epsi = 10^{-5}$ at $\vart = 1.5$. Top left: Temperature profile, Top right: Scalar flux, Bottom left: Energy of the system over time for all the methods, Bottom right: Rank evolution of rank-adaptive integrator over time.}
    \label{fig:Rectangular_pulse_1D_Diff}
\end{figure}

\subsection{Absorber test case}
To study the behavior of the methods in an inhomogeneous medium, we place an absorber in the middle of the domain. That is, we set the absorption coefficient to
\begin{equation*}
    \absorpcoeff(\varx) = \begin{cases}
        5, & \mathrm{if } -0.25 \leq \varx \leq 0.25,\\
        0.5 & \mathrm{else}
    \end{cases}.
\end{equation*}
The remaining parameters, along with the end time, are the same as in the rectangular pulse test case. The temperature and scalar flux, along with other parameters, are depicted in \Cref{fig:Absorber_1D_Kin} for $\epsi = 1$ and in \Cref{fig:Absorber_1D_Diff} for $\epsi = 10^{-5}$.

\begin{figure}[H]
    \centering
    \begin{subfigure}[b]{0.49\linewidth}
        \includegraphics[width=0.9\linewidth]{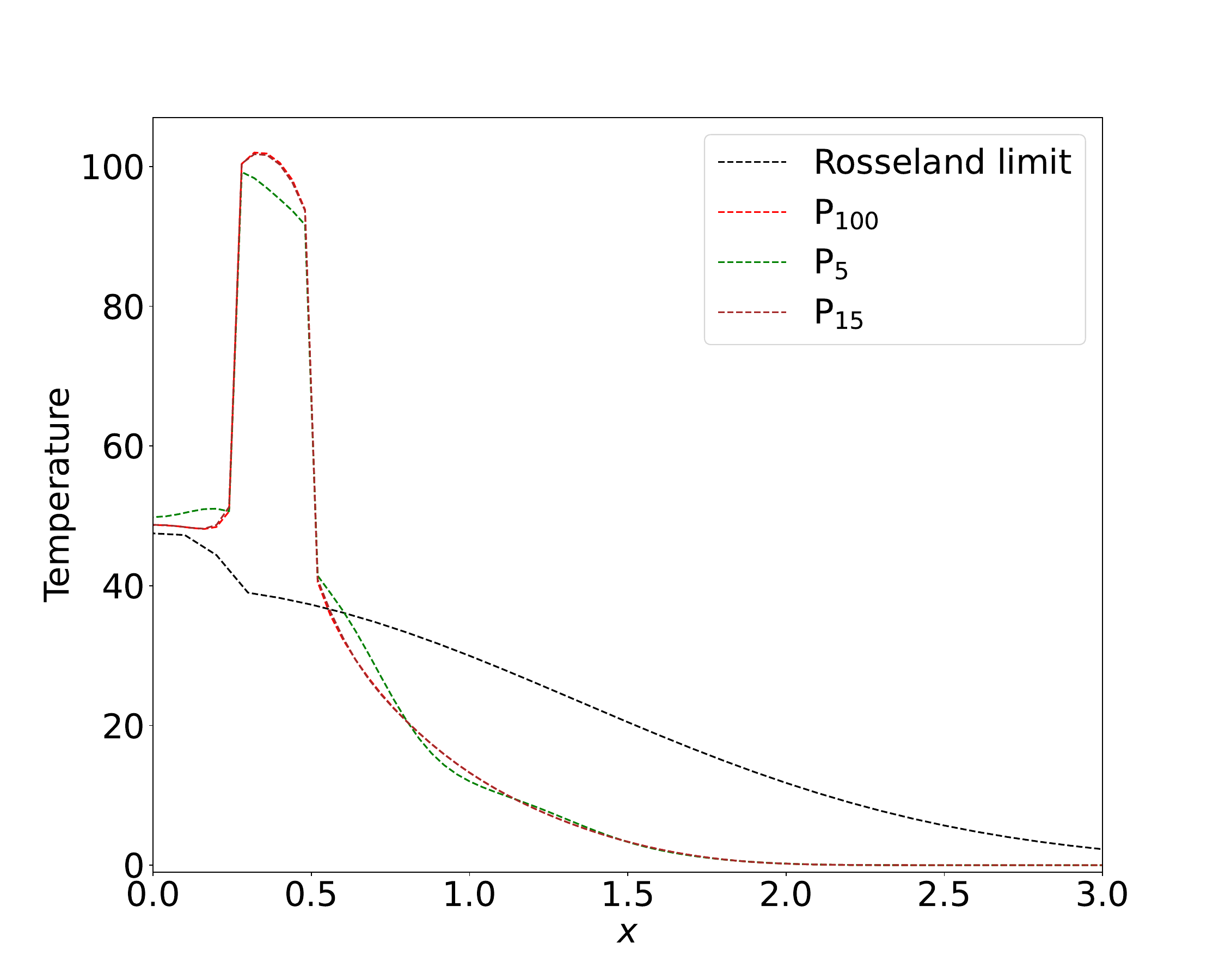}
    \end{subfigure}
    \begin{subfigure}[b]{0.49\linewidth}
        \includegraphics[width=0.9\linewidth]{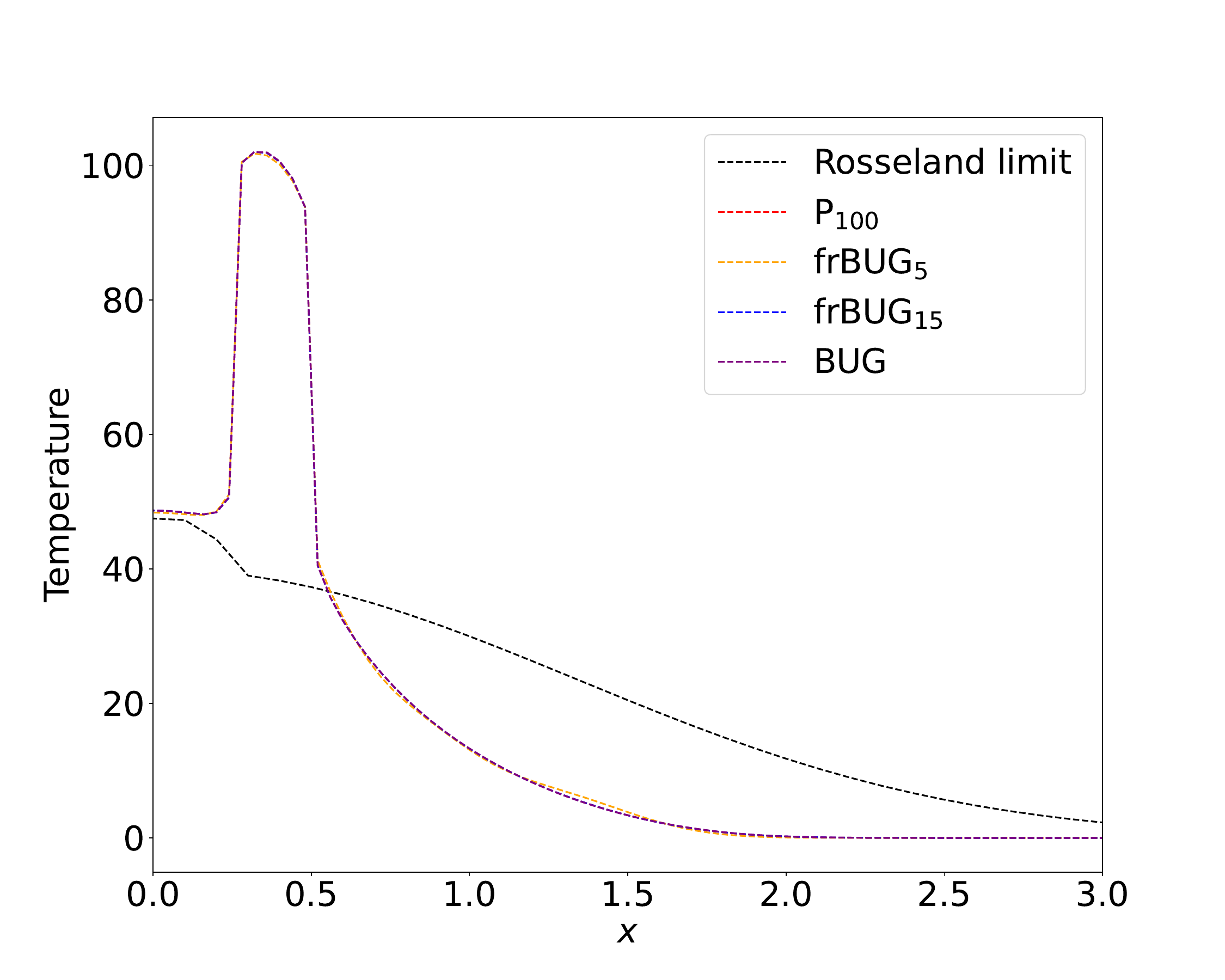}
    \end{subfigure}

    \begin{subfigure}[b]{0.49\linewidth}
        \includegraphics[width=0.9\linewidth]{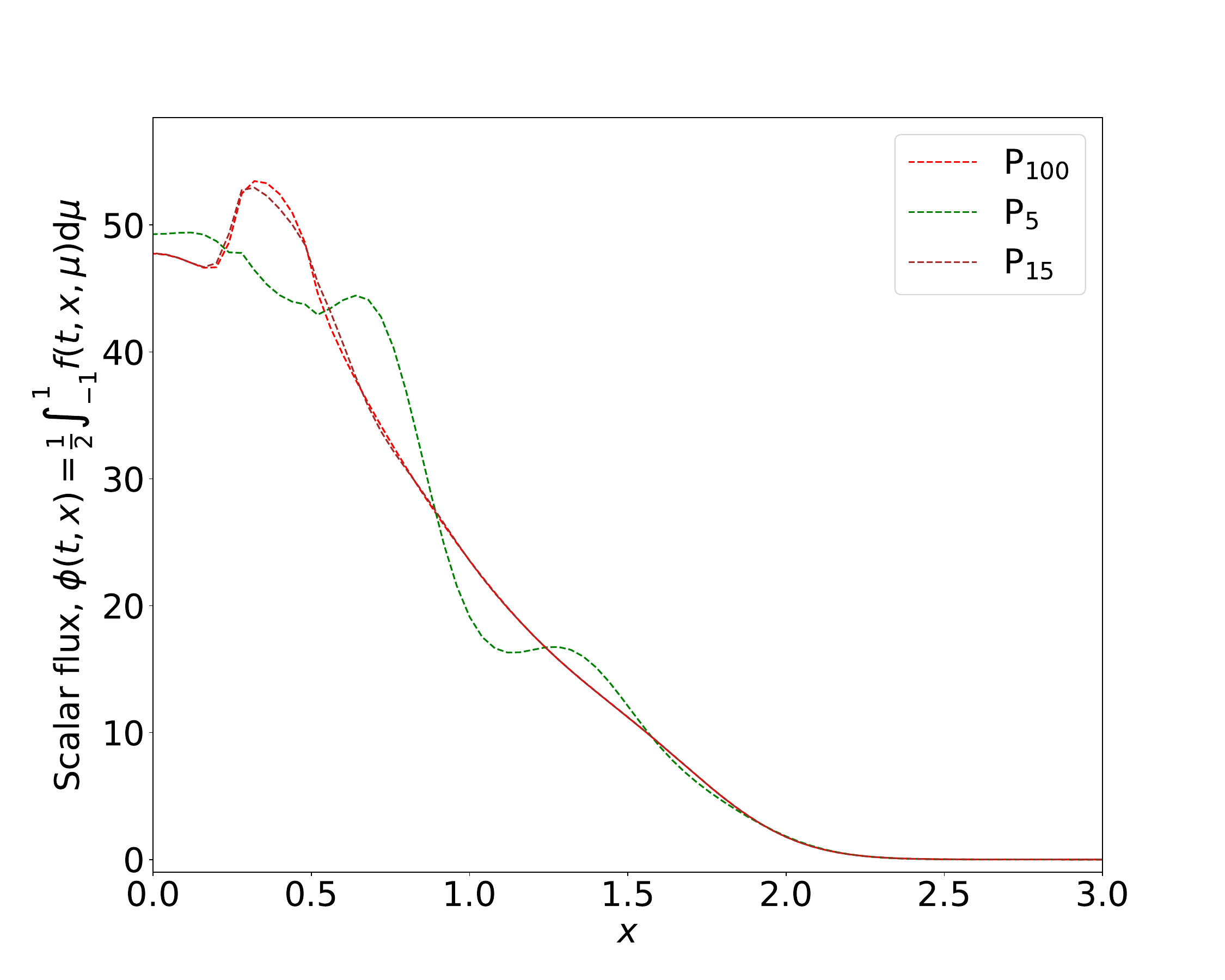}
        \caption{moment methods}
    \end{subfigure}
    \begin{subfigure}[b]{0.49\linewidth}
        \includegraphics[width=0.9\linewidth]{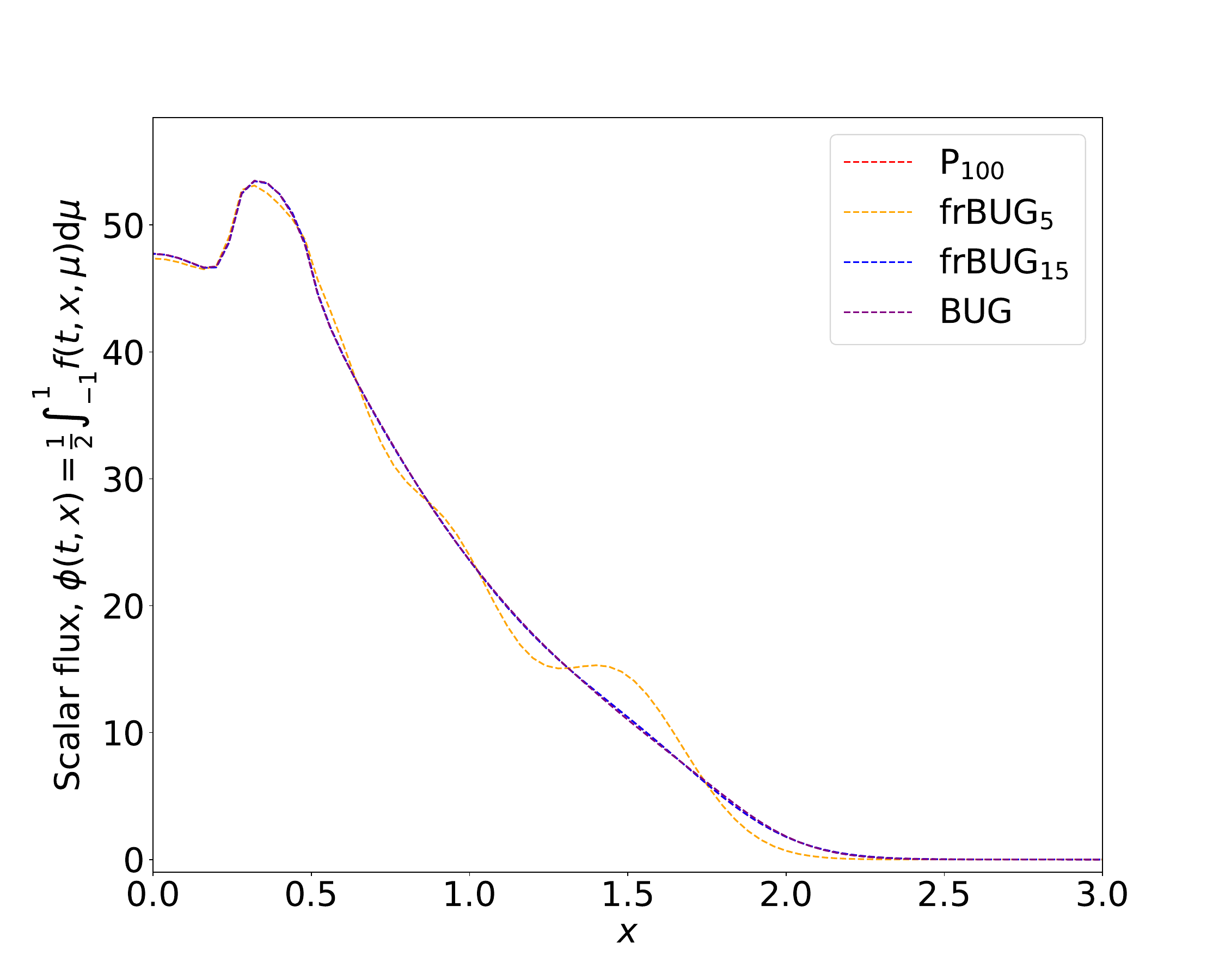}
        \caption{low-rank methods}
    \end{subfigure}

    \begin{subfigure}[b]{0.49\linewidth}
        \includegraphics[width=0.85\linewidth]{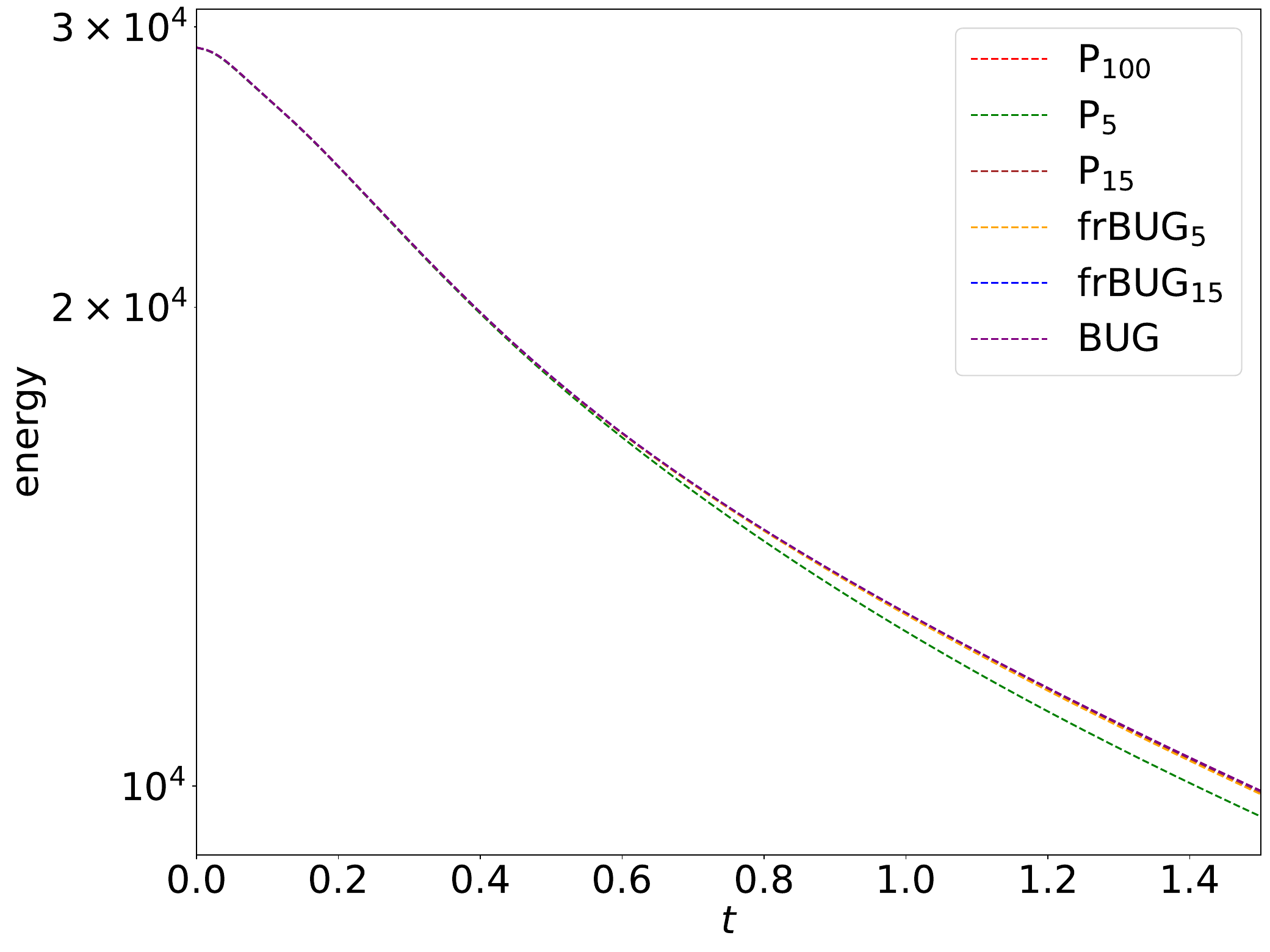}
        \caption{energy over time}
    \end{subfigure}
    \begin{subfigure}[b]{0.49\linewidth}
        \includegraphics[width=0.9\linewidth]{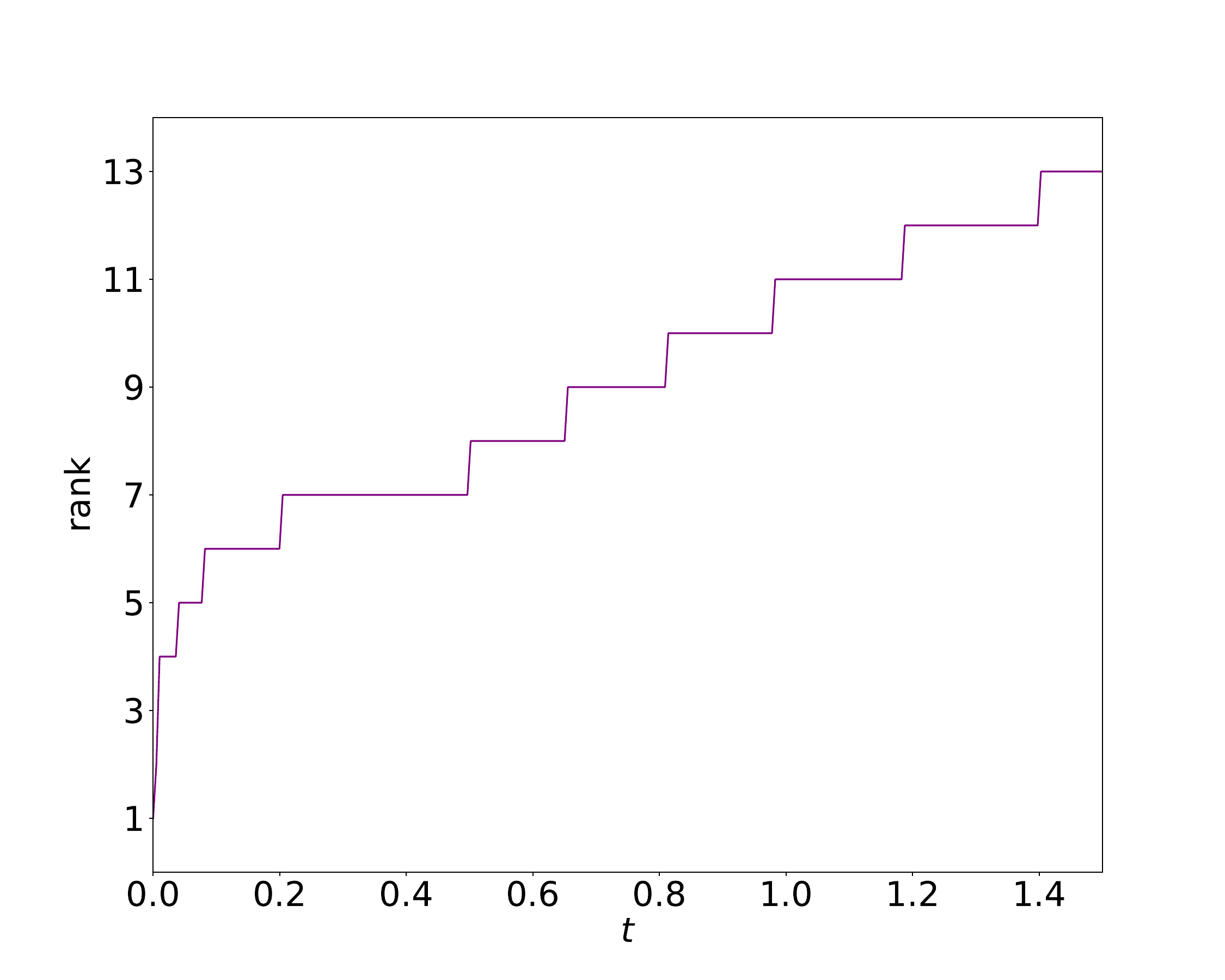}
        \caption{rank over time BUG}
    \end{subfigure}
    \caption{ Numerical results of the absorber test case in the kinetic regime, i.e., $\epsi = 1$ at $\vart = 1.5$. In the first row, we present the temperature profile at end-time for the moment and low-rank methods; in the second row, we have the corresponding scalar flux. In the last row, we have the energy of the system over time for all the methods and the rank evolution of the BUG integrator.}
    \label{fig:Absorber_1D_Kin}
\end{figure}

\begin{figure}[H]
    \centering
    \begin{subfigure}[b]{0.49\linewidth}
        \includegraphics[width=0.9\linewidth]{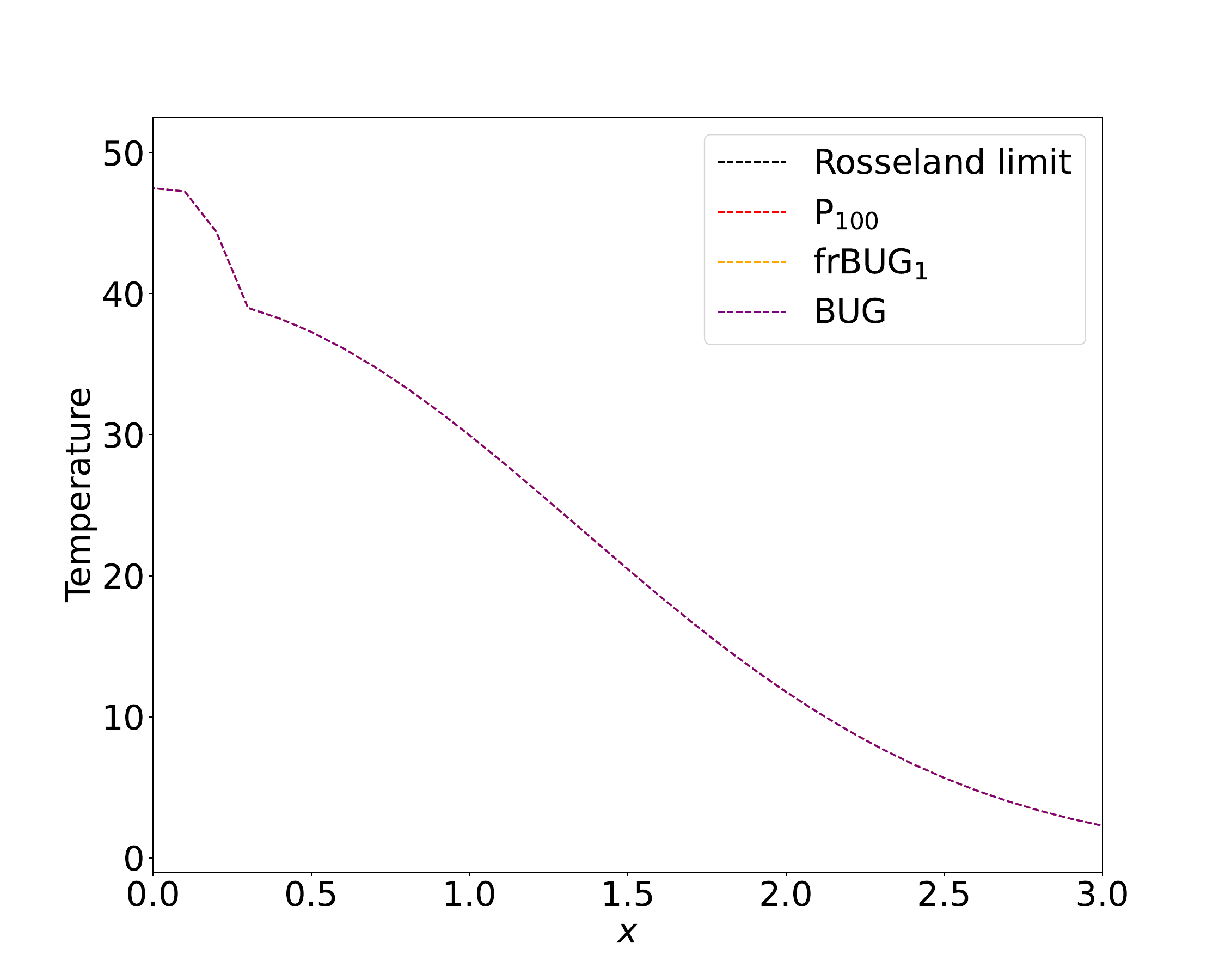}
    \end{subfigure}
    \begin{subfigure}[b]{0.49\linewidth}
        \includegraphics[width=0.9\linewidth]{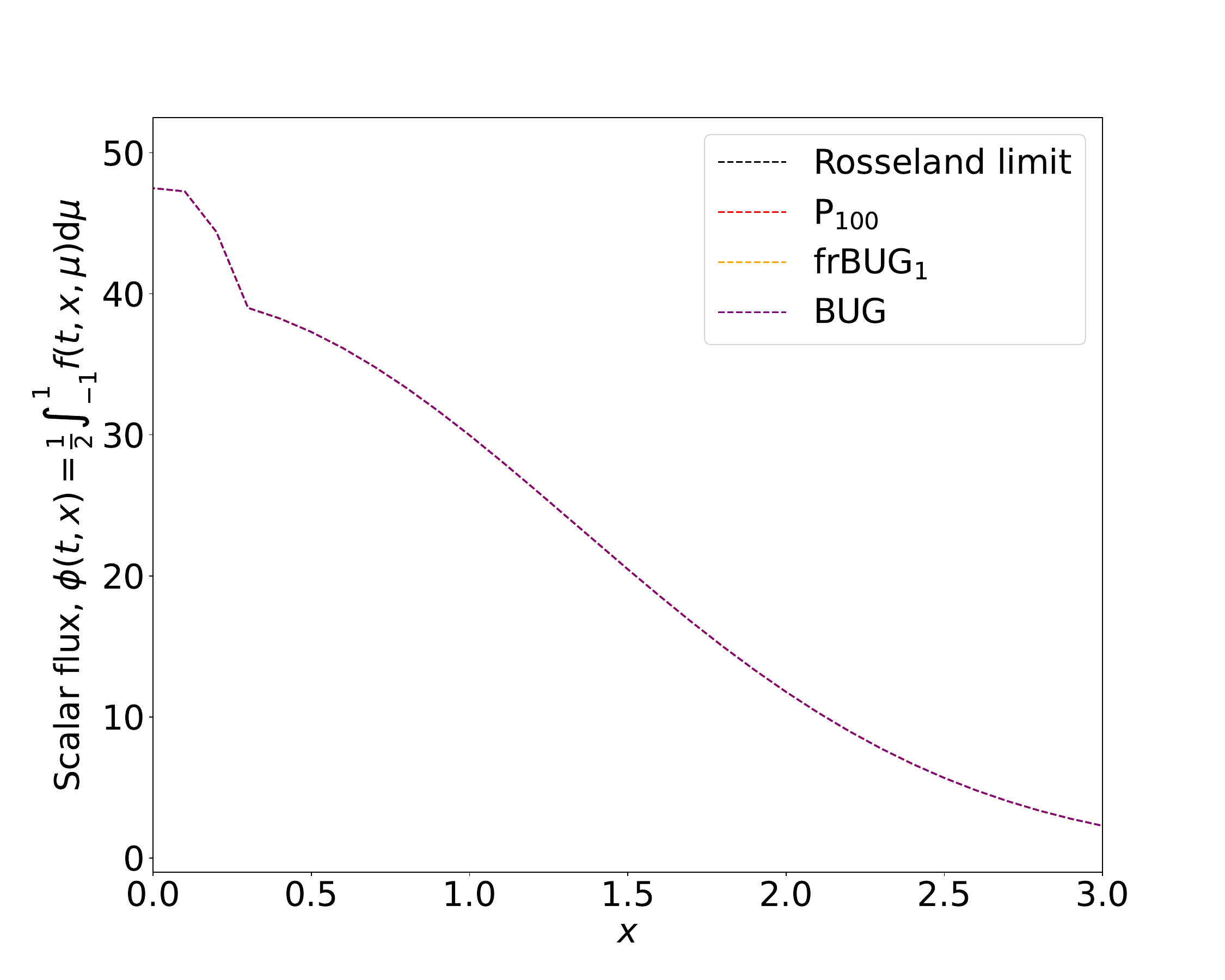}
    \end{subfigure}

    \begin{subfigure}[b]{0.49\linewidth}
        \includegraphics[width=0.85\linewidth]{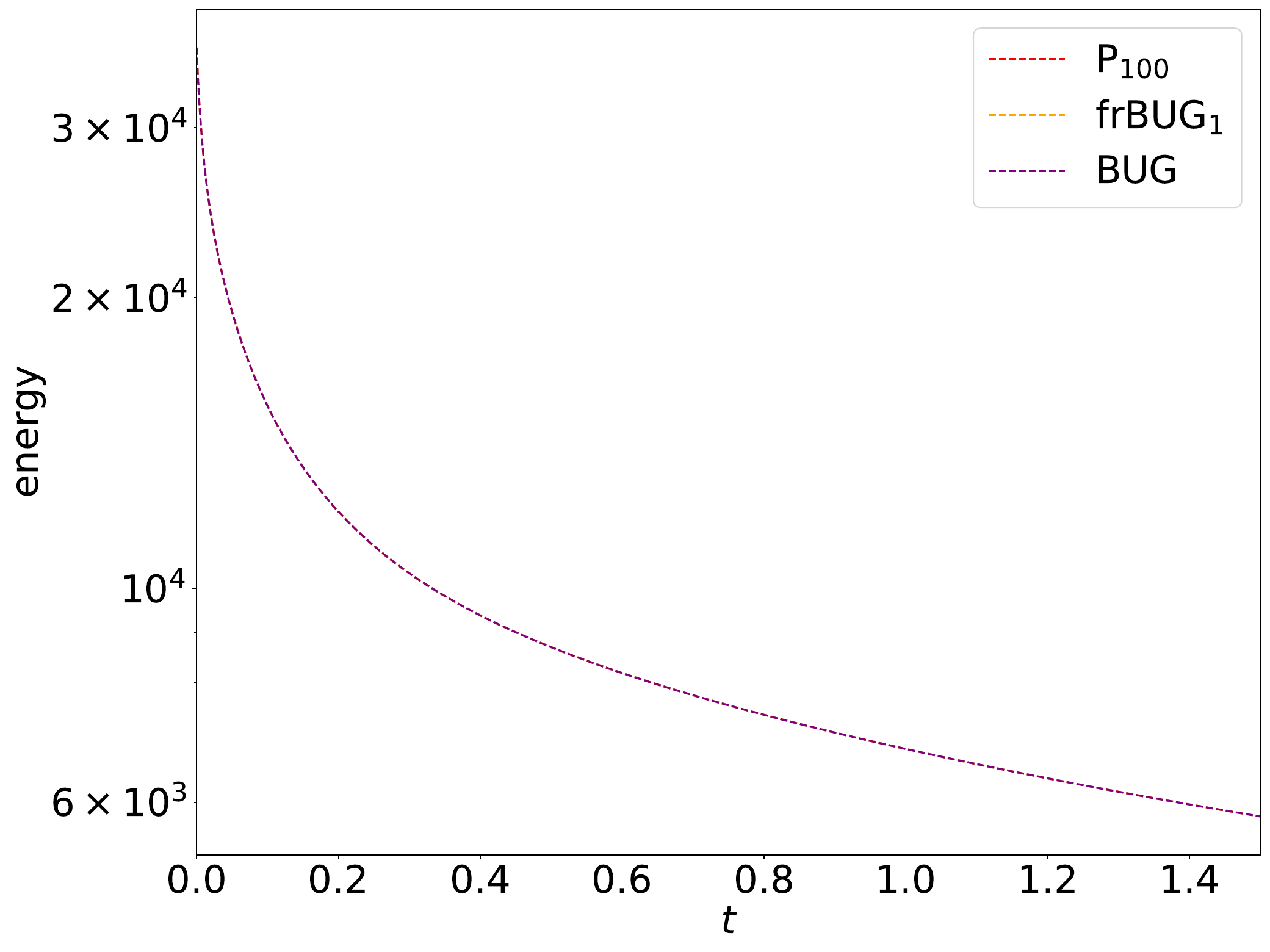}
    \end{subfigure}
    \begin{subfigure}[b]{0.49\linewidth}
        \includegraphics[width=0.9\linewidth]{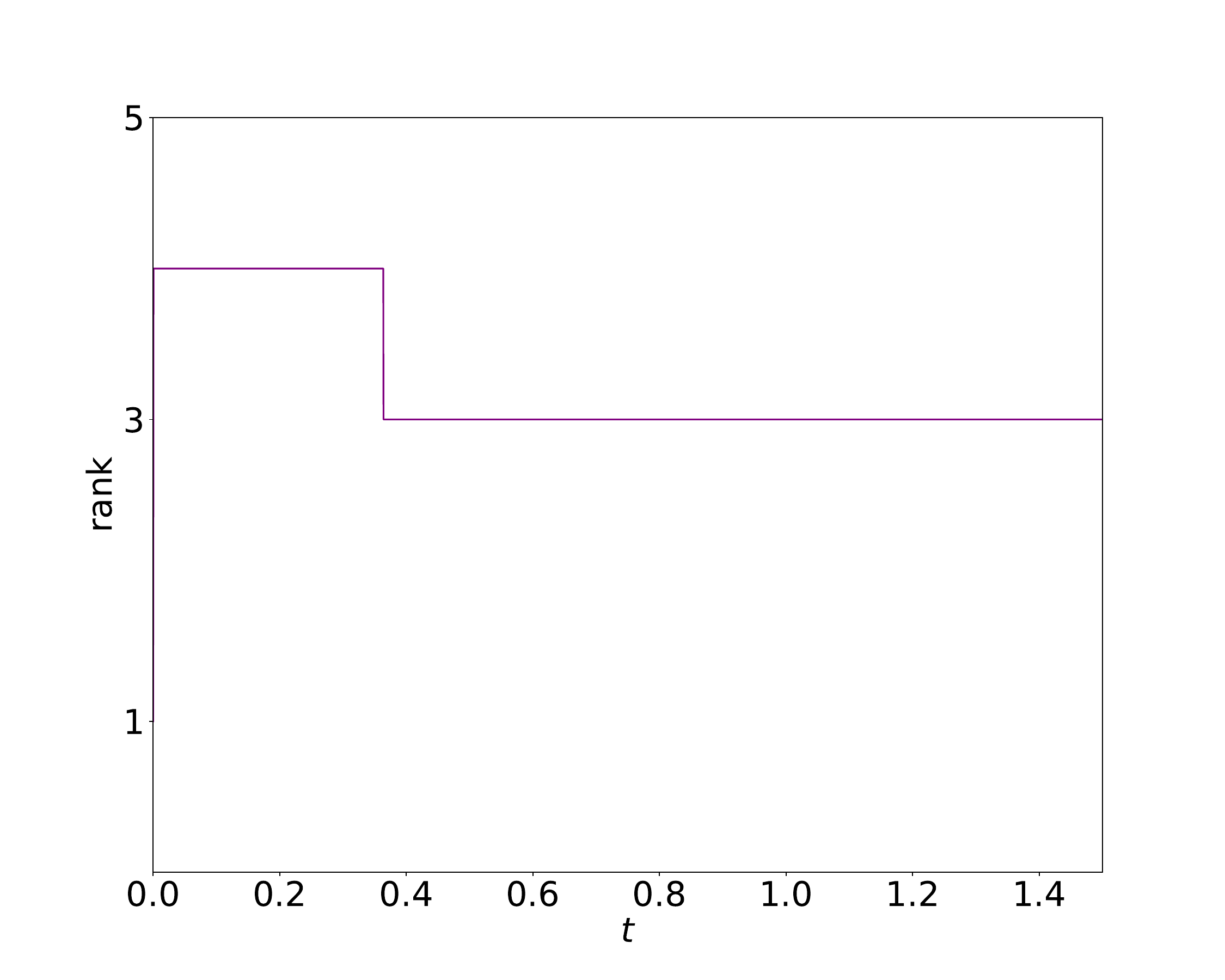}
    \end{subfigure}
    \caption{Numerical results of the absorber test case in the diffusive regime, i.e., $\epsi = 10^{-5}$ at $\vart = 1.5$. Top left: Temperature profile, Top right: Scalar flux, Bottom left: Energy of the system over time for all the methods, Bottom right: Rank evolution of rank-adaptive integrator over time.}
    \label{fig:Absorber_1D_Diff}
\end{figure}

\section{Conclusion}
In this work, we propose a modal macro-micro BUG integrator for the radiative heat transfer equations. We show that this integrator is energy stable for the linearized problem under a CFL condition that captures the kinetic regime and diffusive regime of the thermal radiative transfer equations. Additionally, the full modal macro-micro scheme's stability and the fixed-rank macro-micro BUG scheme have been investigated. 

\section*{Declaration of competing interests}
The authors declare that they have no known competing financial interests or personal relationships that could have
appeared to influence the work reported in this paper.
\section*{Acknowledgment}
\noindent The work of Chinmay Patwardhan and Martin Frank was funded by the Deutsche Forschungsgemeinschaft (DFG, German Research Foundation) – Project-ID 258734477 – SFB 1173.
\printbibliography
\appendix
\begin{appendices}
\section{Proof of Theorem 2}\label{appendix:ProofLinMMMstab}
\noindent We start by stating some lemmas and properties used to prove stability in energy norm (\Cref{theorem:LinMacMicEnStab}) for the linearized modal macro-micro scheme \eqref{eq:LinMacMicsys}.
\begin{lemma}[Lemma~3.3 \cite{einkemmer2022asymptoticpreserving}(Summation by parts)]\label{lemma:SumbyParts}
    For vectors $\bphi_{i+1/2}, \bzeta_{i+1/2}\in\R^{N}$ where $i = 0,\ldots,\Nx$, the equality
    \begin{equation}
        \sum_{i}\bzeta^{\top}_{i+1/2}\Diffpm\bphi_{i+1/2} = -\sum_{i}(\Diffmp\bzeta_{i+1/2})^{\top}\bphi_{i+1/2} 
    \end{equation}
    holds for periodic or zero values at the boundary.
\end{lemma}
\begin{lemma}\label{lemma:boundforDp}
    Let $\bphi_{i+1/2}\in\R^{N+1}$, for $i = 0,\ldots,Nx$, then the following inequality holds
    \begin{equation}
        \sum_{i}\left( \Diffp\bphi_{i+1/2} \right)^{2} \leq \frac{4}{\deltax^{2}}\sum_{i}(\bphi_{i+1/2})^{2}.
    \end{equation}
\end{lemma}
\begin{proof}
    Expanding the left-hand side using the definition of $\Diffp$
    \begin{align*}
        \sum_{i}\left( \Diffp\bphi_{i+1/2} \right)^{2} &= \frac{1}{\deltax^{2}}\sum_{i}(\bphi_{i+3/2} - \bphi_{i+1/2})^{2}\\
        &= \frac{2}{\deltax^{2}}\sum_{i}(\bphi_{i+1/2})^{2} -  \frac{2}{\deltax^{2}}\sum_{i} \bphi_{i+3/2}^{\top}\bphi_{i+1/2}.
    \end{align*}
   Using Young's inequality for the last term on the right-hand side in the above equation we get 
    \begin{equation}
        \left\lvert\frac{2}{\deltax^{2}}\sum_{i} \bphi_{i+3/2}^{\top}\bphi_{i+1/2}\right\rvert \leq\frac{1}{\deltax^{2}}\sum_{i}(\bphi_{i+1/2})^{2} + \frac{1}{\deltax^{2}}\sum_{i}(\bphi_{i+3/2})^{2}.
    \end{equation}
    A change in the index in the last term gives the desired result.
\end{proof}
The constructed macro-micro system \eqref{eq:LinMacMicsys} with the stabilization matrix $\absfluxMat$ is related to the full \PN system through the flux matrix $\FfluxMat = \left( \intgomgsb{\pk{i-1}\pk{j-1}\omg} \right)_{i,j=1}^{N+1}$ and Roe matrix $\absFfluxMat = \FTmat\absquadMat\FTmat^{\top}$, where $\FTmat = \left( \sqrt{\weight_{k}}\pk{i-1}(\quadpt_{k}) \right)_{i,j=1}^{N+1}$ such that $\FfluxMat = \FTmat\quadptMat\FTmat^{\top}$. We additionally define 
\begin{equation*}
    \frac{1}{\pknorm{0}}\bvec = \ba = (\ak{0},0,\ldots,0)^{\top}\in\R^{N}, \qquad \Fba = (0,\ak{0},0,\ldots,0)^{\top}\in\R^{N+1}.
\end{equation*}
\begin{lemma}[Lemma~3.4 \cite{einkemmer2022asymptoticpreserving} (\PN preservation)]\label{lemma:PnPreservation}
    For a given vector $\micvec\in\R^{N}$ define its extension $\bv \coloneqq (0,\mic_{1},\ldots,\mic_{N})^{\top}\in\R^{N+1}$ as well as $\hat{\bv}_{i+1/2} \coloneqq \FTmat^{\top}\bv_{i+1/2}\in\R^{N+1}$. Then,
    \begin{equation*}
        \micvec^{\top}\fluxMat^{2}\micvec = \hat{\bv}^{\top}\quadptMat^{2}\hat{\bv}, \quad \micvec^{\top}\absfluxMat\micvec = \hat{\bv}^{\top}\absquadMat\hat{\bv}, \quad \micvec^{\top}\ba\ba^{\top}\micvec = \hat{\bv}^{\top}\FTmat^{\top}\Fba\Fba^{\top}\FTmat\hat{\bv}.
    \end{equation*}
\end{lemma}
\noindent Two main properties of the advection operator, $\AdvecOp$, that are used in proving energy stability are
\begin{lemma}[Lemma~3.5 \cite{einkemmer2022asymptoticpreserving} (Positivity)]\label{lemma:Positivity}
    For a given discrete function $\micvec^{n}_{i+1/2} $, the advection operator fulfills the properties 
    \begin{equation*}
        \sum_{j}\micvec_{i+1/2}^{n+1,\top}\AdvecOp\micvec_{i+1/2}^{n+1} = \sum_{i}\frac{\deltax}{2}\Diffp\micvec_{i+1/2}^{n+1,\top}\absfluxMat\Diffp\micvec_{i+1/2}^{n+1}\geq 0
    \end{equation*}
    and 
    \begin{equation*}
    \begin{aligned}
        \sum_{j}\micvec_{i+1/2}^{n+1,\top}\AdvecOp\micvec_{i+1/2}^{n} &= \sum_{i}\frac{\deltax}{2}\Diffp\micvec_{i+1/2}^{n+1,\top}\absfluxMat\Diffp\micvec_{i+1/2}^{n+1}\\ &\quad + \sum_{i}(\micvec_{i+1/2}^{n} - \micvec_{i+1/2}^{n+1})^{\top}(\fluxMatp\Diffp + \fluxMatm\Diffm)\micvec_{i+1/2}^{n+1}.
    \end{aligned}
    \end{equation*}
\end{lemma}

\begin{lemma}[Lemma~3.6 \cite{einkemmer2022asymptoticpreserving} (Boundedness)]\label{lemma:boundedness}
     For a given discrete function $\micvec^{n}_{i+1/2}$, the advection operator fulfills the property 
     \begin{equation*}
        \sum_{i}\left[ (\fluxMatp\Diffp + \fluxMatm\Diffm)\micvec^{n+1}_{i+1/2} \right]^{2} \leq 2\beta_{N}\sum_{i}\Diffp\micvec^{n+1,\top}_{i+1/2}\fluxMat^{2}\Diffp\micvec^{n+1}_{i+1/2},
     \end{equation*}
     where $\beta_{N} = \underset{k}{\mathrm{max}}\hspace{2mm}\weight_{k}(N+1)$ is bounded for all $N$.
\end{lemma}
\noindent With these lemma we now present the proof of \cref{theorem:LinMacMicEnStab}:

\begin{proof}
    \noindent We start by plugging in the update equation of the macro variable, $\Temp$, \eqref{eq:LinMacMicsysmac} into the update equation of the mesoscopic variable, $\meso$, \eqref{eq:LinMacMicsysmeso}. This gives
    \begin{subequations}\label{eq:Thm1Eq1}
        \begin{equation}\label{eq:Thm1Eq1.1}
            \frac{(\StefBoltzconst\Temp_{i}^{n+1} + \frac{\epsi^{2}}{\sol}\meso_{i}^{n+1}) - (\StefBoltzconst\Temp_{i}^{n} + \frac{\epsi^{2}}{\sol}\meso_{i}^{n})}{\deltat} + \frac{\pknorm{1}}{2}\DiffO\mic_{1,i}^{n+1} = -\absorpcoeff_{i}\meso_{i}^{n+1},
        \end{equation}
        \begin{equation}\label{eq:Thm1Eq1.2}
            \frac{1}{c}\left( \frac{\micvec_{i+1/2}^{n+1} - \micvec_{i+1/2}^{n}}{\deltat} \right) + \frac{1}{\epsi}\AdvecOp\micvec_{i+1/2}^{n} = -\frac{\absorpcoeff_{i+1/2}}{\epsi^{2}}\micvec_{i+1/2}^{n+1} - \frac{1}{\epsi^{2}}\bvec\hspace{0.5mm}\deltaO(\StefBoltzconst\sol\Temp^{n}_{i+1/2} + \epsi^{2}\meso^{n}_{i+1/2}).
        \end{equation}
    \end{subequations}
    Next we multiply the update equation for the micro equation \eqref{eq:Thm1Eq1.1} by $ (\StefBoltzconst\Temp_{i}^{n+1} + \frac{\epsi^{2}}{\sol}\meso_{i}^{n+1})\deltax $, and sum over $i$,
    \begin{equation*}
        \begin{aligned}
            \frac{1}{\deltat}\sum_{i}(\StefBoltzconst\Temp_{i}^{n+1} + \frac{\epsi^{2}}{\sol}\meso_{i}^{n+1})^{2}\deltax - \frac{1}{\deltat}\sum_{i}(\StefBoltzconst\Temp_{i}^{n+1} &+ \frac{\epsi^{2}}{\sol}\meso_{i}^{n+1})(\StefBoltzconst\Temp_{i}^{n} + \frac{\epsi^{2}}{\sol}\meso_{i}^{n})\deltax\\
            \qquad + \frac{\pknorm{1}}{2}\sum_{i}(\StefBoltzconst\Temp_{i}^{n+1} + \frac{\epsi^{2}}{\sol}\meso_{i}^{n+1})\DiffO\mic_{1,i}^{n+1}\deltax &= -\sum_{i}(\StefBoltzconst\Temp_{i}^{n+1} + \frac{\epsi^{2}}{\sol}\meso_{i}^{n+1})\absorpcoeff_{i}\meso_{i}^{n+1}\deltax.
        \end{aligned}
    \end{equation*}
    Using the summation of products \cref{lemma:sumofsquares} we get
    \begin{equation}\label{eq:Thm1Eq2}
        \begin{aligned}
            \frac{1}{2\deltat}\left( \norm{\StefBoltzconst\Temp^{n+1} + \frac{\epsi^{2}}{\sol}\meso^{n+1}}^{2} - \norm{\StefBoltzconst\Temp^{n} + \frac{\epsi^{2}}{\sol}\meso^{n}}^{2} + \norm{\StefBoltzconst\Temp^{n+1} + \frac{\epsi^{2}}{\sol}\meso^{n+1} - \StefBoltzconst\Temp^{n} - \frac{\epsi^{2}}{\sol}\meso^{n}}^{2} \right) +\\
            \frac{\pknorm{1}}{2}\sum_{i}(\StefBoltzconst\Temp_{i}^{n+1} + \frac{\epsi^{2}}{\sol}\meso_{i}^{n+1})\DiffO\mic_{1,i}^{n+1}\deltax = -\sum_{i}(\StefBoltzconst\Temp_{i}^{n+1} + \frac{\epsi^{2}}{\sol}\meso_{i}^{n+1})\absorpcoeff_{i}\meso_{i}^{n+1}\deltax
        \end{aligned}
    \end{equation}
    Multiply \eqref{eq:Thm1Eq1.2} by $ \micvec^{n+1,\top}_{i+1/2}\deltax $, and sum over $i$, and using the summation \cref{lemma:sumofsquares}
    \begin{equation}\label{eq:Thm1Eq3}
        \begin{aligned}
            \frac{1}{2\deltat}\left(\frac{1}{c}\norm{\micmat^{n+1}}^{2}  - \frac{1}{c}\norm{\micmat^{n}}^{2}  + \frac{1}{c}\norm{\micmat^{n+1} - \micmat^{n}}^{2} \right) + \frac{1}{\epsi}\sum_{i}\micvec^{n+1,\top}_{i+1/2}\AdvecOp\micvec^{n}_{i+1/2}\deltax \\
            \hspace{2cm}\leq - \frac{\absorpcoefflb}{\epsi^{2}}\norm{\micmat^{n+1}}^{2} - \frac{\pknorm{1}}{\epsi^{2}}\sum_{i}\mic^{n+1}_{1,i+1/2}\hspace{0.5mm}\deltaO(\StefBoltzconst\sol\Temp^{n}_{i+1/2} + \epsi^{2}\meso^{n}_{i+1/2})\deltax,
        \end{aligned}
    \end{equation}
    where we use $\absorpcoefflb \leq \absorpcoeff_{i+1/2}, \forall i$, and $\micvec^{n+1,\top}_{i+1/2}\bvec = \pknorm{1}\mic^{n+1}_{1,i+1/2} $. \\
    Then \eqref{eq:Thm1Eq2} $ + \frac{\epsi^{2}}{\pknorm{0}^{2}\sol}\times$ \eqref{eq:Thm1Eq3}, where $\pknorm{0}^{2} = 2$ is the normalization factor of the zeroth order Legendre polynomial, gives
    \begin{equation}\label{eq:Thm1Eq4}
        \begin{aligned}
            \frac{1}{2\deltat}&\left( \norm{\StefBoltzconst\Temp^{n+1} + \frac{\epsi^{2}}{\sol}\meso^{n+1}}^{2} + \norm{\frac{\epsi}{\pknorm{0}c}\micmat^{n+1}}^{2} - \norm{\StefBoltzconst\Temp^{n} + \frac{\epsi^{2}}{\sol}\meso^{n}}^{2}  - \norm{\frac{\epsi}{\pknorm{0}c}\micmat^{n}}^{2}  \right.\\
            &\left. + \norm{\StefBoltzconst\Temp^{n+1} + \frac{\epsi^{2}}{\sol}\meso^{n+1} - \StefBoltzconst\Temp^{n} - \frac{\epsi^{2}}{\sol}\meso^{n}}^{2} + \norm{\frac{\epsi}{\pknorm{0}c}\micmat^{n+1} - \frac{\epsi}{\pknorm{0}c}\micmat^{n}}^{2} \right) \\
            &+ \frac{\pknorm{1}}{2}\sum_{i}(\StefBoltzconst\Temp_{i}^{n+1} + \frac{\epsi^{2}}{\sol}\meso_{i}^{n+1})\DiffO\mic_{1,i}^{n+1}\deltax 
            + \frac{\epsi}{2\sol}\sum_{i}\micvec^{n+1,\top}_{i+1/2}\AdvecOp\micvec^{n}_{i+1/2}\deltax \\ &\leq -\sum_{i}(\StefBoltzconst\Temp_{i}^{n+1} + \frac{\epsi^{2}}{\sol}\meso_{i}^{n+1})\absorpcoeff_{i}\meso_{i}^{n+1}\deltax 
            - \frac{\absorpcoefflb}{2\sol}\norm{\micmat^{n+1}}^{2}- \frac{\pknorm{1}}{\pknorm{0}^{2}}\sum_{i}\mic^{n+1}_{1,i+1/2}\hspace{0.5mm}\deltaO(\StefBoltzconst\Temp^{n}_{i+1/2} + \frac{\epsi^{2}}{\sol}\meso^{n}_{i+1/2})\deltax \,.
        \end{aligned}
    \end{equation}
    \noindent Using discrete integration by parts from \Cref{lemma:SumbyParts} we get
    \begin{equation}\label{eq:Thm1Eq5}
        \sum_{i}\mic^{n+1}_{1,i+1/2}\hspace{0.5mm}\deltaO(\StefBoltzconst\Temp^{n}_{i+1/2} + \frac{\epsi^{2}}{\sol}\meso^{n}_{i+1/2})\deltax = -\sum_{i}\DiffO\mic^{n+1}_{1,i}(\StefBoltzconst\Temp_{i}^{n} + \frac{\epsi^{2}}{\sol}\meso^{n}_{i})\deltax.
    \end{equation}
    With the use of \eqref{eq:Thm1Eq5} we rewrite \eqref{eq:Thm1Eq4} as
    \begin{equation*}\label{eq:Thm1Eq6}
        \begin{aligned}
             \frac{1}{2\deltat}\left( \norm{\StefBoltzconst\Temp^{n+1} + \frac{\epsi^{2}}{\sol}\meso^{n+1}}^{2} + \norm{\frac{\epsi}{\pknorm{0}c}\micmat^{n+1}}^{2} - \norm{\StefBoltzconst\Temp^{n} + \frac{\epsi^{2}}{\sol}\meso^{n}}^{2}  - \norm{\frac{\epsi}{\pknorm{0}c}\micmat^{n}}^{2}  \right.&\\
            \left. + \norm{\StefBoltzconst\Temp^{n+1} + \frac{\epsi^{2}}{\sol}\meso^{n+1} - \StefBoltzconst\Temp^{n} - \frac{\epsi^{2}}{\sol}\meso^{n}}^{2} + \norm{\frac{\epsi}{\pknorm{0}c}\micmat^{n+1} - \frac{\epsi}{\pknorm{0}c}\micmat^{n}}^{2} \right)& \\
            + \frac{\epsi}{2\sol}\sum_{i}\micvec^{n+1,\top}_{i+1/2}\AdvecOp\micvec^{n}_{i+1/2}\deltax \leq  - \frac{\absorpcoefflb}{2\sol}\norm{\micmat^{n+1}}^{2} -\sum_{i}(\StefBoltzconst\Temp_{i}^{n+1} + \frac{\epsi^{2}}{\sol}\meso_{i}^{n+1})&\absorpcoeff_{i}\meso_{i}^{n+1}\deltax \\
            +\frac{\pknorm{1}}{2}\sum_{i}\DiffO\mic^{n+1}_{1,i}(-\StefBoltzconst\Temp^{n+1}_{i} - \frac{\epsi^{2}}{\sol}\meso^{n+1}_{i} +& \StefBoltzconst\Temp^{n}_{i} + \frac{\epsi^{2}}{\sol}\meso^{n}_{i})\deltax.
        \end{aligned}
    \end{equation*}
    Using Young's inequality,
    \begin{equation*}\label{eq:Thm1Eq7}
    \begin{aligned}
        \frac{\pknorm{1}}{2}\sum_{i}\DiffO\mic^{n+1}_{1,i}(-\StefBoltzconst\Temp^{n+1}_{i} - \frac{\epsi^{2}}{\sol}\meso^{n+1}_{i} + \StefBoltzconst\Temp^{n}_{i} + \frac{\epsi^{2}}{\sol}\meso^{n}_{i})\deltax &\leq \alpha\norm{-\StefBoltzconst\Temp^{n+1} - \frac{\epsi^{2}}{\sol}\meso^{n+1} + \StefBoltzconst\Temp^{n} + \frac{\epsi^{2}}{\sol}\meso^{n}}\\
        &\quad + \frac{1}{4\alpha}\sum_{i}\frac{\pknorm{1}^{2}}{4}(\DiffO\mic^{n+1}_{1,i})^{2}\deltax.
    \end{aligned}
    \end{equation*}
    Thus, setting $\alpha = \frac{1}{2\deltat}$ we have
    \begin{equation}\label{eq:Thm1Eq8}
        \begin{aligned}
            \frac{1}{2\deltat}&\left( \norm{\StefBoltzconst\Temp^{n+1} + \frac{\epsi^{2}}{\sol}\meso^{n+1}}^{2} + \norm{\frac{\epsi}{\pknorm{0}c}\micmat^{n+1}}^{2} - \norm{\StefBoltzconst\Temp^{n} + \frac{\epsi^{2}}{\sol}\meso^{n}}^{2}  - \norm{\frac{\epsi}{\pknorm{0}c}\micmat^{n}}^{2}  \right.\\
            &\left.+ \norm{\frac{\epsi}{\pknorm{0}c}\micmat^{n+1} - \frac{\epsi}{\pknorm{0}c}\micmat^{n}}^{2} \right)
            + \frac{\epsi}{2\sol}\sum_{i}\micvec^{n+1,\top}_{i+1/2}\AdvecOp\micvec^{n}_{i+1/2}\deltax\\&\leq  - \frac{\absorpcoefflb}{2\sol}\norm{\micmat^{n+1}}^{2} -\sum_{i}(\StefBoltzconst\Temp_{i}^{n+1} + \frac{\epsi^{2}}{\sol}\meso_{i}^{n+1})\absorpcoeff_{i}\meso_{i}^{n+1}\deltax 
           +\frac{\deltat}{2}\frac{\pknorm{1}^{2}}{4}\sum_{i}(\DiffO\mic^{n+1}_{1,i})^{2}\deltax.\qquad \qquad
        \end{aligned}
    \end{equation}\label{eq:Thm1Eq9}
    This gives an upper bound for the term $\norm{\StefBoltzconst\Temp^{n+1} + \frac{\epsi^{2}}{\sol}\meso^{n+1}}^{2} + \norm{\frac{\epsi}{\pknorm{0}\sol}\micmat^{n+1}}$. Next, we derive an upper bound for $ \norm{\sqrt{\frac{\StefBoltzconst}{\facalph}}\Temp^{n+1}}^{2}$. For that we multiply the temperature update \eqref{eq:LinMacMicsysmac} by $\StefBoltzconst\Temp^{n+1}_{i}\deltax $ and sum over $i$
    \begin{equation*}
        \frac{1}{\deltat}\sum_{i}\StefBoltzconst\Temp^{n+1}_{i}(\Temp^{n+1}_{i} - \Temp^{n}_{i})\deltax = \facalph\sum_{i}\StefBoltzconst\Temp^{n+1}_{i}\absorpcoeff_{i}\meso^{n+1}_{i}\deltax.
    \end{equation*}
    Thus, we get using the summation of products \cref{lemma:sumofsquares}
    \begin{equation}\label{eq:Thm1Eq10}
        \frac{1}{2\deltat}\left( \norm{\sqrt{\frac{\StefBoltzconst}{\facalph}}\Temp^{n+1}}^{2} - \norm{\sqrt{\frac{\StefBoltzconst}{\facalph}}\Temp^{n}}^{2} + \norm{\sqrt{\frac{\StefBoltzconst}{\facalph}}\Temp^{n+1} - \sqrt{\frac{\StefBoltzconst}{\facalph}}\Temp^{n}}^{2} \right) = \sum_{i}\StefBoltzconst\Temp^{n+1}_{i}\absorpcoeff_{i}\meso^{n+1}_{i}\deltax.
    \end{equation}
    Adding \eqref{eq:Thm1Eq8} and \eqref{eq:Thm1Eq10} gives
    \begin{equation}\label{eq:Thm1Eq11}
        \begin{aligned}
            \frac{1}{2\deltat}&\left( \energy^{n+1} - \energy^{n}+ \norm{\frac{\epsi}{\pknorm{0}c}\micmat^{n+1} - \frac{\epsi}{\pknorm{0}c}\micmat^{n}}^{2}\right) + \frac{\epsi}{2\sol}\sum_{i}\micvec^{n+1,\top}_{i+1/2}\AdvecOp\micvec^{n}_{i+1/2}\deltax\\ &\leq  - \frac{\absorpcoefflb}{2\sol}\norm{\micmat^{n+1}}^{2} - \frac{\absorpcoefflb}{\sol}\norm{\epsi\meso^{n+1}}^{2}
            +\frac{\deltat}{2}\frac{\pknorm{1}^{2}}{4}\sum_{i}
           (\DiffO\mic^{n+1}_{1,i})^{2}\deltax - \norm{\sqrt{\frac{\StefBoltzconst}{\facalph}}\Temp^{n+1} - \sqrt{\frac{\StefBoltzconst}{\facalph}}\Temp^{n}}^{2}.
        \end{aligned}
    \end{equation}
    Since $- \frac{\absorpcoefflb}{\sol}\norm{\epsi\meso^{n+1}}^{2} \leq 0$ and $- \norm{\sqrt{\frac{\StefBoltzconst}{\facalph}}\Temp^{n+1} - \sqrt{\frac{\StefBoltzconst}{\facalph}}\Temp^{n}}^{2} \leq 0$ we bound them from above by $0$. 

    Then \eqref{eq:Thm1Eq11} becomes
        \begin{equation}\label{eq:Thm1Eq13}
        \begin{aligned}
            \frac{1}{2\deltat}\left( \energy^{n+1} -\energy^{n} + \norm{\frac{\epsi}{\pknorm{0}c}\micmat^{n+1} - \frac{\epsi}{\pknorm{0}c}\micmat^{n}}^{2} \right)
            + \frac{\epsi}{2\sol}\sum_{i}\micvec^{n+1,\top}_{i+1/2}\AdvecOp\micvec^{n}_{i+1/2}\deltax\\ \leq - \frac{\absorpcoefflb}{2\sol}\norm{\micmat^{n+1}}^{2}
           +\frac{\deltat}{2}\frac{\pknorm{1}^{2}}{4}&\sum_{i}
           (\DiffO\mic^{n+1}_{1,i})^{2}\deltax.\hspace{2cm}
        \end{aligned}
    \end{equation}
    Next, we split the last term on the left-hand side of \eqref{eq:Thm1Eq13} as
    \begin{equation*}\label{eq:Thm1Eq14}
        \sum_{i}\micvec^{n+1}_{i+1/2}\AdvecOp\micvec^{n}_{i+1/2}\deltax = P + Q
    \end{equation*}
    where
    \begin{equation*}
        P = \sum_{i}\micvec^{n+1}_{i+1/2}\AdvecOp\micvec^{n+1}_{i+1/2}\deltax, \quad Q = \sum_{i}\micvec^{n+1}_{i+1/2}\AdvecOp(\micvec^{n}_{i+1/2}-\micvec^{n+1}_{i+1/2})\deltax.
    \end{equation*}
     Then, using \Cref{lemma:Positivity} we have,
     \begin{equation}\label{eq:Thm1Eq15}
         \begin{aligned}
             P &= \frac{\deltax}{2}\sum_{i}\Diffp\micvec^{n+1,\top}_{i+1/2}\absfluxMat\Diffp\micvec^{n+1}_{i+1/2}\deltax,\\ Q &= -\sum_{i}(\fluxMatp\Diffp + \fluxMatm\Diffm)\micvec^{n+1,\top}_{i+1/2}(\micvec^{n}_{i+1/2}-\micvec^{n+1}_{i+1/2}).
         \end{aligned}
     \end{equation}
     Thus, using Young's inequality
     \begin{equation*}\label{eq:Thm1Eq16}
         \abs{Q} \leq \alpha\norm{\micmat^{n+1} - \micmat^{n}} + \frac{1}{4\alpha}\sum_{i}\left[(\fluxMatp\Diffp + \fluxMatm\Diffm)\micvec^{n+1}_{i+1/2}  \right]^{2}\deltax
     \end{equation*}
     and \Cref{lemma:Positivity} we get
     \begin{equation}\label{eq:Thm1Eq17}
          \abs{Q} \leq \alpha\norm{\micmat^{n+1} - \micmat^{n}} + \frac{2\beta_{N}}{4\alpha}\sum_{i}\Diffp\micvec^{n+1,\top}_{i+1/2}\fluxMat^{2}\Diffp\micvec^{n+1}_{i+1/2}\deltax.
     \end{equation}
     We set $\alpha = \frac{\epsi}{2\sol\deltat}$, then $\norm{\frac{\epsi}{\pknorm{0}\sol}\micmat^{n+1} - \frac{\epsi}{\pknorm{0}\sol}\micmat^{n}}$ gets canceled out and we get 

     \begin{equation}\label{eq:Thm1Eq18}
         \begin{aligned}
             \frac{1}{2\deltat}\left( \energy^{n+1} - \energy^{n} \right) + \frac{\epsi\deltax}{4\sol}\sum_{i}\Diffp\micvec^{n+1,\top}_{i+1/2}\absfluxMat\Diffp\micvec^{n+1}_{i+1/2}\deltax - \beta_{N}\frac{\deltat}{2}\sum_{i}\Diffp\micvec^{n+1,\top}_{i+1/2}\fluxMat^{2}\Diffp\micvec^{n+1}_{i+1/2}\deltax \\ \leq - \frac{\absorpcoefflb}{2\sol}\norm{\micmat^{n+1}}^{2}
           +\frac{\deltat}{2}\frac{\pknorm{1}^{2}}{4}\sum_{i}
           (\DiffO\mic^{n+1}_{1,i})^{2}\deltax
         \end{aligned}
     \end{equation}
    We rewrite the last term on the right-hand side of \eqref{eq:Thm1Eq18} as 
    \begin{equation*}\label{eq:Thm1Eq19}
        \frac{\pknorm{1}^{2}}{4}\sum_{i}
           (\DiffO\mic^{n+1}_{1,i})^{2}\deltax = \frac{\pknorm{1}^{2}}{4}\sum_{i}(\Diffp\mic^{n+1}_{1,i})^{2}\deltax = \frac{1}{2}\sum_{i}\Diffp\micvec^{n+1,\top}_{i+1/2}\ba\ba^{\top}\Diffp\micvec^{n+1}_{i+1/2}\deltax
    \end{equation*}
    and we get
    \begin{equation}\label{eq:Thm1Eq20}
        \begin{aligned}
            \frac{1}{2\deltat}\left( \energy^{n+1} - \energy^{n} \right) &\leq - \frac{\absorpcoefflb}{2\sol}\norm{\micmat^{n+1}}^{2}\\ &\quad+ \frac{1}{2}\sum_{i}\Diffp\micvec^{n+1,\top}_{i+1/2}\left[ \deltat\left(\beta_{N}\fluxMat^{2} + \frac{1}{2}\ba\ba^{\top}\right) - \frac{\epsi\deltax}{2\sol}\absfluxMat \right]\Diffp\micvec^{n+1}_{i+1/2}\deltax.
        \end{aligned}
    \end{equation}
    If we define $\bv = (0,\mic_{1},\ldots,\mic_{N})^{\top}\in\R^{N+1} $ and $\hat{\bv} = \FTmat^{\top}\bv\in\R^{N+1} $, then we have from \Cref{lemma:PnPreservation}
    \begin{equation*}\label{eq:Thm1Eq21}
        \Diffp\micvec^{n+1,\top}_{i+1/2}\fluxMat^{2}\Diffp\micvec^{n+1}_{i+1/2} = \Diffp\hat{\bv}^{n+1,\top}_{i+1/2}\quadptMat^{2}\Diffp\hat{\bv}^{n+1}_{i+1/2},
    \end{equation*}
    \begin{equation*}\label{eq:Thm1Eq22}
        \Diffp\micvec^{n+1,\top}_{i+1/2}\absfluxMat\Diffp\micvec^{n+1}_{i+1/2} = \Diffp\hat{\bv}^{n+1,\top}_{i+1/2}\absquadMat\Diffp\hat{\bv}^{n+1}_{i+1/2},
    \end{equation*}
    and
    \begin{equation*}\label{eq:Thm1Eq23}
        \Diffp\micvec^{n+1,\top}_{i+1/2}\ba\ba^{\top}\Diffp\micvec^{n+1}_{i+1/2} = \Diffp\hat{\bv}^{n+1,\top}_{i+1/2}\FTmat^{\top}\Fba\Fba^{\top}\FTmat\Diffp\hat{\bv}^{n+1}_{i+1/2}.
    \end{equation*}
    As given in \autocite[Thoerem~3.2]{einkemmer2022asymptoticpreserving}, since $\FTmat^{\top}\Fba = \frac{1}{\sqrt{3}}\FTmat^{\top}\textbf{e}_{2} = \sqrt{\frac{\weight_{k}}{3}}\pk{1}(\quadpt_{k}) =  \sqrt{\frac{\weight_{k}}{2}}\quadpt_{k} $, we have
    \begin{align*}
        \Diffp\hat{\bv}^{n+1,\top}_{i+1/2}\FTmat^{\top}\Fba\Fba^{\top}\FTmat\Diffp\hat{\bv}^{n+1}_{i+1/2} &= \left( \sum_{k=1}^{N+1}\Diffp\hat{v}^{n+1}_{i+1/2,k}\sqrt{\frac{\weight_{k}}{2}}\quadpt_{k} \right)^{2}\\
        &= \frac{1}{2}\sum_{k,\ell}^{N+1}\Diffp\hat{v}^{n+1}_{i+1/2,k}\sqrt{\weight_{k}}\quadpt_{k}\Diffp\hat{v}^{n+1}_{\ell,i+1/2}\sqrt{\weight_{\ell}}\quadpt_{\ell}\\
        & \overset{\text{Young}}{\leq} \frac{1}{4}\sum_{k,\ell}^{N+1}\left(\Diffp\hat{v}^{n+1}_{i+1/2,k}\right)^{2}\weight_{k}\quadpt_{k}^{2}\\
        &\qquad\qquad +  \frac{1}{4}\sum_{k,\ell}^{N+1}\left(\Diffp\hat{v}^{n+1}_{\ell,i+1/2}\right)^{2}\weight_{\ell}\quadpt_{\ell}^{2}\\
        &= \frac{N+1}{2}\sum_{k}^{N+1}\left(\Diffp\hat{v}^{n+1}_{i+1/2,k}\right)^{2}\weight_{k}\quadpt_{k}^{2}. 
    \end{align*}
    Since we can write $\norm{\micmat^{n+1}}^{2} = \norm{\hat{\bv}^{n+1}}^{2} = \sum_{k,i}(\hat{v}^{n+1}_{i+1/2,k})^{2}\deltax $ we have that \eqref{eq:Thm1Eq20} becomes
    \begin{equation*}\label{eq:Thm1Eq24}
    \begin{aligned}
        \frac{1}{2\deltat}\left( \energy^{n+1} - \energy^{n} \right) &\leq -\frac{\absorpcoefflb}{2\sol}\sum_{k,i}(\hat{v}^{n+1}_{i+1/2,k})^{2}\deltax\\
        & \qquad\qquad +  \frac{1}{2}\sum_{k,i}(\Diffp\hat{v}^{n+1}_{i+1/2,k})^{2} \left[ \deltat\left( \beta_{N}\quadpt_{k}^{2} + \frac{N+1}{4}\weight_{k}\quadpt_{k}^{2} \right) - \frac{\epsi\deltax}{2\sol}\abs{\quadpt_{k}} \right]\deltax.
    \end{aligned}
    \end{equation*}
    \noindent Using \Cref{lemma:boundforDp} we get $\sum_{i}\left( \Diffp\hat{\bv}^{n+1}_{i+1/2} \right)^{2} \leq \frac{4}{\deltax^{2}}\sum_{i}(\hat{\bv}^{n+1}_{i+1/2})^{2}$, thus 
    \begin{equation*}\label{eq:Thm1Eq25}
         \frac{1}{2}\left( \energy^{n+1} - \energy^{n} \right) \leq \frac{1}{2}\frac{\deltat}{\deltax}\sum_{k,i}(\hat{v}^{n+1}_{i+1/2,k})^{2} \left[ \frac{4\deltat}{\deltax^{2}}\left( \beta_{N}\quadpt_{k}^{2} + \frac{1}{4}\beta_{N}\quadpt_{k}^{2} \right) - \frac{4\epsi}{2\sol\deltax}\abs{\quadpt_{k}}  - \frac{\absorpcoefflb}{\sol}\right].
    \end{equation*}
    Hence for ensuring stability we must have for all $k$, where $\quadpt_{k} \neq 0 $,
    \begin{equation*}\label{eq:Thm1Eq26}
         \deltat \leq \frac{1}{5\sol\beta_{N}}\left( \frac{\absorpcoefflb\deltax^{2}}{\quadpt_{k}^{2}} + \frac{2\epsi\deltax}{\abs{\quadpt_{k}}} \right).
    \end{equation*}
    We note that $\beta_{N}$ remains bounded for all $N$.
\end{proof}

\section{Proof of Lemma 1}\label{appendix:ProofAPraBUGAux1}
\begin{proof}
    As the SVD of $\hSmat^{\text{rem}} = \textbf{U}\BSigma\textbf{W}^{\top} $ \eqref{eq:raBUGeq2} where $\textbf{U}$ and $\textbf{W}$ are orthogonal, we have
    \begin{equation*}\label{eq:APraBUGAuxeq1}
        \widehat{\textbf{U}}^{\top}\left(\hSmat^{\text{rem}}\right)^{-\top} = (\widehat{\textbf{U}}^{\top}\textbf{U})\BSigma^{-\top}\textbf{W}^{\top} = (\Smat^{\text{rem}})^{-\top}\widehat{\textbf{W}}^{\top}.
    \end{equation*}
    Consider
    \begin{align*}
        \hXmat^{n+1}\hSmat^{n+1}\begin{bmatrix}
           \textbf{I}_{m} & \textbf{0}\\
           \textbf{0} & \widehat{\textbf{W}}
       \end{bmatrix} &= \widehat{\Kmat}\begin{bmatrix}
           \textbf{I}_{m} & \textbf{0}\\
           \textbf{0} & \widehat{\textbf{W}}
       \end{bmatrix}\\
            &= \begin{bmatrix}
                \Xmat^{\text{ap}}\Smat^{\text{ap}} & \hXmat^{\text{rem}}\hSmat^{\text{rem}}
                    \end{bmatrix}\begin{bmatrix}
           \textbf{I}_{m} & \textbf{0}\\
           \textbf{0} & \widehat{\textbf{W}}
       \end{bmatrix}\\
            &= \begin{bmatrix}
                \Xmat^{\text{ap}} & \hXmat^{\text{rem}}
        \end{bmatrix}
       \begin{bmatrix}
                \Smat^{\text{ap}} & \textbf{0}\\
                \textbf{0} & \hSmat^{\text{rem}}\widehat{\textbf{W}}
            \end{bmatrix}\\
            &= \begin{bmatrix}
                \Xmat^{\text{ap}} & \hXmat^{\text{rem}}
        \end{bmatrix}\begin{bmatrix}
            \textbf{I}_{m} & \textbf{0}\\
            \textbf{0} & \widehat{\textbf{U}}
        \end{bmatrix}\begin{bmatrix}
                \Smat^{\text{ap}} & \textbf{0}\\
                \textbf{0} & \Smat^{\text{rem}}
            \end{bmatrix}
    \end{align*}
    where we use that $\hSmat^{\text{rem}}\widehat{\textbf{W}} = \widehat{\textbf{U}}\Smat^{\text{rem}} $. We have from \eqref{eq:raBUGeq6} that the updated spatial basis matrix has the form $\Xmat^{n+1}\textbf{R}_{2} = \begin{bmatrix}
        \Xmat^{\text{ap}} & \Xmat^{\text{rem}}
    \end{bmatrix} = \begin{bmatrix}
        \Xmat^{\text{ap}} & \hXmat^{\text{rem}}\widehat{\textbf{U}}
    \end{bmatrix} $ and thus
    \begin{equation*}
    \hXmat^{n+1}\hSmat^{n+1}\begin{bmatrix}
           \textbf{I}_{m} & \textbf{0}\\
           \textbf{0} & \widehat{\textbf{W}}
       \end{bmatrix} =\Xmat^{n+1}\textbf{R}_{2}\begin{bmatrix}
           \Smat^{\text{ap}} & \textbf{0}\\
           \textbf{0} & \Smat^{\text{rem}}
       \end{bmatrix}.
    \end{equation*}
    Hence we get the following relation
    \begin{equation*}
    \begin{aligned}
        \begin{bmatrix}
           \textbf{I}_{m} & \textbf{0}\\
           \textbf{0} & \widehat{\textbf{W}}^{\top}
       \end{bmatrix}\hSmat^{n+1,\top}\hSmat^{n+1}\begin{bmatrix}
           \textbf{I}_{m} & \textbf{0}\\
           \textbf{0} & \widehat{\textbf{W}}\end{bmatrix}
           &= \begin{bmatrix}
               \textbf{I}_{m} & \textbf{0}\\
           \textbf{0} & \widehat{\textbf{W}}^{\top}
       \end{bmatrix}\hSmat^{n+1,\top}\hXmat^{n+1,\top}\hXmat^{n+1}\hSmat^{n+1}\begin{bmatrix}
           \textbf{I}_{m} & \textbf{0}\\
           \textbf{0} & \widehat{\textbf{W}} 
       \end{bmatrix}\\
       &= \begin{bmatrix}
           \Smat^{\text{ap},\top} & \textbf{0}\\
           \textbf{0} & \Smat^{\text{rem},\top}
       \end{bmatrix}\textbf{R}_{2}^{\top}\Xmat^{n+1,\top}\Xmat^{n+1}\textbf{R}_{2}\begin{bmatrix}
           \Smat^{\text{ap}} & \textbf{0}\\
           \textbf{0} & \Smat^{\text{rem}}
       \end{bmatrix}\\
       &= \begin{bmatrix}
           \Smat^{\text{ap},\top} & \textbf{0}\\
           \textbf{0} & \Smat^{\text{rem},\top}
       \end{bmatrix}\textbf{R}_{2}^{\top}\textbf{R}_{2}\begin{bmatrix}
           \Smat^{\text{ap}} & \textbf{0}\\
           \textbf{0} & \Smat^{\text{rem}}
       \end{bmatrix}.
    \end{aligned}
    \end{equation*}
    Putting it all together we have 
    \begin{align*}
        \textbf{R}_{2}^{-\top}\begin{bmatrix}
           (\Smat^{\text{ap}})^{-\top} & \textbf{0}\\
           \textbf{0} & \widehat{\textbf{U}}^{\top}(\widehat{\Smat}^{\text{rem}})^{-\top}
       \end{bmatrix}&\hSmat^{n+1,\top}\hSmat^{n+1}\begin{bmatrix}
           \textbf{I}_{m} & \textbf{0}\\
           \textbf{0} & \widehat{\textbf{W}}
       \end{bmatrix} \\
       &= \textbf{R}_{2}^{-\top}\begin{bmatrix}
           (\Smat^{\text{ap}})^{-\top} & \textbf{0}\\
           \textbf{0} & (\Smat^{\text{rem}})^{-\top}\widehat{\textbf{W}}^{\top}
       \end{bmatrix}\hSmat^{n+1,\top}\hSmat^{n+1}\begin{bmatrix}
           \textbf{I}_{m} & \textbf{0}\\
           \textbf{0} & \widehat{\textbf{W}}
       \end{bmatrix}\\
       &= \textbf{R}_{2}^{-\top}\begin{bmatrix}
           \Smat^{\text{ap}} & \textbf{0}\\
           \textbf{0} & \Smat^{\text{rem}}
       \end{bmatrix}^{-\top}\begin{bmatrix}
           \textbf{I}_{m} & \textbf{0}\\
           \textbf{0} & \widehat{\textbf{W}}^{\top}
       \end{bmatrix}\hSmat^{n+1,\top}\hSmat^{n+1}\begin{bmatrix}
           \textbf{I}_{m} & \textbf{0}\\
           \textbf{0} & \widehat{\textbf{W}}
       \end{bmatrix}\\
       & = \textbf{R}_{2}^{-\top}\begin{bmatrix}
           \Smat^{\text{ap}} & \textbf{0}\\
           \textbf{0} & \Smat^{\text{rem}}
       \end{bmatrix}^{-\top}\begin{bmatrix}
           \Smat^{\text{ap},\top} & \textbf{0}\\
           \textbf{0} & \Smat^{\text{rem},\top}
       \end{bmatrix}\textbf{R}_{2}^{\top}\textbf{R}_{2}\begin{bmatrix}
           \Smat^{\text{ap}} & \textbf{0}\\
           \textbf{0} & \Smat^{\text{rem}}
       \end{bmatrix}\\
       &= \textbf{R}_{2}\begin{bmatrix}
           \Smat^{\text{ap}} & \textbf{0}\\
           \textbf{0} & \Smat^{\text{rem}}
       \end{bmatrix}\\
       & = \Smat^{n+1}.
    \end{align*}
\end{proof}

\end{appendices}

\end{document}